\documentclass[12pt]{amsart}

\usepackage{amsmath,amssymb,amscd,amsthm,graphicx,enumerate}
\usepackage[usenames,dvipsnames]{xcolor}

\usepackage{hyperref}
\hypersetup{breaklinks=true}

\usepackage{enumitem} 
\newlist{condenum}{enumerate}{1} 
\setlist[condenum]{label=\bfseries Condition \arabic*.,  ref=\arabic*, wide}

\numberwithin{equation}{section}
\theoremstyle{plain}



\usepackage[utf8]{inputenc}
\usepackage{amssymb}
\usepackage{amsthm}
\usepackage{mathrsfs}
\usepackage{amsfonts}
\usepackage{amsmath}
\usepackage{mathtools}
\usepackage{makecell}
\usepackage{xcolor}
\usepackage[margin=1.25in]{geometry}

\usepackage{lipsum}
\makeatletter
\def\ps@pprintTitle{%
 \let\@oddhead\@empty
 \let\@evenhead\@empty
 \def\@oddfoot{}%
 \let\@evenfoot\@oddfoot}
\makeatother

\newcommand{\R}{R}

\newcommand{\LL}{\mathcal{L}}

\newcommand{\Rm}{\textnormal{Rm}}
\newcommand{\scalar}{R}
\newcommand{\Ric}{\textnormal{Ric}}

\newcommand{\pt}{\partial_t}

\newcommand{\M}{\mathcal{M}}
\newcommand{\ot}{\overline{t}}
\newcommand{\ox}{\overline{x}}

\numberwithin{equation}{section}

\newtheorem{theorem}{Theorem}[section]
\newtheorem{lem}[theorem]{Lemma}
\newtheorem{remark}[theorem]{Remark}
\newtheorem{prop}[theorem]{Proposition}

\newtheorem{cor}[theorem]{Corollary} 

\theoremstyle{definition}
\newtheorem{defn}[theorem]{Definition}
\newtheorem*{theorem*}{Theorem}

\usepackage{xpatch}
\makeatletter
\xpatchcmd{\tableofcontents}{\contentsname \@mkboth}{\small\contentsname \@mkboth}{}{}
\xpatchcmd{\listoffigures}{\chapter *{\listfigurename }}{\chapter *{\small\listfigurename }}{}{}
\makeatother

\begin{document}

\begin{abstract}
We extend the concept of singular Ricci flow by Kleiner and Lott from 3d compact manifolds to 3d complete manifolds with possibly unbounded curvature.
As an application of the generalized singular Ricci flow, we show that for any 3d complete Riemannian manifold with non-negative Ricci curvature, there exists a smooth Ricci flow starting from it. This partially confirms a conjecture by Topping.
\end{abstract}

\title[Ricci flows via singular Ricci flows]{Producing 3d Ricci flows with non-negative Ricci curvature via singular Ricci flows}

\author[Yi Lai]{Yi Lai}
\email{yilai@berkeley.math.edu}
\address[]{Department of Mathematics, University of California, Berkeley, CA 94720, USA}

\maketitle


\setcounter{tocdepth}{1}
%

\begin{section}{Introduction and main results}\label{s: intro}
The Ricci flow was introduced by Hamilton in \cite{hamilton1}.
Since its singularities can occur along proper subsets of the manifold, in order to continue the flow, Hamilton introduced Ricci flow with surgery in \cite{hamilton2}.
Based on the earlier work of Hamilton, Perelman constructed Ricci flow with surgery on any compact 3 dimensional Riemannian manifold, and hence proved the Geometrization and Poincar\'e Conjectures in \cite{Pel1}\cite{Pel2}\cite{Pel3}. 
After this, the Ricci flow with surgery was also constructed for complete non-compact manifolds with bounded geometry by Bessi\`ere, Besson, and Maillot in \cite{Besson}.

Perelman's Ricci flow with surgery is a sequence of ordinary compact Ricci flows such that the final time-slice of each flow is isometric, modulo surgery, to the initial time-slice of the next one. 
The surgery process is regulated by several parameters, one of them being the surgery scale $\delta>0$.
Perelman showed that $\delta>0$ can be chosen arbitrarily small, and as such he conjectured that the Ricci flow with surgery should converge to a canonical Ricci flow through singularities.

Recently, such a canonical flow, named the singular Ricci flow, was constructed by Kleiner and Lott in \cite{KL1}, and shown to be unique by Bamler and Kleiner in \cite{BK}. 
The singular Ricci flow is a 0-complete Ricci flow spacetime starting from a compact manifold, which satisfies the Hamilton-Ivey pinching, and the canonical neighborhood assumption at scales less than a time-dependent constant.

In this paper, we introduce a new weak solution of Ricci flow that we call a generalized singular Ricci flow, which allows the initial manifold to be a complete manifold with possibly unbounded curvature. We have the following existence theorem.
\begin{theorem} \label{t: construction1}
For any 3d complete Riemannian manifold $(M,g)$, there is a generalized singular Ricci flow starting from $(M,g)$.
\end{theorem}

The generalized singular Ricci flow has many properties similar to those of a singular Ricci flow.
In particular, it satisfies the canonical neighborhood assumption in a distance-dependent way.
Its precise definition of a generalized singular Ricci flow will be given in Definition \ref{d: generalized SRF}.

The existence is obtained from a compactness result for singular Ricci flows, which states that a sequence of singular Ricci flows converges to a generalized singular Ricci flow starting from a complete manifold $(M,g)$, if the sequence of their initial manifolds converges to $(M,g)$:

\begin{theorem}\label{t: convergence}
Let $\M_i$ be a sequence of singular Ricci flows starting from compact manifolds $M_i$, $x_i\in M_i$. Suppose $(M_i,x_i)$ converges smoothly to a 3d complete manifold $(M,x_0)$ as $i\rightarrow\infty$.
Then by passing to a subsequence, $(\M_i,x_i)$ converges smoothly to a generalized singular Ricci flow starting from $M$.
\end{theorem}

Theorem \ref{t: convergence} can be compared to the convergence of a sequence of singular Ricci flows, when the sequence of their initial time-slices converges to a compact manifold, see \cite[Prop 5.39]{KL1}.
In that result, the initial time-slices have uniformly bounded curvature and injectivity radius.
So the local geometry in each singular Ricci flow is uniformly controlled by its scalar curvature, which guarantees their convergence to a singular Ricci flow.
However, in our case, the initial time-slices may not have uniformly bounded geometry.
Instead, we will show that the scalar curvature controls the local geometry in a uniform distance-dependent way, which ensures the convergence in Theorem \ref{t: convergence}.

Before stating our next main result, we recall some results of the existence theory of Ricci flow with non-compact initial conditions. Much less is known about it compared to the compact case.
In \cite{Shi1987}, Shi showed that if $(M,g)$ is an n-dimensional complete Riemannian manifold with bounded curvature, then there exists a complete Ricci flow with bounded curvature for a short time. 
Since then, many efforts have been made to relax the bounded-curvature assumption, in order to obtain a Ricci flow starting from a complete non-compact manifold.

In \cite{CW}, Cebazas-Rivas and Wilking proved that a smooth complete Ricci flow exists on a complete n-dimensional manifold with non-negative complex sectional curvature, which in dimension 3 is the same as non-negative sectional curvature. 
Recently, Simon and Topping \cite{mollification} showed that a complete Ricci flow exists on a complete 3d Riemannian manifold, if its
Ricci curvature has a negative lower bound and the volume is globally non-collapsed (i.e. there is a uniform positive lower bound on the volume of every unit ball).
In \cite{BCRW}, Bamler, Cebazas-Rivas and Wilking proved that the same thing holds in dimension $n$, assuming a certain curvature is bounded below, and the volume is non-collapsed.
In \cite{Lai}, by a combination of methods in \cite{mollification} and \cite{BCRW}, the author generalized both works. 

The volume non-collapsing assumption is necessary in \cite{mollification}\cite{BCRW}\cite{Lai}, where the curvature is allowed to be negative somewhere.
For example, see \cite[Example 2.4]{Top}, for any arbitrarily small $\epsilon>0$ we can construct a complete 3-manifold with $\Ric\ge-\epsilon$, by connecting countably many three-spheres by necks that become longer and thinner.
So the necks would pinch in times converging to zero, and hence a Ricci flow cannot exist for any short time.

This leads to an open question: whether a smooth complete Ricci flow exists for a 3 dimensional complete manifold with non-negative Ricci curvature, see e.g. \cite[Conjecture 7.1]{Top}. 
Our next main result gives a partial affirmative answer to this question:

\begin{theorem}\label{t: existence}
Let $(M,g)$ be a 3d complete Riemannian manifold with $\Ric\ge 0$. There exist $T>0$ and a smooth Ricci flow $(M,g(t))$ on $[0,T)$, with $g(0)=g$ and $\Ric(g(t))\ge0$. 
Moreover, if $T<\infty$, then $\limsup_{t\nearrow T}|\Rm|(x,t)=\infty$ for all $x\in M$.
\end{theorem}
We remark that the completeness of this flow is not guaranteed in this paper. Instead, we show that it can be embedded in a smooth Ricci flow spacetime with complete time-slices. Also, it is possible for the maximal existence time to be finite, such as the standard solution and the cylindrical solutions.

A common strategy to produce a smooth Ricci flow with a complete non-compact initial condition is by a limiting argument: first construct a sequence of local Ricci flows starting from larger and larger balls in $M$, and then try to get a uniform lower bound on the existence times, as well as an upper bound on the curvature norms. 
Then by Hamilton's compactness theorem for Ricci flow, we obtain a smooth limit Ricci flow starting from $M$. 
This argument typically works when there is a non-collapsing assumption \cite{mollification,BCRW,Lai}, or the curvature condition is relatively strong \cite{CW}.

However, it seems hard to apply the limiting argument to prove Theorem \ref{t: existence}, for $\Ric\ge 0$ is a relatively weak curvature assumption, and there is no uniform lower bound on the volume of all unit balls on certain manifolds, as shown by examples in \cite{collapsed_balls}.
In this paper, we produce a smooth Ricci flow by showing that a generalized singular Ricci flow starting from a complete manifold with $\Ric\ge0$ is actually smooth. The existence of the generalized singular Ricci flow is guaranteed by Theorem \ref{t: construction1}.

The paper is organized as follows. In Section \ref{s: preliminary}, we review some basic concepts in Perelman's Ricci flow with surgery and the singular Ricci flow. 
Section \ref{s: Preparatory results} is for some technical lemmas. 
In Section \ref{s: CNA and NC}, we generalize Perelman's no local collapsing and canonical neighborhood theorem to singular Ricci flows.
It provides a distance-dependent lower bound on the non-collapsing scale and canonical neighborhood scale,
assuming the geometry is bounded in a parabolic neighborhood of the base point.

In Section \ref{s: heat kernel}, we define a heat kernel $H$ for a singular Ricci flow $\M$. For any point $(x_0,t_0)\in\M$, $H(x_0,t_0;\cdot,\cdot)$ is a positive solution to the conjugate heat equation on $\M$, which is a $\delta$-function around $(x_0,t_0)$.
Moreover, we show that the heat kernel decays polynomially fast to zero as the curvature blows up. 
This implies that the overall amount of heat is a constant, i.e. the integral of $H(x_0,t_0;\cdot,t)$ at all times $t$ prior to $t_0$ is equal to one. Note that for the ordinary heat kernel of a compact smooth Ricci flow, the constancy of the integral is easily shown by a computation using integration by part. 
Moreover, the polynomial decay near the high curvature region also implies that Perelman's Harnack inequality holds for singular Ricci flow.
With these properties of the heat kernel, we are able to generalize Perelman's pseudolocality theorem to singular Ricci flow in Section \ref{s: pseudolocality}. 

In Section \ref{s: construction of non-compact SRF}, we define the generalized singular Ricci flow, and prove Theorem \ref{t: construction1} and \ref{t: convergence}. 
The proofs depend on a compactness theorem, which states that assuming there is a uniform distance-dependent canonical neighborhood assumption in a sequence of pointed singular Ricci flows, then a subsequence converges smoothly to a semi-generalized singular Ricci flow, which satisfies most properties of the generalized singular Ricci flow. 
The compactness theorem can be proved by first taking a Gromov-Hausdorff limit, and showing that the convergence is smooth on the subset of points which are limits of points with bounded curvature. This induces a semi-generalized singular Ricci flow.
To prove Theorem \ref{t: convergence}, by applying the compactness theorem, we get a semi-generalized singular Ricci flow in which the base point $x_0$ survives until its curvature goes unbounded.
Then a generalized singular Ricci flow is obtained
by varying the base points and gluing up all the corresponding semi-generalized singular Ricci flows.



In Section \ref{s: it is actually smooth} we prove Theorem \ref{t: existence}. First, by a maximum principle argument, we show in Lemma \ref{l: Ricci} that the generalized singular Ricci flow $\M$ preserves the non-negativity of Ricci curvature. 
Suppose the curvature blows up in a ball of finite radius. Then by the canonical neighborhood assumption, we can show that the curvature blow-up is due to the asymptotic formation of a cone-like
point.
Doing a further rescaling at this cone-like point, we obtain a Ricci flow solution whose final time-slice is a part of a non-flat metric cone, which is impossible.
So $\M$ is in fact a non-singular Ricci flow spacetime with complete time-slices.
Restricting the spacetime on $M$, we obtain a smooth Ricci flow.

I thank my advisor Richard Bamler for helpful discussions and many comments. I also thank John Lott and Guoqing Wu for comments, and Paula Burkhardt and Angxiu Ni for correcting my English.
\end{section}

\begin{section}{Preliminary}\label{s: preliminary}
In this section, we collect some notions and concepts of Perelman's Ricci flow with surgery \cite{KL} and singular Ricci flow \cite{KL1}, \cite{BK2} that will be frequently used later. 

\begin{subsection}{Ricci flow spacetime}

\begin{defn}[Ricci flow spacetime]
A Ricci flow spacetime is a tuple $(\M,\mathfrak{t},\partial_{\mathfrak{t}},g)$ (sometimes abbreviate as $\M$ or $(\M,g(t))$) with the following properties:
\begin{enumerate}
    \item $\M$ is a smooth $4$-manifold with (smooth) boundary $\partial\M$.
    \item $\mathfrak{t}: \M\rightarrow[0,T)$, where $T$ can be infinity, is a smooth function without critical points. For any $t\ge 0$ we denote by $\M_t:=\mathfrak{t}^{-1}(t)\subset \M$ the time-t-slice of $\M$.
    \item $\partial\M=\M_0$, i.e. the boundary of $\M$ is equal to the initial time-slice. 
    \item $\partial_{\mathfrak{t}}$ is a smooth vector field (the time vector field), which satisfies $\partial_{\mathfrak{t}}\mathfrak{t}\equiv 1$.
    \item $g$ is a smooth inner product on the spacial subbundle ker$(d\mathfrak{t})\subset T\M$. For any $t\ge 0$ we denote by $g(t)$ the restriction of $g$ to the time-t-slice $\M_t$, which is a Riemannian metric.
    \item $g$ satisfies the Ricci flow equation: $\LL_{\partial_t}g=-2\Ric(g(t))$. 
\end{enumerate}
We call the Riemannian metric $G:=dt^2+g$ the spacetime metric.

\end{defn}

\begin{defn}[Points in a Ricci flow spacetime]
Let $(\M,\mathfrak{t},\partial_{\mathfrak{t}},g)$ be a Ricci flow spacetime and $x\in\M$ be a point. Set $t:=\mathfrak{t}(x)$. We sometimes write $x$ as $(x,t)$ to indicate its time, when there is no ambiguity. Consider the maximal trajectory $\gamma_x: I\rightarrow \M$, $I\subset[0,\infty)$ of the time-vector field $\partial_{\mathfrak{t}}$ such that $\gamma_x(t)=x$. Note that $\mathfrak{t}(\gamma_x(t'))=t'$ for all $t'\in I$. For any $t'\in I$ we say that $x$ survives until time $t'$ and we write \begin{equation}
    x(t'):=\gamma_x(t').
\end{equation}
Similarly, for a subset $X\subset \M_t$, we say that $X$ survives until time $t'$ if this is true for every $x\in X$, and we write $X(t')=\{x(t'): x\in X\}$.
\end{defn}

\begin{defn}[Distance and metric balls]
Let $(\M,\mathfrak{t},\partial_{\mathfrak{t}},g)$ be a Ricci flow spacetime. For any two points $x,y \in \M_t$ we denote by $d_t(x,y)$, or simply $d(x,y)$ the distance between $x,y$ within $(\M_t,g(t))$.

For any $x\in\M_t$ and $r\ge 0$ we denote by $B_t(x,r)\subset \M_t$ the $r$-ball around $x$ with respect to the Riemannian metric $g(t)$.

\end{defn}

\begin{defn}[Parabolic neighborhood]
Let $(\M,\mathfrak{t},\partial_{\mathfrak{t}},g)$ be a Ricci flow spacetime. For any $y\in \M$ let $I_y\subset[0,\infty)$ be the set of all times until which $y$ survives. Let $x\in\M$ and $a\ge 0, b\in\mathbb{R}$. Set $t=\mathfrak{t}(x)$. Then we define the \textit{parabolic neighborhood} $P(x,a,b)\subset\M$ to be:
\begin{equation}
    P(x,a,b):=\bigcup_{y\in B_t(x,a)}\bigcup_{t'\in[t,t+b]\cap I_y} y(t').
\end{equation}
If $b<0$, we replace $[t,t+b]$ by $[t+b,t]$.
We call $P(x,a,b)$ \textit{unscathed} if $B(x,a)$ is relatively compact in $\M_t$ and if $B(x,a)$ survives until $t+b$. 
\end{defn}

\begin{defn}[Admissible curve and accessibility]
Let $\M$ be a Ricci flow spacetime, we say $\gamma:[c,d]\rightarrow\M$ is an admissible curve if $\gamma(t)\in \M_t$ for all $t\in[c,d]$. 
We say a point $x\in\M$ with $\mathfrak{t}(x)<\mathfrak{t}(x_0)$ is accessible from $x_0$ if there is an admissible curve running from $(x,t)$ to $(x_0,t_0)$.

Let $x_0\in\M_t$, $t>0$. We denote by $\M(x_0)$ the subset consisting of all points in $\M$ that are accessible to $x_0$.
\end{defn}

\begin{defn}[Hamilton-Ivey pinching]
Let $M$ be a 3 dimensional Riemannian manifold and $\varphi>0$. We say that the curvature at $x\in M$ is $\varphi$-positive if there is an $X>0$ with $\Rm(x)\ge-X$ such that
\begin{equation}\label{e: Hamilton-Ivey for manifold}
    R(x)\ge-\frac{3}{\varphi^{-1}}\quad \textnormal{and} \quad R(x)\ge X(\log X+\log(\varphi^{-1})-3).
\end{equation}

Let $(M,g(t)), t\in[0,T]$ be a 3 dimensional compact Ricci flow and $\varphi\in\mathbb{R}_+\cup\infty$. We say that the curvature at $(x,t)\in M\times[0,T]$ is $\varphi$-positive if there is an $X>0$ with $\Rm(x,t)\ge-X$ such that
\begin{equation}\label{e: Hamilton-Ivey}
    R(x,t)\ge-\frac{3}{\varphi^{-1}+t}\quad \textnormal{and} \quad R(x,t)\ge X(\log X+\log(\varphi^{-1}+t)-3).
\end{equation}
\end{defn}
The Hamilton-Ivey pinching theorem \cite[Appendix B]{KL} says that if the curvature is $\varphi$-positive at time $0$, then the curvature is $\varphi$-positive at all positive times.
Moreover, the same conclusion also holds for singular Ricci flow \cite[Theorem 1.3]{KL1}.

\begin{defn}[$\kappa$-non-collapsed]
Let $(M,g)$ be a 3 dimensional Riemannian manifold, $x\in M$ and $\kappa,r_0>0$. We say $M$ is $\kappa$-non-collapsed at $x$ at scales less than $r_0$, if $r^{-3}vol(B_g(x,r))\ge\kappa>0$, for all $0<r\le r_0$ such that $|\Rm|\le r^{-2}$ holds on $B_g(x,r)$.
\end{defn}

\begin{defn}[Normalized manifold]
Let $(M,g)$ be a $3$-dimensional compact orientable connected Riemannian manifold that
\begin{enumerate}
    \item is not a higher spherical space form,
    \item has scalar curvature $R<1$ everywhere,
    \item is 1-non-collapsed at scales less than $1$ and
    \item satisfies the 1-positive curvature condition at time $0$.
\end{enumerate}
Then we say $(M,g)$ has normalized geometry. For a Ricci flow spacetime, we say it has normalized initial condition if it starts from a manifold $(M,g)$ with normalized geometry.
\end{defn}

\begin{defn}[Curvature scale]\label{d: curvature scale}
Let $(M,g)$ be a 3 dimensional Riemannian manifold and $x\in M$ a point. Let the curvature scale at $x$ be
\begin{equation}
    \rho(x)=R_+^{-1/2},
\end{equation}
where $R_+=\max\{R,0\}$, and we use the convention $0^{-1/2}=\infty$.
\end{defn}

\begin{defn}(0-complete)\label{d: complete}
We say a Ricci flow spacetime $\M$ is $0$-complete if for any smooth curve $\gamma:[0,s_0)\rightarrow \M$  that satisfies $\inf_{[0,s_0)} \rho(\gamma(s))>0$ and one of the following
\begin{enumerate}
\item $\gamma([0,s_0))$ is contained in a time-slice $\M_t$, and has finite length with respect to the horizontal metric in $\M_t$, or
\item $\gamma$ is the integral curve of $-\partial_t$, or $\partial_t$.
\end{enumerate}
Then $\lim_{s\rightarrow s_0}\gamma(s)$ exists.

Also, we say a spacetime is backward (resp. forward) 0-complete if in case (2), $\gamma$ is only the integral curve of $-\partial_t$ (resp. $\partial_t$).

We say a manifold is 0-complete if it satisfies condition (1).

\end{defn}


\end{subsection}

\begin{subsection}{Singular Ricci flow}

\begin{defn}[$\kappa$-solution]
An ancient Ricci flow $(M,g(t)_{t\in(\infty,0]})$ on a 3 dimensional manifold $M$ is called a $\kappa$-solution if it satisfies the following:
\begin{enumerate}
    \item $(M,g(t))$ is complete for all $t\in(-\infty,0]$,
    \item $|\Rm|$ is bounded on $M\times(-\infty,0]$,
    \item $\textnormal{sec}\ge 0$ on $M\times(-\infty,0]$,
    \item $(M,g(t))$ is $\kappa$-non-collapsed at all scales for all $t\in(-\infty,0]$.
\end{enumerate}

\end{defn}

\begin{defn}[Geometric closeness]
We say that a pointed Riemannian manifold $(M,g,x)$ is $\epsilon$-close to another pointed Riemannian manifold $(\overline{M},\overline{g},\ox)$ at scale $\lambda>0$ if there is a diffeomorphism onto its image
\begin{equation}
    \psi: B^{\overline{M}}(\ox,\epsilon^{-1})\rightarrow M
\end{equation}
such that $\psi(\ox)=x$ and
\begin{equation}
    \|\lambda^{-2}\psi^{*}g-\overline{g}\|_{C^{[\epsilon^{-1}]}(B^{\overline{M}}(\ox,\epsilon^{-1}))}<\epsilon.
\end{equation}
Here the $C^{[\epsilon^{-1}]}$-norm of a tensor $h$ is defined to be the sum of the $C^0$-norms of the tensors $h,\nabla^{\overline{g}}h,\nabla^{\overline{g},2}h,...,\nabla^{\overline{g},[\epsilon^{-1}]}h$ with respect to the metric $\overline{g}$.

Similarly, we say a pointed Ricci flow $(M,g(t),(x,0))$ is $\epsilon$-close to a pointed Ricci flow $(\overline{M},\overline{g}(t),(\ox,0))$ on $[a,b]$ ($a\le0\le b$) at scale $\lambda>0$ if $g(t)$ is defined on $[\lambda^2a,\lambda^2b]$, and
there is a diffeomorphism onto its image
\begin{equation}
    \psi: B^{\overline{M}}_{\overline{g}(0)}(\ox,\epsilon^{-1})\rightarrow M
\end{equation}
such that $\psi(\ox)=x$ and
\begin{equation}
    \|\lambda^{-2}\psi^{*}g(\lambda^2t)-\overline{g}(t)\|_{C^{[\epsilon^{-1}]}(B_{\overline{g}(0)}^{\overline{M}}(\ox,\epsilon^{-1}))}<\epsilon
\end{equation}
for all $t\in[a,b]$, where the norm is measured with respect to the metric $\overline{g}(t)$. 
In particular, when $a=-\epsilon^{-1}$ and $b=0$, we simply say $(M,g(t),(x,0))$ is $\epsilon$-close to $(\overline{M},\overline{g}(t),(\ox,0))$.
\end{defn}

\begin{defn}[$\delta$-neck and strong $\delta$-neck]
Let $(M,g)$ be a 3 dimensional Riemannian manifold and $\delta>0$. Suppose $U\subset M$ is an open subset, $x\in U$. We say $U$ is a $\delta$-neck centered at $x$, if $(U,g)$ is $\delta$-close to the standard cylindrical metric on $(-\delta^{-1},\delta^{-1})\times S^2$ at scale $\rho(x)$. 

Let $(M,g(t))$ be a Ricci flow. Suppose $U\subset M$ is an open subset and $x$ is a point in $U$. We say that $U$ is a strong $\delta$-neck on $[-c,0]$ centered at $x$ for some $c>0$, if $(U,g(t),x)$ is $\delta$-close to the standard cylindrical flow on the time interval $[-c,0]$ at scale $\rho(x)$. We simply call it a strong $\delta$-neck when $c=-\delta^{-1}$.
\end{defn}

\begin{defn}[Canonical neighborhood assumption]
Let $(M,g)$ be a 3 dimensional Riemannian manifold and $\epsilon>0$. We say that $(M,g)$ satisfies the $\epsilon$-canonical neighborhood assumption at some point $x\in M$ if there is a $\kappa$-solution $(\overline{M},\overline{g}(t)_{t\in(-\infty,0]})$ and a point $\ox\in\overline{M}$ such that $\rho(\ox,0)=1$ and $(M,g,x)$ is $\epsilon$-close to $(\overline{M},\overline{g}(0),\ox)$ at scale $\rho(x)>0$.

We say that $(M,g)$ satisfies the $\epsilon$-canonical neighborhood assumption at scales $(r_1,r_2)$, for some $r_2>r_1>0$, if $M$ satisfies the $\epsilon$-canonical neighborhood assumption at every point $x\in M$ with $r_1<\rho(x)<r_2$.

\end{defn}

\begin{lem}\label{l: derivative}(Gradient estimate, see e.g. \cite[Lemma 8.1]{BK})
There exist $\overline{\epsilon}$, and $\eta>0$ such that for all $\epsilon\le\overline{\epsilon}$ the following holds: If $\M$ is a Ricci flow spacetime satisfying the $\epsilon$-canonical neighborhood assumption at some point $x\in\M$, then
\begin{equation}\label{e: derivative}
    |\nabla\rho|(x)\le \eta,\quad
    |\partial_t\rho^2(x)|\le \eta.
\end{equation}


\end{lem}

Hereafter, we always assume $\epsilon>0$ to be smaller than the $\overline{\epsilon}$ from the above lemma whenever we talk about the $\epsilon$-canonical neighborhood assumption.

\begin{lem}\label{l: either neck or cap}(\cite[Lemma 8.2]{BK})
For every $\delta>0$ there are constants $C_0(\delta),\epsilon_{can}(\delta)>0$ such that if
    $\epsilon\le\epsilon_{can}(\delta)$,
then the following holds.

Let $(M,g)$ be a Riemannian manifold that satisfies the $\epsilon$-canonical neighborhood assumption at some point $x\in M$. 
Then $x$ is contained in a compact, connected domain $V\subset M$ such that $\textnormal{diam}(V)\le C_0\rho(x)$ and $\rho(y_1)\le C_0\rho(y_2)$ for all $y_1,y_2\in V$, and one of the following hold:
\begin{enumerate}
    \item $V$ is a $\delta$-neck at scale $\rho(x)$ and $x$ is its center.
    \item $V$ is a closed manifold without boundary. 
    \item Either $V$ is a $3$-disk or is diffeomorphic to a twisted interval bundle over $\mathbb{R}P^2$ and $\partial V$ is a central 2-sphere of a $\delta$-neck. We call $V$ a $\delta$-cap and $x$ its center. Moreover, for any $y_1,y_2\in\partial V$, we have $d(y_1,x)+d(y_2,x)\ge d(y_1,y_2)+100\rho(x)$.
\end{enumerate}
\end{lem}

\begin{defn}[$\delta$-tube and capped $\delta$-tube]
A \textit{$\delta$-tube} $T$ in a Riemannian 3-manifold $M$ is a submanifold diffeomorphic to $S^2\times\mathbb{R}$ which is a union of $\delta$-necks with the central spheres that are isotopic to the 2-spheres of the product structure.

A capped $\delta$-tube is a connected submanifold that is the union of a $\delta$-cap and a $\delta$-tube where the intersection of them is diffeomorphic to $S^2\times\mathbb{R}$ and contains an end of the $\delta$-tube and an end of the $\delta$-cap.
\end{defn}

\begin{lem}(High curvature regions are covered by tubes and capped-tubes)\label{l: central sphere decomposition}
For some sufficiently small $\delta>0$, there exist 
$\epsilon_{can}(\delta), C_0(\delta), \lambda(\delta),\Lambda(\delta)>0$ such that the following holds:

Let $(M,g,x_0)$ be a 3 dimensional Riemannian manifold which is $0$-complete, $x_0\in M$, $\rho(x_0)\ge C_0$. Suppose the $\epsilon_{can}$-canonical neighborhood assumption holds at scales $(0,1)$ on $B_g(x_0,d)$ for some $d\ge 2$.
Let $r_0\in(0,1)$, then there exists a collection $S$ of disjoint $\delta$-tubes and capped $\delta$-tubes in $B_g(x_0,d)$ such that
\begin{enumerate}
    \item For all $x\in B_g(x_0,d-1)$ with $\rho(x)\le \lambda r_0$, $x$ is contained in $\bigcup_{V\in S}V$.
    \item $\rho\le \Lambda r_0$ on $\bigcup_{V\in S}V$.
    \item The boundary components of all $V\in S$ are central spheres of some $\delta$-neck with scale equal to $r_0$.
    \item Suppose for some $C>0$, $vol(B_g(x_0,d))\le C$. Then there is $N(C,r_0)>0$ such that the number of elements in $S$ is less than $N$.
\end{enumerate}
\end{lem}

\begin{proof}[Proof of Lemma \ref{l: central sphere decomposition}]
Let $x\in B_g(x_0,d)$ with $\rho(x)< 1$. Then by $\rho(x_0)\ge C_0$ and Lemma \ref{l: either neck or cap}, $x$ must be the center of a $\delta$-neck or a $\delta$-cap.
Note also that by \cite[Proposition 19.21]{MT}, a connected subset consisting of points which are centers of a $\delta$-neck or a $\delta$-cap is a $\delta$-tube or a capped $\delta$-tube.
Now the proof is an easy adaptation of \cite[Lemma 9.4]{BK2}.
\end{proof}

\begin{defn}[Singular Ricci flow]\label{d: SRF}
If $\epsilon>0$ and $r(t):[0,\infty)\rightarrow (0,\infty)$ is a non-increasing function. Then we say a Ricci flow space time $\M$ is an $(r,\epsilon)$-singular Ricci flow if the following holds:
\begin{enumerate}
    \item $\M_0$ is a compact orientable manifold;
    \item $\M$ is $0$-complete;
    \item $\M_{[0,t]}$ satisfies the $\epsilon$-canonical neighborhood assumption at scales $(0,r(t))$.
\end{enumerate}
We call $t$ a singular time if $\M_t$ is not compact.
\end{defn}

It was shown in \cite[Theorem 1.3]{KL1} that for any 3 dimensional compact Riemannian manifold $(M,g)$, there exists an $(r,\epsilon)$-singular Ricci flow with $\M_0$ isometric to $(M,g)$.

\end{subsection}

\begin{subsection}{Distance distortion estimates}\hfill\\
In this subsection, we review some standard distance distortion estimates under different curvature conditions. which are originally due to Hamilton and Perelman.

\begin{lem}\label{l: distance laplacian}(see e.g. \cite[Theorem 18.7]{RFTandA3})
Let $(M,g(t))_{t\in[0,T]}$ be a Ricci flow of dimension $n$. Let $K,r_0>0$.
\begin{enumerate}
    \item Let $x_0\in M$ and $t_0\in(0,T)$. Suppose that $\Ric\le(n-1)K$ on $B_{t_0}(x_0,r_0)$. Then the distance function $d(x,t)=d_t(x,x_0)$ satisfies
\begin{equation}
    (\pt-\Delta)|_{t=t_0}d\ge -(n-1)(\frac{2}{3}Kr_0+r_0^{-1}).
\end{equation}  
   \item Let $t_0\in[0,T)$ and $x_0,x_1\in M$. Suppose 
\begin{equation}
    \Ric(x,t_0)\le(n-1)K,
\end{equation}
for all $x\in B_{t_0}(x_0,r_0)\cup B_{t_0}(x_1,r_0)$.
Then 
\begin{equation}
    \pt|_{t=t_0}d_t(x_0,x_1)\ge -2(n-1)(Kr_0+r_0^{-1}).
\end{equation}
\end{enumerate}
\end{lem}

\begin{lem}\label{l: expanding lemma}(see e.g. \cite[Lemma 2.1, 2.2]{mollification})
Let $(M,g(t))_{t\in[0,T]}$ be a Ricci flow of dimension $n$, $x_0\in M$.
Let $K,A>0$.
\begin{enumerate}
    \item Suppose $\Ric_{g(t)}\ge -K$ on $ B_t(x_0,A)\subset\subset M$ for all $t\in[0,T]$. Then the following holds for all $t\in[0,T]$:
\begin{equation}
    B_0(x_0,Re^{-KT})\subset B_t(x_0,Re^{-K(T-t)}).
\end{equation}
    \item Suppose $\Ric_{g(t)}\le K$ on $ B_t(x_0,R)\subset\subset M$ for all $t\in[0,T]$. Then the following holds for all $t\in[0,T]$:
\begin{equation}
    B_t(x_0,Re^{-Kt})\subset B_0(x_0,R).
\end{equation}
\end{enumerate}

\end{lem}

\end{subsection}

\end{section}

\begin{section}{Preparatory results}\label{s: Preparatory results}
In this section we prove some technical lemmas that will be used later.
First, in Lemma \ref{l: volume comparison}-\ref{l: abundance of neck points}, we study manifolds that are $0$-complete and satisfy a canonical neighborhood assumption at scales depending on the distance to a base point $x_0$. We show for such manifolds that any metric ball of a fixed radius centered at $x_0$ is uniformly totally bounded (Lemma \ref{l: volume comparison}), and for any point $x\in M$, there exists a minimizing geodesic from $x_0$ to $x$ (Lemma \ref{l: existence of minimizing geodesic}). 

Second, we show in Lemma \ref{l: close to k-solution} that for a Ricci flow spacetime with some appropriate canonical neighborhood assumption, the closeness of a time-slice to a $\kappa$-solution implies the closeness in a parabolic region. As a consequence, Lemma \ref{l: blow-up converges to a k-solution} shows that a blow-up sequence in a singular Ricci flow converges to a $\kappa$-solution defined from $-\infty$ to its maximal existence time.

\begin{lem}(Metric balls are uniformly totally bounded)\label{l: volume comparison}
Let $(M,g,x_0)$ be a 3 dimensional connected Riemannian manifold, $x_0\in M$. 
Suppose $M$ is 0-complete and $1$-positive. 
Suppose for any $A,\epsilon_{can}>0$, there are  $r(A,\epsilon_{can}),\kappa(A)>0$ such that the $\epsilon_{can}$-canonical neighborhood assumption and the $\kappa(A)$-non-collapsing assumption hold at scales less than $r(A,\epsilon_{can})$ on $B_g(x_0,A)$.

Then for any given $A,\epsilon>0$, there exist $V(A),N(A,\epsilon)>0$ such that $vol(B_g(x_0,A))\le V(A)$, and
the number of elements in any $\epsilon$-separating subset of $B_g(x_0,A)$ is bounded above by $N(A,\epsilon)$.
\end{lem}

\begin{proof}
Fix some small $\delta>0$ and let $\epsilon_{can}(\delta),C_0(\delta)>0$ be from Lemma \ref{l: either neck or cap}. Let $\eta>0$ be from Lemma \ref{l: derivative}.
Let $r_0=\frac{1}{7\eta}\cdot\min\{ r(A+1,\epsilon_{can}),\frac{1}{6}\}$. 

We claim that there exists a universal constant $C_1>0$ such that $vol(B_g(x,3r_0))\le C_1$ for all $x\in B_g(x_0,A)$.
In fact, suppose first that $\rho(x)<4\eta r_0$, then by Lemma \ref{l: derivative} we get $\rho<7\eta r_0$ on $B_g(x,3r_0)$.
By the assumption of $r_0$, we see that $B_g(x,3r_0)$ is contained in some $\delta$-tube or capped $\delta$-tube with diameter less than $1$, which has volume less than $C_1$.
So the claim holds.
Next, suppose $\rho(x)\ge4\eta r_0$, then by Lemma \ref{l: derivative} we get $\rho\ge\eta r_0$ on $B_g(x,3r_0)$. So the claim follows from the Bishop-Gromov volume comparison.

Now suppose by induction that for some $k\in\mathbb{N}$ with $kr_0\le A$, there exist $C_k(A)>0$ and $N_k(A,\epsilon)\in\mathbb{N}_+$ such that the following holds:
\begin{enumerate}
    \item[$\mathcal{A}$(k):] $vol(B_g(x_0,kr_0))\le C_k$;
    \item[$\mathcal{B}$(k):] For any $\epsilon\le r_0$, an $\epsilon$-separating subset in $B_g(x_0,(k-1)r_0)$ has at most $N_k$ elements.
\end{enumerate}
To show $\mathcal{A}(k+1)$, consider a maximal $r_0$-separating subset $\{y_j\}$ in $B_g(x_0,(k-1)r_0)$. 
Then $B_g(x_0,(k-1)r_0)$ is covered by the union of all $B_g(y_j,r_0)$, and by the triangle inequality we get
\begin{equation}\label{e: inclusion_2}
    B_g(x_0,(k+1)r_0)\subset \bigcup_j B_g(y_j,3r_0).
\end{equation}
So $\mathcal{A}$(k+1) follows from the above claim and $\mathcal{B}$(k). It remains to establish $\mathcal{B}(k+1)$.



Let $\{x_j\}_{j=1}^{m}$ be an $\epsilon$-separating set in $B_g(x_0,kr_0)$. 
First, by the non-collapsing assumption and $\mathcal{A}(k+1)$ we see that the number of $x_j$ with $\rho(x_j)\ge(2C_0)^{-1}\epsilon$ is bounded above in terms of $A,\epsilon$.
So we may assume $\rho(x_j)< (2C_0)^{-1}\epsilon$ for all $j$.
Then by $\mathcal{A}(k+1)$ and Lemma \ref{l: central sphere decomposition}, we may further assume that all $x_j$ are contained in a single $\delta$-tube or capped $\delta$-tube $V\subset B_g(x_0,(k+1)r_0)$.

Pick a point $y\in\partial V$, we can arrange the order of $\{x_j\}_{j=1}^m$ in a way such that $d_V(y,x_{j+1})\ge d_V(y,x_{j})$ for each $j\le m-1$. Here $d_V$ denotes the length metric in $V$ induced by $g$.
We claim that each $x_j$, $j\le m-1$, is the center of a $\delta$-neck. 
Otherwise, $x_j$ is the center of a $\delta$-cap $\mathcal{C}\subset V$.
By Lemma \ref{l: either neck or cap} we have diam($\mathcal{C})\le C_0\rho(x_j)\le \epsilon/2$. 
Since $d_V(x_j,x_{j+1})\ge\epsilon$, we get $x_{j+1}\in V-\mathcal{C}$, and $x_{j+1}$ is the center of a $\delta$-neck.
Connecting $y$ with $x_j$ by a minimizing geodesic, then it must intersect the central sphere at $x_{j+1}$, which has diameter less than $10(2C_0)^{-1}\epsilon<\epsilon/2$. So it is easy to see $d_V(y,x_{j+1})< d_V(y,x_j)$, a contradiction.

So by the triangle inequality we get
\begin{equation}\label{e: pi}
    \begin{split}
        d_V(y,x_{j+1})&
        \ge d_V(y,x_{j})+\epsilon-2\cdot10\cdot(2C_0)^{-1}\epsilon\ge d_V(y,x_{j})+\epsilon/2,
    \end{split}
\end{equation}
for all $j\le m-1$. 
In particular, this implies
\begin{equation}
    (m-1)\epsilon/2\le d_V(x_{m},y)\le \textnormal{diam}(V)\le 2(k+1)r_0,
\end{equation}
and hence $m\le 4(k+1)r_0\epsilon^{-1}+1$. This established $\mathcal{B}(k+1)$.
\end{proof}

\begin{lem}\label{l: existence of minimizing geodesic}
Under the same assumptions as Lemma \ref{l: volume comparison}.
Then for any $x\in M$, there exists a minimizing geodesic connecting $x$ to $x_0$.
\end{lem}

\begin{proof}
By Lemma \ref{l: volume comparison} we have $vol(B_g(x_0,A))$ is bounded above, so by Lemma \ref{l: derivative} it is easy to see that $|\Rm|$ is a proper function restricted on $\overline{B_g(x_0,A)}$ for all $A>0$.
Suppose $d=d_g(x_0,x)>0$. Fix a sufficiently small number $\delta>0$, and let $\epsilon_{can}=\epsilon_{can}(\delta),C_0=C_0(\delta)$ be from Lemma \ref{l: either neck or cap}. 
Let $r_0=\min\{C_0^{-1}\rho(x_0),C_0^{-1}\rho(x),r(d+2,\epsilon_{can}),1\}$.
Then by Lemma \ref{l: central sphere decomposition} and Lemma \ref{l: volume comparison}, all points in $B_g(x_0,d+1)$ with $\rho\le r_0$ are contained in the union of a finite collection $S$ of disjoint $\delta$-tubes and capped $\delta$-tubes in $B_g(x_0,d+2)$.

For each $i\in \mathbb{N}$, let $\gamma_i:[0,d]\rightarrow M$ be a smooth curve joining $x_0$ and $x$ with constant speed, such that the length of $\gamma_i$ satisfies $L(\gamma_i)<d+\frac{1}{i}$. So $\gamma_i\subset B_g(x_0,d+1)$.
We claim that the curvature on $\gamma_i$ is uniformly bounded for all $i$.
Suppose not, we may assume there is a sequence of points $x_i\in\gamma_i$ such that $\rho(x_i)\rightarrow 0$ as $i\rightarrow\infty$.
By the finiteness of $S$ we may also assume there is some $\mathcal{T}\in S$ that contains all $x_i$.
Then $\mathcal{T}$ must be a $\delta$-tube with curvature blowing up in one end.
Taking a point $y\in\mathcal{T}$ such that $\rho(y)<\frac{1}{2}\min\{\rho(x_0),\rho(x)\}$.
Then for all large $i$, $\gamma_i$ passes through the $\delta$-neck centered at $y$ at least twice, which contradicts the almost minimality of $\gamma_i$.
So the claim holds.
Therefore, by the properness of $|\Rm|$, there is a compact set $K\subset M$ such that $\gamma_i\subset K$ for all $i$.

Since $|\gamma'_i|=\frac{L(\gamma_i)}{d}\rightarrow 1$, it follows that $\gamma_i$'s are uniformly Lipschiz-continuous on $[0,d]$. Since $\gamma_i$'s are equicontinuous and map into a compact set of $M$, the Arzela-Ascoli Lemma applies. So by passing to a subsequence, we may assume that $\gamma_i$ uniformly converges to some continuous curve $\gamma_{\infty}:[0,d]\rightarrow M$.
Since $\int_{0}^d |\gamma'_i|^2\,dt=\frac{L(\gamma_i)^2}{d}\le 4d$, we can apply weak compactness to the sequence $\{\gamma_i\}$. By passing to a subsequence, we may assume $\gamma_i$ weakly converges to $\gamma_{\infty}$ in $W^{1,2}$. 

Let $E(\gamma)=\int_0^d |\gamma'(t)|^2\,dt$ be the energy function on all $W^{1,2}$-path connecting $x_0$ and $x$.
Then by Cauchy-Schwarz inequality we have
\begin{equation}
    E(\gamma)=\int_0^d|\gamma'(t)|^2\,dt\ge\frac{(\int_0^d|\gamma'|)^2}{d}=\frac{L(\gamma)^2}{d}\ge d.
\end{equation}
By the weak semi-continuity of the $E$-energy, it follows that the $W^{1,2}$-path $\gamma_{\infty}$ has energy $E(\gamma_{\infty})\le d$. 
So $\gamma_{\infty}$ minimizes the energy $E$ in $W^{1,2}$. 
Therefore, $\gamma_{\infty}$ is a smooth solution to the geodesic equation, and hence it is a minimizing geodesic.

\end{proof}

The next lemma says that if a ball is scathed, then we can find a minimizing geodesic in the ball along which the curvature blows up, and it is covered by $\delta$-necks.

\begin{lem}(Minimal geodesic covered by $\delta$-necks)\label{l: abundance of neck points}
Under the same assumptions as Lemma \ref{l: volume comparison}. Let $\delta>0$. 
Suppose $\inf_{B_g(x_0,A)}\rho=0$ for some $A>0$. 
Then there exists a minimizing geodesic $\gamma:[0,1)\rightarrow B_g(x_0,A)$ such that $R(\gamma(s))\rightarrow\infty$ as $s\rightarrow1$, and $\gamma(s)$ is the center of a $\delta$-neck for all $s$ close to $1$.
\end{lem}

\begin{proof}
Let $x_i\in B_g(x_0,A)$ be a sequence of points such that $R(x_i)\rightarrow\infty$ as $i\rightarrow\infty$. 
Since by Lemma \ref{l: volume comparison} the ball $B_g(x_0,A)$ is totally bounded, we may assume by passing to a subsequence that $\{x_i\}$ is Cauchy.
So there is a $\delta$-tube $\mathcal{T}$, which blows up at one end, that contains $x_i$ for all large $i$. 
By Lemma \ref{l: existence of minimizing geodesic}, there exists $\gamma_i:[0,1]\rightarrow M$, which is the minimizing geodesic connecting $x_0$ and $x_i$. 
Noting that $\gamma_i$ passes through any $\delta$-neck in $\mathcal{T}$ at most once, after passing to a subsequence, $\gamma_i$ converges to a minimizing geodesic $\gamma:[0,1)\rightarrow M$. Moreover, by the minimality of $\gamma$, $\gamma(s)$ is the center of a $\delta$-neck for $s$ close enough to $1$.

\end{proof}

The following lemma shows that for a Ricci flow spacetime satisfying some appropriate canonical neighborhood assumption, suppose a time-slice is close enough to that of a $\kappa$-solution, then a parabolic region of a certain size is close to that in the $\kappa$-solution.

\begin{lem}(Time-slice closeness implies spacetime closeness)\label{l: close to k-solution}
Let $(M_{\infty},g_{\infty}(t),x_{\infty})$ be a $\kappa$-solution, and $a,b>0$ be constants such that $g_{\infty}(t)$ is defined for all $t\in[-a,b]$. Let $\delta>0$. Then there exists $\epsilon>0$ such that the following holds:

Let $\M$ be a Ricci flow spacetime, $x_0\in\M$, $t_0:=\mathfrak{t}(x_0)$. Suppose $(\M_{t_0},x_0)$ is $\epsilon$-close to $(M_{\infty},g_{\infty}(0),x_{\infty})$. Suppose also that
$B_{t_0}(x_0,\epsilon^{-1})$ survives on $[t_0-a,t_0+b]$, and the $\epsilon$-canonical neighborhood assumption holds at every point in $\bigcup_{t\in[t_0-a,t_0+b]}(B_{t_0}(x_0,\epsilon^{-1}))(t)$. 

Then $(\M,x_0)$ is $\delta$-close to $(M_{\infty},g_{\infty}(t),x_{\infty})$ on the time interval $[-a,b]$.
\end{lem}

\begin{proof}
Suppose the assertion does not hold, then there is a sequence of spacetimes $\M_i$ and points $x_{0i}\in\M_i$, and a sequence of $\epsilon_i>0$ with $\lim_{i\rightarrow\infty}\epsilon_i=0$, such that the assumptions are satisfied for each $i$, but the conclusion fails. 

Let $-a^*,b^*$ be the infimum and supremum of $s_1,s_2\in[-a,b]$, respectively, such that there exists $C$ (may depend on $s_1,s_2$) such that for all $d>0$, $|\Rm|\le C$ in $\bigcup_{t\in[t_{0i}-s_1,t_{0i}+s_2]}(B_{t_{0i}}(x_{0i},d))(t)$.
Then by the gradient estimate we have $a^*,b^*>0$. 
So by passing to a subsequence, we may assume $(\M_{i},x_{0i})$ converges to a smooth complete Ricci flow $(\widehat{M},\widehat{g}(t),\widehat{x})$, $t\in(-a^*,b^*)$, which has bounded curvature in any compact subinterval in $(-a^*,b^*)$.

Note that $(\widehat{M},\widehat{g}(0),\widehat{x})$ is isometric to $(M_{\infty},g_{\infty}(0),x_{\infty})$, and each time-slice of $(\widehat{M},\widehat{g}(t))$ is a time-slice of a $\kappa$-solution, it is easy to see that the flow $(\widehat{M},\widehat{g}(t),\widehat{x})$ is isometric to $(M_{\infty},g_{\infty}(t),x_{\infty})$ for all $t\in(-a^*,b^*)$.

So $(\widehat{M},\widehat{g}(t),\widehat{x})$ extends to a complete Ricci flow with bounded curvature on $[-a^*,b^*]$. Applying the gradient estimates on $(B_{t_{0i}}(x_{0i},\epsilon^{-1}))(t)$ for $t=t_{0i}-a^*$ and $t_{0i}+b^*$, and seeing the assumptions of $a^*$ and $b^*$, we get $a^*=a$ and $b^*=b$. From this it follows that $(\M_{i},x_{0i})$ is $\delta$-close to $(M_{\infty},g_{\infty}(t),x_{\infty})$ on $[-a,b]$, which is a contradiction.

\end{proof}

\begin{lem}(Blow-up sequence converges to a $\kappa$-solution)\label{l: blow-up converges to a k-solution}
Let $\M$ be a singular Ricci flow with normalized initial condition. Let $x_i\in M$ be a sequence of points such that $\sup_i\mathfrak{t}(x_i)<\infty$, and $\sup_i|\Rm|(x_i)=\infty$. Then there exists a subsequence $\{x_{i_k}\}$, such that $(\M,x_{i_k})$ converges to a $\kappa$-solution on its maximal existence time interval.
\end{lem}

\begin{proof}
By \cite[Theorem 1.3]{KL1}, the canonical neighborhood assumption for singular Ricci flows, and the gradient estimates, the conditions in Lemma \ref{l: close to k-solution} are satisfied. Assume $(\M_{\mathfrak{t}(x_i)},x_i)$ converges to the time-0-slice of a $\kappa$-solution $(M_{\infty},g_{\infty}(t),x_{\infty})$. Then by Lemma \ref{l: close to k-solution}, after passing to a subsequence, $(\M,x_{i_k})$ converges to a $\kappa$-solution on its maximal existence time interval.
\end{proof}

\end{section}

\begin{section}{Canonical neighborhood theorem}
\label{s: CNA and NC}
The main results in this section are a canonical neighborhood theorem (Proposition \ref{p: noncollapsing}) for singular Ricci flows, and a bounded curvature at bounded distance theorem by reduced volume (Theorem \ref{t: bdd curvature at bdd distance}).

Recall that Perelman proved a canonical neighborhood theorem for compact smooth Ricci flows \cite[Theorem 26.2]{KL}, which says that assuming the Ricci flow has normalized initial condition, then for any $T>0$ there exists $r(T)\ge0$ such that the canonical neighborhood assumption holds in $\M_{t\le T}$ at scales less than $r(T)$.

He also proved a local version of the theorem \cite[Proposition 85.1]{KL}, in which he assumed the curvature is bounded in a backward parabolic neighborhood of a point $x_0$, $\mathfrak{t}(x_0)=t_0$, and showed that for all $A>0$ there exists $r(A)>0$, such that the canonical neighborhood assumption holds in $B_{t_0}(x_0,A)$ at scales less than $r(A)$.


The following proposition extends the local theorem to singular Ricci flows. 

\begin{prop}(Canonical neighborhood theorem)\label{p: noncollapsing}
For any $A>0$, $\epsilon>0$, there are constants $\kappa(A), r(A,\epsilon),\overline{r}(A),K(A)>0$, such that the following holds: Let $\M$ be a singular Ricci flow, $x_0\in\M$, $\mathfrak{t}(x_0)=t_0>0$. Suppose for some $r_0>0$ with $2r_0^2<t_0$ the following holds:
\begin{enumerate}
\item $\M$ is unscathed on a parabolic neighborhood $P(x_0,r_0,-r_0^2)$.
\item $|\Rm|\le r_0^{-2}$ on $P(x_0,r_0,-r_0^2)$.
\item $vol(B_{t_0}(x_0,r_0))\ge A^{-1}r_0^3$.
\end{enumerate}
Then \begin{enumerate}
\item[(a)]The solution is $\kappa$-non-collapsed at scales less than $r_0$ in $B_{t_0}(x_0,Ar_0)$.
\item[(b)]The $\epsilon$-canonical neighborhood assumption holds in $B_{t_0}(x_0,Ar_0)$ at scales less than $rr_0$.
\item[(c)]If $r_0\le\overline{r}\sqrt{t_0}$ then $|\Rm|\le Kr_0^{-2}$ in $B_{t_0}(x_0,Ar_0)$.
\end{enumerate}
\end{prop}

As an application of Proposition \ref{p: noncollapsing}, as well as other results by Perelman, we can generalize his bounded curvature at bounded distance theorem by using the reduced volume:
\begin{theorem}(Bounded curvature at bounded distance by reduced volume)\label{t: bdd curvature at bdd distance}
For any $A,\kappa>0$, there exist $\varphi(A,\kappa),K(A,\kappa)>0$ such that the following holds:
Let $\M$ be a singular Ricci flow, $x_0\in\M$, $t_0:=\mathfrak{t}(x_0)$.
Suppose the reduced volume $\tilde{V}_{x_0}(1)\ge\kappa$, and $\M_{t\ge t_0-1}$ is $\varphi$-positive. 

If $|\Rm|(x_0)\le 1$, then $|\Rm|\le K$ in $B_{t_0}(x_0,A)$.
\end{theorem}

In the theorem, $\tilde{V}_{x_0}(1)$ is the reduced volume at $\tau=1$, for the base point $x_0$. It is defined by
\begin{equation}
    \tilde{V}_{x_0}(\tau)=\int_{\M_{t_0-\tau}}\tau^{-\frac{3}{2}}e^{-\ell_{x_0}(x)}\,d_{t_0-\tau}x.
\end{equation}


To prove Proposition \ref{p: noncollapsing} and Theorem \ref{t: bdd curvature at bdd distance}, we first prove some results in $\LL$-geometry for singular Ricci flows.

\begin{defn}
The $\LL_+$-length of an admissible curve $\gamma:[t_0-\tau,t_0]\rightarrow\M$ is
\begin{equation}
    \LL_+(\gamma)=\int_{t_0-\tau}^{t_0}\sqrt{t_0-t}(R_+(\gamma(t))+|\gamma'(t)|^2)\,dt.
\end{equation}
\end{defn}

\begin{lem}\label{l: L length_0}
Given constants $T,E$ and $\Lambda>0$, there exists $r=r(\Lambda, T,E)$ such that the following holds.
Let $\M$ be a singular Ricci flow with normalized initial condition. Suppose $\gamma:[a,b]\rightarrow\M$ is a smooth admissible curve with $b\le T$. Suppose $b-a\ge E$, and $\rho(\gamma(a))\le r$. 

Then $\int_{a}^{b} R_+(\gamma(t))+|\gamma'(t)|^2\,dt>\Lambda$.
\end{lem}

\begin{proof}
Suppose for some $\Lambda,T,E>0$, the conclusion does not hold, then we can find a sequence of singular Ricci flows $(\M_k,g_k(t))$ and a sequence of smooth curves $\gamma_k:[a_k,b_k]\rightarrow\M_k$ satisfying the assumptions in the theorem.
In particular, we have $\rho(\gamma_k(a_k))\le r_k$, and $r_k\rightarrow0$ as $k\rightarrow\infty$, but 
\begin{equation}\label{e: Lambda}
    \int_{a_k}^{b_k} (R_+(\gamma_k(t))+|\gamma_k'(t)|^2)\,dt\le \Lambda.
\end{equation}

We rescale each $(\M_k,g_k(t),\gamma_k(a_k))$ by $\rho_k^{-2}:=\rho^{-2}(\gamma_k(a_k))$ and then shift time $a_k$ to $0$, to obtain a sequence of Ricci flow spacetimes $(\tilde{\M}_k,\tilde{g}_k(t),\gamma_k(a_k))$ defined on $[0,\infty)$. 
Then $\gamma_k:[0,(b_k-a_k)\rho_k^{-2}]\rightarrow\tilde{\M_k}$ is an admissible curve and $(b_k-a_k)\rho_k^{-2}\rightarrow\infty$ as $k\rightarrow\infty$.
Note that \eqref{e: Lambda} is invariant under rescaling. 
Applying Lemma \ref{l: blow-up converges to a k-solution} and passing to a subsequence we may assume that $(\tilde{\M}_k,\tilde{g}_k(t),\gamma_k(a_k))$ converges to a $\kappa$-solution $(M_{\infty},g_{\infty}(t),x_{\infty})$ on $[0,\infty)$. So by \cite{classification}, $(M_{\infty},g_{\infty}(t))$ is either a Bryant soliton \cite{classification}, or has finite extinction time.

Suppose first that $(M_{\infty},g_{\infty}(t))$ is a Bryant soliton. Let $\M_{\infty}$ denote the Ricci flow spacetime associated to it.
Fix a large $A>0$ and let $\mathcal{P}_A:=\bigcup_{t\in[0,A^{3/2}]}B_t( x_{\infty},A)$. 
Then for large $k$, there exists a time-preserving diffeomorphism $\phi_k:U_k\rightarrow V_k$, with $\phi_k(x_{\infty})=\gamma_k(a_k)$, where $U_k$ and $V_k$ are open subsets of $\M_{\infty}$ and $\tilde{\M_k}$ respectively, such that given any compact subset $K\subset\M_{\infty}$ and $\delta>0$, we have $K\subset U_k$ for all large $k$, and $\|\phi_k^*G_k-G_{\infty}\|_{C^{[\delta^{-1}]}(K,G_{\infty})}$, where $G_k$ and $G_{\infty}$ are spacetime metrics of $\tilde{\M_k}$ and $\M_{\infty}$ respectively.
So we have $\mathcal{P}_A\subset U_k$ for all large $k$.
Let $\widehat{\gamma}_{k}\subset U_k$ be the image of $\gamma_k|_{V_k}$ under $\phi_k^{-1}$.

Since $(b_k-a_k)\rho_k^{-2}\rightarrow\infty$ as $k\rightarrow\infty$, we see that $\gamma_k$ must exit $\phi_k(\mathcal{P}_A)$ at some time $T_k\in[0,A^{3/2}]$, and accordingly $\widehat{\gamma}_k$ exits $\mathcal{P}_A$ at the time $\widehat{T}_{k}\in[0,A^{3/2}]$.
Then for sufficiently large $k$, we have
\begin{equation}\label{e: one half}
\begin{split}
    \int_{0}^{T_k} |\gamma_k'(t)|^2dt&\ge\frac{1}{2}\int_{0}^{\widehat{T}_{k}}|\widehat{\gamma}_{k}'(t)|^2dt,\\
    \int_{0}^{T_k} R(\gamma_k(t))dt&\ge\frac{1}{2}\int_{0}^{\widehat{T}_{k}} R(\widehat{\gamma_k}(t))dt.
\end{split}
\end{equation}

Suppose $T_{k}<A^{3/2}$. Since $\Ric\ge 0$, we have $\pt g_{\infty}(t)\le 0$ for all $t\ge 0$.
So $|\widehat{\gamma}_{k}'(t)|_{g_{\infty}(t)}\ge |\widehat{\gamma}_{k}'(t)|_{g_{\infty}(T_k)}$ for all $t\in[0,T_k]$, and hence
\begin{equation}\label{e: far away}
    \begin{split}
        \int_{0}^{T_k}|\widehat{\gamma}_{k}'(t)|^2\,dt\ge\int_{0}^{T_k}|\widehat{\gamma}_{k}'(t)|_{T_k}^2\,dt
        \ge \frac{d^2_{T_k}(\widehat{\gamma}_{k}(0),\widehat{\gamma}_{k}(T_k))}{T_k}\ge A^{1/2}.
    \end{split}
\end{equation}
Otherwise, we have $T_{k}=A^{3/2}$, and $\widehat{\gamma}_{k}(t)\in B_t((\widehat{\gamma}_{k}(0))(t),A)$ for all $t\in[0,A^{3/2}]$. 
Since $(M_{\infty},g_{\infty}(t))$ is a Bryant soliton, we have $R(\widehat{\gamma}_{k}(t))\ge\frac{C}{A}$, where $C$ is a constant depending only on the curvature of the tip.
So we have
\begin{equation}
    \int_{0}^{A^{3/2}}R(\widehat{\gamma}_{k}(t))\,dt\ge\int_{0}^{A^{3/2}}\frac{C}{A}\,dt=CA^{1/2}.
\end{equation}
In both cases, taking $A$ sufficiently large, it follows by \eqref{e: one half} that
\begin{equation}
    \int_{0}^{T_k}(R(\gamma_{k}(t))+|\gamma_{k}'|^2)\,dt>\Lambda,
\end{equation}
a contradiction to \eqref{e: Lambda}. 

Now suppose $(M_{\infty},g_{\infty}(t))$ has a finite extinction time $T_{\infty}<\infty$. 
Let $\theta\in(0,T_{\infty})$ and consider $\mathcal{P}_A:=\bigcup_{t=0}^{\theta}B_t( x_{\infty},A)$.
Let $T_k$ and $\widehat{T}_k$ be defined as above.
Then if $\widehat{T}_k<\theta$, then \eqref{e: far away} holds.
Otherwise, since the scalar curvature blows up at the rate of $(T_{\infty}-t)^{-1}$ when the time goes up to $T_{\infty}$, we have
\begin{equation}
    \int_{0}^{\theta}R(\widehat{\gamma}_k(t))dt\ge \int_{0}^{\theta} \frac{C}{T_{\infty}-t}dt=-C\log(T-\theta). 
\end{equation}
By taking $\theta$ sufficiently close to $T_{\infty}$ and $A$ sufficiently large, we get a contradiction. 

\end{proof}

The following lemma says that an admissible curve that contains a point of large curvature must have large $\LL_+$-length.
\begin{lem}\label{l: L length}
For all $\Lambda<\infty, \overline{r},T>0$, there is a constant $\delta=\delta(\Lambda,\overline{r},T)$ with the following property:

Let $\M$ be a singular Ricci flow with normalized initial condition, and $x_0\in\M_{t_0}$ with $0<t_0\le T$. Suppose $r_0\ge \overline{r}$, and $P(x_0,r_0,-r_0^2)$ is unscathed, and $|\Rm|\le r_0^{-2}$ on $P(x_0,r_0,-r_0^2)$.
Suppose also $\gamma:[t_1,t_0]\rightarrow\M$ is an admissible curve ending at $(x_0,t_0)$, and there exists $t\in[t_1,t_0]$ such that $\rho(\gamma(t))<\delta$.

Then $\LL_+(\gamma)>\Lambda$.
\end{lem}

\begin{proof}

First, let $\Delta t=\frac{1}{4}10^{-4}\overline{r}^4\Lambda^{-2}$. 
By taking $\delta<\overline{r}$, we see that $\gamma$ must exit $P(x_0,r_0,-r_0^2)$ at some time $\tilde{t}$. First, suppose $\tilde{t}>t_0-\Delta t$. Then by the Schwarz inequality we get
\begin{equation}
\begin{split}
    \int_{\tilde{t}}^{t_0} \sqrt{t_0-s}|\gamma'(s)|^2\,ds
    &\ge\left(\int_{\tilde{t}}^{t_0}|\gamma'(s)|\,ds\right)^2\left(\int_{\tilde{t}}^{t_0}(t_0-s)^{-1/2}\,ds\right)^{-1}
    \\
    &\ge\frac{1}{2} 10^{-2}r_0^2(\Delta t)^{-1/2}>\Lambda,
\end{split}
\end{equation}
where the factor $10^{-2}$ comes from the distance distortion on $P(x_0,r_0,-r_0^2)$.

So now we may assume that $\gamma$ exits $P(x_0,r_0,-r_0^2)$ at a time $\tilde{t}\le t_0-\Delta t$. 
Then the conclusion follows immediately from Lemma \ref{l: L length_0}.

\end{proof}

\begin{proof}
Note by Lemma \ref{l: L length} and the properness of scalar curvature from \cite[Theorem 1.3]{KL1}, a minimizing sequence of admissible curves between any two points converges to a smooth minimizing $\LL$-geodesic.
Now the rest of proof for part (a) is the same as \cite[Proposition 85.1(a)]{KL}. 

For part (b), suppose that for some $A>0$ the claim is not true. Then there is a sequence of singular Ricci flows $\M_k$ which provide a counterexample.
In particular, there exists $x_k\in B_{t_{0k}}(x_{0k},Ar_{0k})$ with $\rho(x_k)\le r_kr_{0k}$ but at which the $\epsilon$-canonical neighborhood assumption does not hold, where $r_k\rightarrow0$ as $k\rightarrow\infty$.

Omitting the subscripts for a moment, by Lemma \ref{l: blow-up converges to a k-solution} we can apply a point-picking and find points $\ox\in B_{\ot}(x_0,2Ar_0)$, $\ot\in[t_0-r_0^2/2,t_0]$ with $\overline{\rho}:=\rho(\ox,\ot)\le rr_0$ satisfying the following: the $\epsilon$-canonical neighborhood assumption does not hold at $(\ox,\ot)$, but it holds at all points in $\overline{P}$ with $\rho\le \overline{\rho}/2$, where
\begin{equation}
    \overline{P}=\{x\in\M_t:d_t(x_0(t),x)\le d_{\ot}(x_0(\ot),\ox)+r^{-1}\overline{\rho},t\in[\ot-\frac{1}{4}r^{-2}\overline{\rho}^{2},\ot]\}.
\end{equation}
The rest of proof is the same as \cite[Prop 85.1(b)]{KL}, in which we can extract a convergent subsequence of $(\overline{P}_k,g_k(t),\ox_k)$ that converges to a $\kappa$-solution.
Then this contradicts the assumption of $\ox_k$ for large $k$, and thus proves part (b).

Part (c) follows from (b) in the same way as \cite[Prop 85.1(c)]{KL}.

\end{proof}

The following lemma is a corollary of Proposition \ref{p: noncollapsing}(c). Given the curvature and reduced volume bound at a single point, it provides curvature bound in a backward parabolic neighborhood of this point.

\begin{lem}\label{c: curvature blow up}
For any $\kappa>0$, there exists $r(\kappa)>0$ such that the following holds:
Let $\M$ be a singular Ricci flow, $x_0\in\M$, $\mathfrak{t}(x_0)=t_0$. 
Suppose the reduced volume $\tilde{V}_{x_0}(1)\ge\kappa$, and $\M_{t\ge t_0-1}$ is $1$-positive.

If $|\Rm|(x_0)\le 1$, then $|\Rm|(x)\le r^{-2}$ for all $x\in P(x_0,r,-r^2)$.
\end{lem}

\begin{proof}
We will first show that $|\Rm|(x)\le r^{-2}$ for all $x\in B_{t_0}(x_0,r)$. Then the assertion in the theorem follows immediately from this by \cite[Theorem 54.2]{KL}, which gives upper bound on curvature at earlier smaller balls, assuming lower bounds for the volume and curvature in a ball.

For any $x\in\M$, put $\overline{\rho}(x)=\sup\{r>0: |\Rm|(y)\le r^{-2} \textnormal{ for all } y\in B_{\mathfrak{t}(x)}(x,r)\}$.
Suppose the assertion is not true. Then there is a sequence of singular Ricci flows $\M_i$ and points $x_{0i}\in\M_i$, $\mathfrak{t}(x_{0i})=t_{0i}$, that satisfy the assumptions, but $\lim_{i\rightarrow\infty}\overline{\rho}(x_{0i})=0$.

By the reduced volume comparison theorem (see e.g. \cite[Lem 78.11]{KL}), $\tilde{V}_{x_0}(1)\ge\kappa$ implies that there exists $\kappa'(\kappa)>0$ such that $\M_i$ is $\kappa'$-non-collapsed at $x_{0i}$ at scales less than $1$.
Rescale and do a time-shifting to the flows $(\M_{i},x_{0i})$ to get a sequence of new flows, which are still denoted by $(\M_{i},x_{0i})$, such that in the new flows we have $\overline{\rho}(x_{0i})=1$ and $\mathfrak{t}(x_{0i})=0$. 
So $vol(B_0(x_{0i},1))\ge\kappa'>0$ for all $i$. 

By \cite[Theorem 54.2]{KL}, we can find $C(\kappa)>0$ such that $|\Rm|\le C$ in $P(x_{0i},C^{-1/2},-C^{-1})$.
So by applying Proposition \ref{p: noncollapsing}(c) and some distance distortion estimates, we can find a smooth Ricci flow $(U,g_{\infty}(t),x_{\infty})$ (possibly incomplete), $t\in[-c,0]$, for some $c>0$, such that there are diffeomorphisms $\phi_i:U\rightarrow\M_i$ such that $\phi_i(x_{\infty})=x_{0i}$ and $\lim_{i\rightarrow\infty}\|\phi_i^*g_i(t)-g_{\infty}(t)\|_{C^k(U)}=0$ for all $k\in\mathbb{N}$ and $t\in[-c,0]$, and $B_0(x_{\infty},2)$ is relatively compact in $U$.

By the 1-positive pinching assumption, we see that $\Rm(x,t)\ge 0$ for all $x\in U$ and $t\in[-c,0]$. Also, we have $|\Rm|(x_{\infty},0)=0$, $\overline{\rho}(x_{\infty},0)=1$, and hence $|\Rm|(y,0)=1$ for some $y\in U$. However, this contradicts with the strong maximum principle, see e.g. \cite[Theorem 4.18]{MT}.
\end{proof}

Theorem \ref{t: bdd curvature at bdd distance} is a immediate consequence of combining Lemma \ref{c: curvature blow up} and Proposition \ref{p: noncollapsing}(c).

\end{section}

\begin{section}{heat kernel for singular Ricci flow}\label{s: heat kernel}
Let $\M$ be a singular Ricci flow and $x_0\in\M_{t_0}$, $0<t_0<T$. Our first main result in this section is Theorem \ref{t: heat kernel}, in which
we find a positive solution $u$ to the conjugate heat equation $(-\pt-\Delta+R)u=0$, which is a $\delta$-function at $x_0$.
Moreover, $u$ is positive at all points that are accessible to $x_0$, and vanishes elsewhere.
We also show that outside a backward parabolic neighborhood $P(x_0,r_0,-r_0^2)$ around $x_0$, we have $u\le C_0$.
We say such a function $u$ is the heat kernel of $\M$ starting from $x_0$.

The next main result is Theorem \ref{e: uR^m less than C_m}, which improves the bound $u\le C_0$ to the following for any integer $m\ge 1$:
\begin{equation}\label{u}
    u R^m\le C_m,
\end{equation}
where $C_m>0$ depends on $m$.
This indicates that the heat kernel $u$ is sufficiently small at points where the curvature is sufficiently large.
With this estimate we are able to show that the heat kernel in Theorem \ref{t: heat kernel} has many good properties as those of the ordinary heat kernel for compact smooth Ricci flows.

To compare, let $u(\cdot,\cdot)$ be the heat kernel starting from $(x_0,t_0)$ for a smooth compact Ricci flow $(M,g(t))$, $x_0\in M$. Then a direct computation by integration by parts shows
\begin{equation}
    \pt\int u(x,t)\,d_tx=\int (\pt u-Ru)\,d_tx=\int -\Delta u\,d_tx=0.
\end{equation}
So we get $\int u\,d_tx=1$ for all $t< t_0$.
For the singular Ricci flow and the generalized heat kernel, we show in Corollary \ref{l: integral equal to 1} that $\int u\,d_tx=1$ also holds for the heat kernel in Theorem \ref{t: heat kernel}.
Also, by a similar argument we show in Corollary \ref{l: v less than or equal to 0} that Perelman's differential Harnack inequality $v\le 0$ holds.

\begin{subsection}{Construction of the heat kernel}\hfill\\
\begin{theorem}\label{t: heat kernel}
 Let $\M$ be a singular Ricci flow with normalized initial condition, and $x_0\in\M$ with $t_0:=\mathfrak{t}(x_0)>0$. Then there exists a function $u:\M_{t<t_0}\rightarrow\mathbb{R}$ which is a smooth solution to the conjugate heat equation, i.e. $(-\pt-\Delta+R)u=0$, and it satisfies the following properties
\begin{enumerate}
\item\label{a: vanishing} $u$ is a positive smooth function on $\M(x_0)$, and $u$ vanishes on $\M-\M(x_0)$. Recall $\M(x_0)$ is the subset of all points that are accessible to $x_0$.
\item\label{a: delta} $u$ is a delta-function at $x_0$, in the sense that $\lim_{t\nearrow t_0}\int_{\M_t} u(x)h(x)d_tx=h(x_0)$ for any smooth function $h$ on $\M$ that has a compact support.
\item\label{a: C_0} Suppose for some $r_0,T>0$ we have $|\Rm|\le r_0^{-2}$ on $\mathcal{P}_0:=P(x_0,r_0,-r_0^2)$, and $t_0<T$. Then there exists $C_0(r_0,T)>0$ such that $u\le C_0$ on $\M_{t<t_0}-\mathcal{P}_0$.
\end{enumerate}
\end{theorem}

For simplicity, we use the following variant of Perelman's Ricci flow with surgery, in which the surgeries are done slightly before singular times (the times where curvature blows up), instead of exactly at them. 
This Ricci flow with surgery can be obtained with little modification to that of Perelman's. 
It is also constructed for complete manifold with bounded geometry in \cite{Besson}.

\begin{defn}[Ricci flow with surgery]\label{d: RFS}
A Ricci flow with surgery is given by
\begin{enumerate}
    \item A collection of Ricci flows $\{(M_k\times[t_k^-,t_k^+],g_k(\cdot))\}_{1\le k\le N}$, where $N\le\infty$, $M_k$ is a compact (possibly empty) manifold, $t_k^+=t^-_{k+1}$ for all $1\le k\le N$. 
    \item A collection of isometric embeddings $\{\psi_k: Y^+_k\rightarrow Y_{k+1}^-\}_{1\le k\le N}$ where $Y_k^+\subset M_k$ and $Y_{k+1}^-\subset M_{k+1}$, $1\le k<N$, are compact 3 dimensional submanifold with boundary. The $Y_k^{\pm}$'s are the subsets which survive the transition from one flow to the next, and the $\psi_k$'s give the identification between them.
\end{enumerate}

\end{defn}

We call each final time $t_k^{+}$ a surgery time.
Let $X_k^+$ and $X_{k+1}^-$ denote the interior of $Y_k^+$ and $Y_{k+1}^-$ respectively for $1\le k<N$. We can associate a Ricci flow spacetime $\M$ to the Ricci flow with surgery by taking $\M$ to be the disjoint union of 
\begin{equation}
    (M_k\times[t_k^-,t_k^+))\cup(X_k^+\times\{t_k^+\})
\end{equation}
for $1\le k\le N$, and removing the following subset  
\begin{equation}
   (M_{k+1}-X_{k+1}^-)\times\{t_{k+1}^-\}
\end{equation}
for all $1\le k<N$ (making identifications using the $\psi_k$'s as gluing maps).

\begin{proof}[Proof of Theorem \ref{t: heat kernel}]
Let $\M_i$ be a sequence Ricci flow with surgery spacetimes starting from $(M,g)$, with the surgery scale $\delta_i\rightarrow 0$ as $i\rightarrow\infty$. 
Then $\M_i$ have the local control required for the application of the spacetime compactness theorem in \cite[Theorem 2.20]{KL1}, and hence $\M_i$ converges to the singular Ricci flow $\M$ as $i\rightarrow\infty$. 
Assume $x_{0i}\in\M_i$ converges to $x_0$ when $i\rightarrow\infty$.
We shall construct smooth non-negative solutions to the conjugate heat equation on each $\M_i$ starting from $x_{0i}$, and then take a limit of them to obtain the desired heat kernel on $\M$. 

Let $\M_i$ be fixed below, and assume the compact Ricci flows that form $\M_i$ are $\{(M_k\times[t_k^-,t_k^+],g_k(\cdot))\}_{1\le k\le N}$, where $t_N^+=t_{0i}=\mathfrak{t}(x_{0i})$.
We shall define $u_i$ on each of these Ricci flows and then restrict it to $\M_i$ to get a smooth function.

First, on $M_N\times[t_N^-,t_N^+)$, let $u_i$ be the ordinary heat kernel of $(M_N,g_N(t))$ which starts from $(x_{0i},t_{0i})$.
Note that $u_i$ vanishes at $(x,t)$ if $x$ are not in the same component with $x_{0i}$.
Then suppose by induction that $u_i$ has been defined on $M_j\times[t_j^-,t_j^+)$ for all $j=k, ...,N$ such that the following holds:
\begin{enumerate}
    \item $u_i\ge 0$ is a smooth solution to the conjugate heat equation on $\M_{i, \,t\ge t_k^-}$.
    \item $u_i$ vanishes at points that are not accessible to $x_{0i}\in\M_i$.
    \item $\int_{M_{j}}u_i(x,t)d_tx\le 1$ for all $t\in[t_j^-,t_j^+)$, and $j=k,...,N$.
\end{enumerate}

Then we define $u_i$ on $M_{k-1}\times \{t_{k-1}^+\}$ by letting $u_i(x,t_{k-1}^+)=u_i(x,t_k^-)$ if $x\in X_{k-1}^+$, and $u_i(x,t_{k-1}^+)=0$ if $x\in M_{k-1}-X_{k-1}^+$.
Then for any $(x,t)\in M_{k-1}\times[t_{k-1}^-,t_{k-1}^+)$,
set 
\begin{equation}\label{e: def}
 u_{i}(x,t)=\int_{X_{k-1}^+} u_i(y,t_{k-1}^+)H(y,t_{k-1}^+;x,t)\,d_{t_{k-1}^+}y,    
\end{equation}
where $H(\cdot,\cdot;\cdot,\cdot)$ is the ordinary heat kernel for the smooth compact Ricci flow $(M_{k-1},g_{k-1}(t))$, $t\in [t_{k-1}^-,t_{k-1}^+]$.
So $u_{i}$ is a smooth solution to the conjugate heat equation on $M_{k-1}\times[t_{k-1}^-,t_{k-1}^+)\cup X_{k-1}^{+}\times\{t_{k-1}^{+}\}$, and hence it is smooth on $\M_{i,\,t\ge t_{k-1}^-}$. 
Moreover, by \eqref{e: def} we get for all $t\in[t_{k-1}^-,t_{k-1}^+)$ that
\begin{equation}
    \begin{split}
        \int_{M_{k-1}} u_i(x,t)d_tx&\le \int_{M_{k-1}}\int_{X_{k-1}^+} u_i(y,t_{k-1}^+)H(y,t_{k-1}^+;x,t)\,d_{t_{k-1}^+}y\, d_tx\\
        &= \int_{X_{k-1}^+} u_i(y,t_{k-1}^+)\,d_{t_{k-1}^+}y\le 1.
    \end{split}
\end{equation}
It is clear that assumptions (1) and (2) also hold for $j=k-1$.
So by induction, we obtain a smooth non-negative solution $u_i$ on $\M_i$ which satisfies $\int_{\M_i(t)}u_i(x,t)d_tx\le 1$, and $u_i\equiv0$ on $\M_i-\M_i(x_{0i})$.

Next, suppose $r_0,T$ are the constants in assertion \eqref{a: C_0}. Then since $(\M_i,x_{0i})$ converges to $(\M,x_0)$, we have $|\Rm|\le 2r_0^{-2}$ on $\mathcal{P}_{0i}:=P(x_{0i},r_0,-r_0^2)$ for all large $i$.
Since $\M_i$ has the normalized initial condition, the scalar curvature satisfies $R\ge -6$ anywhere on $\M_i$, see \cite[Lemma 79.11]{KL}. So $\tilde{u}_i:=u_i e^{-6(t_0-t)}$ satisfies 
\begin{equation}
    (-\pt-\Delta)\tilde{u}_i\le0.
\end{equation}
Let $\Gamma$ be the parabolic boundary of $\mathcal{P}_{0i}$, i.e.
\begin{equation}
    \Gamma=\left(\bigcup_{t\in[t_{0i}-r_0^2,t_{0i})}\partial B_{t_{0i}}(x_{0i},r_0)(t)\right) \cup B_{t_{0i}}(x_{0i},r_0)(t_{0i}-r_0^2).
\end{equation}
Then since $\int u_i \,d_tx\le 1$, we can apply the parabolic mean value inequality (see e.g. \cite[Theorem 25.2]{RFTandA3}) at any points in $\Gamma$, and hence find a constant $C_1(r_0)>0$ such that $\tilde{u}_i\le C_1$ on $\Gamma$. 

Suppose by induction that $\tilde{u}_i\le C_1$ on $\M_{i,[t_k^-,t_0)}-\mathcal{P}_{0i}$ for some $k\le N$.
Without loss of generality, we may assume that $t_{0i}-r_0^2$ is a surgery time.
Then if $M_{k-1}\times[t_{k-1}^-,t_{k-1}^+)$ does not intersect $\mathcal{P}_{0i}$, by the inductive assumption we can apply the maximum principle and get $\tilde{u}_i\le C_1$ on $\M_{i,[t_{k-1}^-,t_{k-1}^+)}$.
Otherwise, apply the maximum principle on $\bigcup_{t\in[t_{k-1}^-,t_{k-1}^+)}(M_k(t)\setminus B_{t_{0i}}(x_{0i},r_0)(t))$, whose boundary is contained in $\Gamma$. Then by inductive assumption and $\tilde{u}_i\le C_1$ on $\Gamma$, we get $\tilde{u}_i\le C_1$ on $\M_{i,[t_{k-1}^-,t_{k-1}^+)}$.
So by induction, $\tilde{u}_i\le C_1$ holds on $\M_{i,t<t_{0i}}-\mathcal{P}_{0i}$.
Let $C_0=C_1e^{6T}$, then $u_i\le C_0$ on $\M_{i,t<t_{0i}}-\mathcal{P}_{0i}$.

Then applying the interior H\"{o}lder estimate, we can bound the derivatives of $u_i$ in terms of its $C^0$-norm and the curvature norm nearby. So by passing to a subsequence, we may assume that $u_i$ converges smoothly to a  non-negative smooth solution $u$ to the conjugate heat equation on $\M$. It follows immediately that 
\begin{equation}\label{e: u less than C_0}
    u(x)\le C_0
\end{equation}
for all $x\in \M-\mathcal{P}_{0}$, which proves assertion \eqref{a: C_0} in the theorem.


Now we establish the properties claimed in Theorem \ref{t: heat kernel}. 
First, we show that $u$ is a $\delta$-function at $x_0$.
For large $i$, let $\phi_i$ be a cut-off function whose support is contained in $B_{t_{0i}}(x_{0i},r_0)$, $\phi_i\equiv 1$ on $B_{t_{0i}}(x_{0i},\frac{r_0}{2})$, and $|\nabla\phi_i|$ and $|\Delta\phi_i|$ are bounded above in terms of $r_0$. 
Here the derivatives and norms are considered with respect to $g(t_{0i})$.
By the choice of $r_0$, there exists a universal constant $c>0$ such that $\frac{1}{2}g(t)\le g(t_{0i})\le 2g(t)$ for all $t\in[t_{0i}-cr_0^2,t_{0i}]$ on $B_{t_{0i}}(x_{0i},r_0)$.
So for all $t\in[t_{0i}-cr_0^2,t_{0i})$, a direct computation using integration by parts shows that there is a constant $C=C(r_0)>0$ such that
\begin{equation}
\begin{split}
    \pt\int_{B_{t_{0i}}(x_{0i},r_0)} u_i\phi_i\,d_tx
    =\int_{B_{t_{0i}}(x_{0i},r_0)} -u_i\Delta_{g(t)}\phi_i\,d_tx
    \le |\Delta_{g(t)}\phi_i|
    \le C,
\end{split}
\end{equation}
where we also used $\int_{\M_{i,t}} u_i(x,t)\,d_tx\le 1$. 
Integrating this we get for all $t\in[t_{0i}-cr_0^2,t_{0i})$ that $\int_{B_{t_{0i}}(x_{0i},r_0)} u_i(x,t)\,d_tx\ge 1-C(t_{0i}-t)$.
Passing to the limit this implies
\begin{equation}\label{e: delta function_1}
    \int_{(B_{t_{0}}(x_{0},r_0))(t)} u(x)\,d_tx\ge 1-C(t_0-t).
\end{equation}
Letting $t\nearrow t_0$, we get
\begin{equation}\label{e: delta function}
    \lim_{t\nearrow t_0}\int_{\M_t} u(x)\,d_tx=1.
\end{equation}

By a same argument we can show that
$\lim_{t\nearrow t_0}\int_{\M_t} u(x)h(x) \,d_tx=h(x_0)$,
for all smooth function $h$ that has compact support. So $u$ is a $\delta$-function at $x_0$, which verifies assertion \eqref{a: delta} in Theorem \ref{t: heat kernel}.

Now we verify assertion \eqref{a: vanishing} in Theorem \ref{t: heat kernel}. First, the positivity of $u$ on $\M(x_0)$ is an easy consequence of the Harnack inequality for parabolic equations. 
So it remains to show that $u$ vanishes on $\M-\M(x_0)$.

To show this, let $y\in\M-\M(x_0)$, $\mathfrak{t}(y)=t\in[0,t_0)$.
Since by \cite[Theorem 6.3]{KL2} the non-singular times, at which the time-slices have bounded curvature, are dense, we may assume without loss of generality that $t$ is a non-singular time.
Otherwise, we can find a sequence of non-singular times $s_k>t$ that converges to $t$ as $k\rightarrow\infty$, and $y(s_k)\in\M-\M(x_0)$. 
So $\M_t$ has finitely many connected components which are closed manifolds.
Since $y\in\M-\M(x_0)$, it is easy to see that $x_0(t)$ and $y$ are in different connected components in $\M_t$.

Then by \cite[Theorem 1.3]{KL1}, there is a sequence of time-preserving diffeomorphisms $\{\phi_i:U_i\rightarrow\M_i\}$, where $U_i$ are open subsets of $\M$ such that given any $\overline{t},\overline{R}>0$, if $i$ is sufficiently large then 
\begin{equation}
    U_i\supset\{x\in\M: \mathfrak{t}(x)\le\overline{t}\;\textnormal{and}\;R(x)\le\overline{R}\},
\end{equation}
and $\{\phi_i^*g_i\}$ converges smoothly on compact subsets of $\M$ to $g$.
So $\M_t$ is contained in $U_i$ for all large $i$ and $\phi_i(\M_t)$ is a finite union of closed manifolds.
In particular, $\phi_i(x_0(t))$ and $\phi_i(y)$ are in different connected components in $\phi_i(\M_t)$, which implies
$\phi_i(y)\in\M_i-\M(x_{0i})$.
So $u_i$ vanishes at $\phi_i(y)$.
Letting $i\rightarrow\infty$ we get $u(y)=0$.
\end{proof}

\end{subsection}

\begin{subsection}{Further properties of the heat kernel}\hfill\\
In this subsection we investigate more properties of the heat kernel in Theorem \ref{t: heat kernel}.
The first main result is Theorem \ref{l: semilocal}, which is
a semi-local maximum principle for the heat kernel.
Then in  
Theorem \ref{e: uR^m less than C_m} we derive from Theorem \ref{l: semilocal} a polynomial decay estimate of the heat kernel.
Corollary \ref{l: integral equal to 1} and \ref{l: v less than or equal to 0} are applications of Theorem \ref{e: uR^m less than C_m}.

The main ingredient in the proof of the semi-local maximum principle is the following vanishing theorem of the Bryant soliton:
\begin{prop}\label{t: vanishing}
(Vanishing theorem on Bryant soliton) Let $(M,g(t)),t\in\mathbb{R}$ be a Bryant soliton with tip $x_0\in M$, $R(x_0,0)=1$, and $u(x,t):M\times[0,\infty)\rightarrow\mathbb{R}$ be a smooth non-negative solution to the conjugate heat equation. Suppose there are constants $C>0$ and $m\in\mathbb{N}_+$ such that $u(x,t)R^m(x,t)\le C$ for all $x\in M$ and $t\in[0,\infty)$. 

Then $u\equiv0$.
\end{prop}
 
\begin{proof}
Let $C$ denote all the constants that depend only on the constants $C$ and $m$ in the assumption.
Without loss of generality, it suffices to show $u(x,0)=0$ for all $x\in M$. 

Let $H(y,t;z,s)$, $t>s$ be the heat kernel of $(M,g(t))$. Then the following holds for all $t>0$ 
\begin{equation}
    \begin{split}
        u(x,0)&=\int H(y,t;x,0)u(y,t)\,d_ty\\
              &=\int_{d_t(y,x_0)\le 1} H(y,t;x,0)u(y,t)\,d_ty +\int_{d_t(y,x_0)> 1}H(y,t;x,0)u(y,t)\,d_ty\\
              &=\mathcal{I}(t)+\mathcal{J}(t).
    \end{split}
\end{equation}

We shall show that $\mathcal{I}(t)$ and $\mathcal{J}(t)$ converge to $0$ as $t\rightarrow\infty$.
First, to estimate $\mathcal{I}(t)$ we note that $d_t(y,x_0)\le 1$ implies $R(y,t)> C^{-1}$.
So it follows from the assumption  
$uR^m\le C$ and the Bishop-Gromov volume comparison that
\begin{equation}
    \begin{split}
        \mathcal{I}(t)&\le C\int_{d_t(y,x_0)\le 1}H(y,t;x,0)\,d_ty
        \le C \sup_{d_t(y,x_0)\le 1} H(y,t;x,0).
    \end{split}
\end{equation}
By the Gaussian bound for heat kernel of a Ricci flow with bounded curvature \cite[Corollary 26.26]{RFTandA3}, we have $\lim _{t\rightarrow\infty} H(y,t;x,0)=0$ uniformly for all $y\in M$.
So $\mathcal{I}(t)\rightarrow 0$, as $t\rightarrow\infty$.

Next we estimate $\mathcal{J}(t)$. Since $d_t(y,x_0)\ge 1$, it follows that $C^{-1}d_t(y,x_0)^{-1}\le R(y,t)\le Cd_{t}(y,x_0)^{-1}$, which together with $u R^m\le C$ implies  
\begin{equation}\label{e: J}
    \begin{split}
        \mathcal{J}(t)&\le C \int_{d_t(y,x_0)> 1}H(y,t;x,0)d^m_t(y,x_0)\,d_ty.
    \end{split}
\end{equation}
We claim that the following estimate holds for all $(y,t)$ such that $t\ge C$ and $d_t(y,x_0)\ge 1$:
\begin{equation}\label{heat kernel bound}
    H(y,t;x,0)\le C\,\textnormal{exp}\left(-\frac{d_0^2(x,y)}{Ct}\right).
\end{equation}

Assume for a moment that the claim is true, and we use it to prove the proposition. 
For any $(y,t)$ such that $t\ge C$ and $d_t(y,x_0)\ge 1$, we have $d_s(y,x_0)\ge 1$ for all $s\in[0,t]$ because of $\Ric\ge0$.
Since $R\ge C^{-1}$ on $B_s(x_0,1)$ for any $s\in\mathbb{R}$, by a distance distortion estimate we get $d_s(x_0,y)\ge C^{-1}(t-s)$ for all $s\in[0,t]$.
So when $t\ge C$ we have
\begin{equation}\label{e: x}
    \begin{split}
        d_0(x,y)&\ge d_0(x_0,y)-d_0(x,x_0)\ge d_0(x_0,y)(1-\frac{d_0(x,x_0)}{C^{-1}t})\ge\frac{1}{2}d_0(x_0,y),
    \end{split}
\end{equation}
substituting which into \eqref{heat kernel bound} we get
\begin{equation}
    H(y,t;x,0)\le C\,\textnormal{exp}\left(-\frac{d_0^2(x_0,y)}{Ct}\right).
\end{equation}
Putting this into \eqref{e: J} and using the Bishop-Gromov volume comparison theorem we get
\begin{equation}
    \begin{split}
        \mathcal{J}(t)&\le\int_{d_0(y,x_0)>C^{-1}t}C\,\textnormal{exp}\left(-\frac{d_0^2(x_0,y)}{Ct}\right)d^m_0(x_0,y)\;\;d_0y
        \le Ce^{-\frac{t}{C}}.
    \end{split}
\end{equation}
Hence $\mathcal{J}(t)\rightarrow 0$ as $t\rightarrow\infty$. 
Therefore, by letting $t\rightarrow\infty$, we obtain $u(x,0)=0$.

Now we establish \eqref{heat kernel bound} to finish the proof.
Fix a pair $(y,t)$ such that $t\ge C$ and $d_t(y,x_0)\ge 1$. 
The value of $C$ will be determined later.
For any $s\in[0,1]$ and $z\in B_s(y,1)$, let $\ell(z,s)$ be the reduced length from $(z,s)$ to $(y,t)$, and let $\gamma:[s,t]\rightarrow M$ be a curve such that $\gamma|_{[2,t]}\equiv y$ and $\gamma|_{[s,2]}$ is a minimal geodesic connecting $y$ and $z$ with respect to $g(0)$. 

For $\tau\in[0,t-s]$, we have
$R(y,t-\tau)\le C d_{t-\tau}^{-1}(y,x_0)$ and $d_{t-\tau}(y,x_0)\ge C^{-1}\tau$, and hence $R(y,t-\tau)\le\frac{C}{\tau}$.
For $\tau\in[t-2,t-s]$,
we have $R(\gamma(t-\tau),t-\tau)\le C$ and $|\gamma'|(t-\tau)\le C$.
Putting these together we can estimate the $\LL$-length of $\gamma$:
\begin{equation}
    \begin{split}
        \mathcal{L}(\gamma)&=\int_{0}^{t-2}\sqrt{\tau}R(y,\tau)\,d\tau+
        \int_{t-2}^{t-s}\sqrt{\tau}(R(\gamma(t-\tau),t-\tau)+|\gamma'|^2)\,d\tau\\
        &\le \int_{0}^{t-2}\sqrt{\tau}\frac{C}{\tau}\,d\tau+\int_{t-2}^{t-s}C\sqrt{\tau}\,d\tau\le C\sqrt{t},
    \end{split}
\end{equation}
and hence
\begin{equation}
    \ell(z,s)=\frac{\mathcal{L}(z,s)}{2\sqrt{t-s}}\le \frac{\mathcal{L}(\gamma)}{2\sqrt{t-s}}\le C.
\end{equation}
Recall the heat kernel lower bound by Perelman in
\cite[Corollary 9.5]{Pel1}
we get:
\begin{equation}\label{reduced length}
    H(y,t;z,s)\ge \frac{1}{4\pi(t-s)^{3/2}}e^{-\ell(z,s)}\ge\frac{C}{t^{\frac{3}{2}}},
\end{equation}
for all $s\in[0,1]$ and $z\in B_s(y,1)$. So by the Bishop-Gromov volume comparison we get
\begin{equation}\label{low}
    \int_{B_s(y,1)}H(y,t;z,s)\,d_sz\ge \frac{C}{t^{3/2}}\ge\frac{C}{d_0(x,y)^{3/2}},
\end{equation}
for all $s\in [0,1]$.

Note by the multiplication inequality for heat kernel in \cite{HN} we have 
\begin{equation}\label{e: hein-naber}
     \left(\int_{B_s(y,1)}H(y,t;z,s)\,d_sz\right)\left(\int_{B_s(x,1)}H(y,t;z,s)\,d_sz\right)\le C\, \textnormal{exp}\left(-\frac{(d_s(x,y)-2)^2}{4C(t-s)}\right).
\end{equation}
So by substituting \eqref{low} into \eqref{e: hein-naber} we get
\begin{equation}
    \left(\int_{B_s(x,1)}H(y,t;z,s)\,d_sz\right)\le C\,d_0(x,y)^\frac{3}{2}\, \textnormal{exp}\left(-\frac{d_0(x,y)^2}{4C(t-s)}\right)\le
    C\, \textnormal{exp}\left(-\frac{d_0(x,y)^2}{4C(t-s)}\right),
\end{equation}
where we also used $d_s(x,y)-2\ge C^{-1}d_0(x,y)-2\ge C^{-1}d_0(x,y)$ for $t$ very large. 
Integrating this for all $s\in[0,1]$, and then applying the parabolic mean value inequality (see e.g. \cite{RFTandA3}) to $H(y,t;\cdot,\cdot)$ at $(x,0)$, we obtain 
\begin{equation}
    H(y,t;x,0)\le C\, \textnormal{exp}\left(-\frac{d^2_0(x,y)}{4Ct}\right),
\end{equation}
which confirms claim \eqref{heat kernel bound} and hence completes the proof.
\end{proof}

\begin{prop}{(A semi-local Maximum Principle)}\label{l: semilocal}
Given $r_0,T>0$, $r_0^2<T$ and $m\in\mathbb{N}$, there exist $\epsilon=\epsilon(r_0,T,m)>0$ and $r=r(r_0,T,m)>0$ such that for any $t_0\in(r_0^2,T)$ the following holds: 

Let $\M$ be a singular Ricci flow with normalized initial condition, $x_0\in \M$, $\mathfrak{t}(x_0)=t_0>0$.
Suppose $|\Rm|\le r_0^{-2}$ on $\mathcal{P}_0:=P(x_0,r_0,-r_0^2)$.
Let $u$ be the heat kernel in Theorem \ref{t: heat kernel} which starts from $x_0$.
Then for any $x\in \M_{t<t_0}-\mathcal{P}_0$ with $R(x)>r^{-2}$, there exists $y\in\M_{t<t_0}-\mathcal{P}_0$ with $\mathfrak{t}(y)\ge \mathfrak{t}(x)$ such that
\begin{equation}
  \begin{cases}
  u(y)\ge(1+\epsilon)u(x)\quad and\\
  (uR^m)(y)\ge (1+\epsilon)(uR^m)(x).
  \end{cases}
\end{equation}
\end{prop}

\begin{proof}
Suppose the conclusion is not true, then there are sequences $\{\epsilon_k\}$ and $\{r_k\}$ both converging to zero, and a sequence of Ricci flow spacetime $(\M_k,g_k(t))$, along with the heat kernels $u_k$ starting from $x_{0k}\in\M_k$, $t_{0k}=\mathfrak{t}(x_{0k})$, which contradict the lemma at points $x_{k}\in \M_{k,t<t_{0k}}-\mathcal{P}_{0k}$, $\mathfrak{t}(x_k)=t_k$. 
This means $\rho(x_k)<r_k$, and for any $y\in \M_{k,t<t_{0k}}-\mathcal{P}_{0k}$ with $\mathfrak{t}(y)\ge t_k$, we have either
\begin{equation}\label{e: 1}
    u_k(y)<(1+\epsilon_k)u_k(x_k),
\end{equation}
or 
\begin{equation}\label{e: 2}
    (u_kR^m)(y)<(1+\epsilon_k)(u_kR^m)(x_k).
\end{equation}
This implies $u(x_k)>0$, and by item \eqref{a: vanishing} in Theorem \ref{t: heat kernel} we get $x_k\in\M_k(x_{0k})$.

Let $\rho_0=\min\{\frac{1}{2}r_0,r(\epsilon_{can},T+1)\}$, where $r(\cdot,\cdot)$ is the canonical neighborhood scale function for singular Ricci flow with normalized initial condition in Definition \ref{d: SRF}, and $\epsilon_{can}$ is sufficiently small.
Let $P_k=P(x_k,\frac{\rho_0}{4\eta},\frac{\rho_0^2}{4\eta})\subset\M_k$, where $\eta>0$ is from Lemma \ref{l: derivative}.
Then by Lemma \ref{l: derivative} we have $\rho<r_0$ on $P_k$ for all large $k$. Seeing that $\rho\ge r_0$ on $\mathcal{P}_{0k}$, it implies that $P_k\subset\M_{k,t<t_{0k}}-\mathcal{P}_{0k}$.
Now rescaling the spacetimes in $P_k$ by $R(x_k)$ and shifting the times $t_k$ to $0$, we get a sequence of Ricci flow spacetimes $(\tilde{P}_k,\tilde{g}_k(s)_{s\ge0},(x_k,0))$, where $\tilde{g}_k(s)$ denotes the horizontal Riemannian metric.
Since $r_k\rightarrow 0$ as $k\rightarrow\infty$, by Lemma \ref{l: blow-up converges to a k-solution} we may assume by passing to a subsequence that $(\tilde{P}_k,\tilde{g}_k(s)_{s\ge0},(x_k,0))$ converges to a $\kappa$-solution $(M_{\infty},g_{\infty}(s)_{s\ge0},(x_{\infty},0))$.

Let $\tilde{u}_k(y)=\frac{u_k(y)}{u_k(x_k)}$ for all $y\in \tilde{P}_k$.
Then for all $y\in \tilde{P}_k$ we have either
\begin{equation}\label{1}
    \tilde{u}_k(y)<1+\epsilon_k,
\end{equation}
or
\begin{equation}\label{2}
    (\tilde{u}_kR^m)(y)<1+\epsilon_k.
\end{equation}
Then since $R>0$ on $(M_{\infty},g_{\infty}(s),x_{\infty})$, we deduce that $\tilde{u}_k$ has locally bounded $C^0$-norm.
By H\"{o}lder estimate this implies that the $C^k$-norm of $\tilde{u}_k$ is locally bounded bound for any $k\in\mathbb{N}$. So by passing to a subsequence we may assume that $\tilde{u}_k$ converges smoothly to a smooth non-negative solution $\tilde{u}$ to the conjugate heat equation of the flow $(M_{\infty},g_{\infty}(s)_{s\ge0},x_{\infty})$, and $\tilde{u}(x_{\infty},0)=1$. 
Since $\epsilon_k\rightarrow0$ as $k\rightarrow\infty$, we have that one of the following holds for all $y\in M_{\infty}$ and $s\ge 0$, 
\begin{equation}\label{e: 1'}
    \tilde{u}(y,s)\le 1,
\end{equation}
or
\begin{equation}\label{e: 2'}
    (\tilde{u}R^m)(y,s)\le 1.
\end{equation}

We claim that $(M_{\infty},g_{\infty}(s))$ is either a cylindrical solution (the standard solution on $S^2\times\mathbb{R}$, or its quotient by the map that is a reflection on $\mathbb{R}$ and an antipodal map on $S^2$), or the Bryant soliton. Using the classification result of non-compact $\kappa$-solutions \cite{classification}, it suffices to show that $M_{\infty}$ is not compact. 

Suppose this is not true.
On the one hand, by the compactness of $M_{\infty}$, for large $k$ there exists a diffeomorphism $\phi_k:M_{\infty}\rightarrow U_k$ such that $\phi_k(x_{\infty})=x_k$, where $U_k=\phi_k(M_{\infty})$ is a connected component in $\M_{k,\,t_k}$ and $x_k\in U_k$.
Also, for any given $\delta>0$, the following holds:
\begin{equation}\label{e: closeness}
    \|r^{-2}_k\phi_k^*(g_k(t_k))-g_{\infty}(0)\|_{C^{[\delta^{-1}]}(M_{\infty},g_{\infty}(0))}\le\delta.
\end{equation}

On the other hand, since $u_k(x_{k})>0$ and by item \ref{a: vanishing} in Theorem \ref{t: heat kernel} we see that $x_k\in\M_k(x_{0k})$.
By the component stability theorem, \cite[Proposition 5.32]{KL1}, for any $t<t_{0k}$, the time-t-slice of $\M_k(x_{0k})$ is the connected component of $\M_{k,t}$ that contains $x_{0k}(t)$.
So we deduce that $U_k$ is equal to $\M_k(x_{0k})(t_k)$, which is the time-$t_k$-slice of $\M_k(x_{0k})$.

Since $\inf_{M_{\infty}}R(\cdot,0)\ge c$ for some $c>0$, by \eqref{e: closeness} we get
\begin{equation}\label{e: R}
    \inf_{\M_k(x_{0k})(t_k)} R \ge \frac{1}{2}cr_k^{-2},
\end{equation}
for all large $k$.
Then by the maximum principle for scalar curvature we get
\begin{equation}
    R(x_{0k})\ge\inf_{\M_k(x_{0k})(t\ge t_k) }R \ge \frac{1}{2}cr_k^{-2}.
\end{equation}
For sufficiently large $k$, this contradicts the assumption $R(x_{0k})\le r_0^{-2}$. So $M_{\infty}$ must be non-compact.

So first we suppose $(M_{\infty},g_{\infty}(s)),s\in[0,\infty)$ is the Bryant soliton.  
Since the curvature is uniformly bounded everywhere, if $\tilde{u}(y,s)\le 1$ for some $(y,s)\in M_{\infty}\times[0,\infty)$, then 
\begin{equation}\label{e: 2''}
    \tilde{u}(y,s)R^m(y,s)\le C,
\end{equation}
where $C$ depends only on the curvature at the tip of $(\M_{\infty},g_{\infty}(0))$.
Combining with $\eqref{e: 2'}$, we see that \eqref{e: 2''} holds at all $(y,s)\in M_{\infty}\times[0,\infty)$.
By the vanishing theorem, Proposition \ref{t: vanishing}, we get $\tilde{u}(x_{\infty},0)=0$, contradiction.

Next, suppose  $(M_{\infty},g_{\infty}(s))$, is a cylindrical solution with $R(x_{\infty},0)=1$.
Then the flow exists on $[0,\frac{3}{2})$,
and $R(y,s)\ge 1$ for all $y\in M_{\infty}$ and $s\in[0,\frac{3}{2})$.
So \eqref{e: 2'} implies \eqref{e: 1'}, and hence $\tilde{u}(y,s)\le 1$ for all $(y,s)\in M_{\infty}\times[0,\frac{3}{2})$. Noting $\tilde{u}(x_{\infty},0)=1$, we can apply the maximum principle at $(x_{\infty},0)$ and get 
\begin{equation}
    (-\pt-\Delta)\tilde{u}\ge0,\quad\textnormal{at}\quad(x_{\infty},0).
\end{equation}
This is impossible seeing that $(-\pt-\Delta+R)\tilde{u}=0$ and $\tilde{u}(x_{\infty},0)R(x_{\infty},0)>0$.
\end{proof}

\begin{theorem}
\label{e: uR^m less than C_m}
Under the same assumption as in Proposition \ref{l: semilocal}, 
there exists $C_m=C(r_0,T,m)>0$ such that the following holds for all $x\in\M_{t<t_0}-\mathcal{P}_0$:
\begin{equation}
    u(x)R^m(x)\le C_m.
\end{equation}
\end{theorem}

\begin{proof}
Let $C_m=C_0 r^{-2m}$, where $C_0=C_0(r_0,T)$ is from item \eqref{a: C_0} in Theorem \ref{t: heat kernel}, and $r=r(r_0,T,m)>0$ is from Theorem \ref{l: semilocal}.
Then it's clear that $uR^m\le C_m$ for all the points in $\M_{t<t_0}-\mathcal{P}_0$ that satisfy $R\le r^{-2}$.
We shall show that $uR^m\le C_m$ holds everywhere on $\M_{t<t_0}-\mathcal{P}_0$.
Suppose by contradiction that this is not true. 
Then there exists $x_1\in\M_{t<t_0}-\mathcal{P}_0$ such that $uR^m(x_1)>C_m$ and $R(x_1)>r^{-2}$. 

Suppose by induction that there are $\{x_k\}\subset\M_{t<t_0}-\mathcal{P}_0$, $t_k=\mathfrak{t}(x_k)$, $k=1,2,...,N$, such that $t_k\ge t_{k-1}$, and the following holds for all $k$:
\begin{equation}\label{two}
    \begin{cases}
    uR^m(x_k)\ge (1+\epsilon)uR^m(x_{k-1}),\quad\textnormal{and}\\
    u(x_k)\ge (1+\epsilon)u(x_{k-1}),
    \end{cases}
\end{equation}
where $\epsilon=\epsilon(r_0,T,m)>0$ is from Theorem \ref{l: semilocal}.
Since $uR^m(x_N)\ge uR^m(x_1)> C_m$, it follows from the definition of $C_m$ that $R(x_N)>r^{-2}$. This allows us to apply Proposition \ref{l: semilocal} and get a point $x_{N+1}\in \M_{t<t_0}-\mathcal{P}_0$, $\mathfrak{t}(x_{N+1})=t_{N+1}\ge t_N$ which together with $x_N$ satisfies \eqref{two}. 
So by induction we get an infinite sequence $\{x_k\}_{k=1}^{\infty}$ which satisfies \eqref{two}. 
Then we can deduce from the second inequality in \eqref{two} that $u(x_k)\rightarrow\infty$ as $k\rightarrow\infty$, which contradicts item \eqref{a: C_0} in Theorem \ref{t: heat kernel}.
\end{proof}

\begin{cor}\label{l: integral equal to 1}
Under the same assumption as in Proposition \ref{l: semilocal}. Then
\begin{equation}\label{e: integral equal to 1}
    \int_{\M_t} u\,d_tx=1,
\end{equation}
for all $t\in[0,t_0)$.
\end{cor}

\begin{proof}
Without loss of generality, we may assume $t=0$.
First, we fix some small $\delta_{\#}>0$ and $\epsilon<\epsilon_{can}(\delta_{\#})$ from Lemma \ref{l: either neck or cap}.
Let $\eta$ be from Lemma \ref{l: derivative}.
Let $m\in\mathbb{N}$ be greater than $1$. We use $C$ to denote all the constants depending on $\delta_{\#},r_0,T$ and $m$.  

Let $\delta>0$, whose value will be determined in the course of the proof. Choose a division of $[0,t_0]$ by $0=t_1<t_2<...<t_N=t_0$, such that $t_{i+1}-t_i\le\delta^2$ for all $i=1,...,N-1$, and $N\le (t_0+1)\delta^{-2}\le (T+1)\delta^{-2}$.
By Lemma \ref{l: central sphere decomposition}, for sufficiently small  $\delta$, there is an open domain $\Omega\subset\M_{t_{2i+1}}$ such that the boundary components of $\Omega$ are finitely many central spheres in some $\delta_{\#}$-neck, whose number is bounded by $C$, and the area of $\partial\Omega$ is less than $C\delta^2$. Moreover, we have $\rho\ge 2\sqrt{\eta}\delta$ on $\Omega$, and $\rho\le C\delta$ on $\M_{t_{2i+1}}-\Omega$.

Then by Lemma \ref{l: derivative} we see that $\sqrt{\eta}\delta\le\rho(x(t))\le C\delta$ for all $x\in\Omega$ and $t\in[t_{2i},t_{2i+1}]$.
So $\Omega$ survives until time $t_{2i}$, and
Area$_t(\partial\Omega(t))\le C\delta^2$. 
Then applying Theorem \ref{e: uR^m less than C_m} at points in $\partial\Omega(t)$ and then using the interior H\"{o}lder estimate, we get $|\nabla u|\le C\delta^{2m-1}$ on $\partial\Omega(t)$.
Let $t\in[t_{2i},t_{2i+1}]$, then
\begin{equation}\label{e: differential u}
    \begin{split}
        \pt\int_{\Omega(t)} u(x)\,d_tx&=\int_{\Omega(t)} -\Delta u(x)\,d_tx
        =\int_{\partial\Omega(t)}\frac{\partial u}{\partial \vec{n}}\,d_tS\le C\cdot\delta^{2m-2},
    \end{split}
\end{equation}
where $\vec{n}$ is the inwards unit normal vector field on $\partial\Omega(t)$.
Integrating this on $[t_{2i},t_{2i+1}]$, we get
\begin{equation}\label{e: previous time}
    \int_{\Omega(t_{2i})} u(x)\,d_{t_{2i}}x
    \ge\int_{\Omega} u(x)\,d_{t_{2i+1}}x-C\delta^{2m}.
\end{equation}

Also, applying Theorem \ref{e: uR^m less than C_m} on $\M_{t_{2i+1}}-\Omega$ and
using $vol(\M_{t_{2i+1}})\le C$, we get
\begin{equation}
   \int_{\M_{t_{2i+1}}-\Omega} u(x)\,d_{t_{2i+1}}x 
   \le C\delta^{2m},
\end{equation}
which combining with \eqref{e: previous time} gives
\begin{equation}\label{e: estimate on singular interval}
    \begin{split}
        \int_{\M_{t_{2i}}} u(x)\,d_{t_{2i}}x
        \ge\int_{\M_{t_{2i+1}}}u(x)\,d_{t_{2i+1}}x-C\delta^{2m}.
    \end{split}
\end{equation}

Note $\lim_{t\nearrow t_0}u(x)\,d_tx=1$, by induction we have
\begin{equation}
    \int_{\M_0} u(x)\,d_0x\ge 1-CN\delta^{2m}\ge 1-C(T+1)\delta^{2m-2}.
\end{equation}
Letting $\delta\rightarrow 0$, the conclusion follows immediately.
\end{proof}

\begin{cor}\label{l: v less than or equal to 0}
  Under the same assumption as in Proposition \ref{l: semilocal}.
  Let $f$ be a smooth function on $\M(x_0)$ such that $u=(4\pi(t_0-t))^{-3/2}e^{-f}$. Then
\begin{equation}\label{e: v}
  v=[(t_0-t)(2\Delta f-|\nabla f|^2+\R)+f-n]u\le 0.
\end{equation}
\end{cor}

\begin{proof}
Suppose the conclusion does not hold. Then without loss of generality we may assume that there exists $x_1\in\M_0$ such that $v(x_1)>0$.
Let $h_0\ge 0$ be a smooth function on $\M_0$ which is supported in a neighborhood of $x_1$ in which $v>0$, and $h_0(x_1)>0$.
Then $\int_{\M_0}h_0v\,d_0x>0$.
In the same way we constructed $u$, we can find a smooth and bounded function $h\ge 0$ on $\M$ with $h(x)=h_0(x)$ for all $x\in\M_0$, which solves the heat equation $(\pt-\Delta)h=0$.

Since $(-\pt-\Delta+\R)v\le 0$, see e.g. \cite{KL}, for any open domain $\Omega\subset\M_t$ with smooth boundary, we have
\begin{equation}
    \begin{split}
        \pt\int_{\Omega} -hv\,d_tV&\le\int_{\partial\Omega}\left(\frac{\partial h}{\partial\vec{n}}v-\frac{\partial v}{\partial\vec{n}}h\right)\,d_tS.
    \end{split}
\end{equation}
Applying Theorem \ref{e: uR^m less than C_m} in the same way as Corollary \ref{l: integral equal to 1}, we get
\begin{equation}\label{e: the same argument}
    \int_{\M_t}hv\,d_tx\le\lim_{s\nearrow t_0}\int_{\M_s}hv\,d_sx
\end{equation}
for all $t\in[0,t_0)$.
It was shown in \cite{Ni2005} that $\int_{\M_t}hv\,d_tx$ approaches to zero as $t$ goes up to $t_0$. So \eqref{e: the same argument} implies $\int_{\M_0}hv\,d_0x\le0$, a contradiction. 

\end{proof}

\end{subsection}
\end{section}

\begin{section}{Pseudolocality theorem on singular Ricci flow}
\label{s: pseudolocality}

In this section, we generalize Perelman's pseudolocality theorem for compact Ricci flows to singular Ricci flows.
The main ingredient is the heat kernel in Section \ref{s: heat kernel}, especially Corollary \ref{l: integral equal to 1} and \ref{l: v less than or equal to 0}.

\begin{theorem}(Pseudolocality theorem)\label{t: pseudolocality}
For every $\alpha>0$ there exists $\delta,\epsilon>0$ with the following property. 
Let $(\M,g(t))$ be a singular Ricci flow and $x_0\in\M_{t_0}$ for some $t_0\ge 0$.
Suppose $R\ge-1$ on $B_{t_0}(x_0,2)$, and for any $\Omega\subset B_{t_0}(x_0,2)$ we have $vol(\partial\Omega)^3\ge(1-\delta)c_3vol(\Omega)^{2}$, where $c_3$ is the Euclidean isoperimetric constant at dimension 3. 
Then $\bigcup_{t\in [t_0,t_0+\epsilon^2]}B_t(x_0(t),\epsilon)$ is unscathed, and $|\Rm|(x)<\frac{\alpha}{\mathfrak{t}(x)-t_0}+\epsilon^{-2}$ holds for all $x\in\bigcup_{t\in [t_0,t_0+\epsilon^2]}B_t(x_0(t),\epsilon)$.
\end{theorem}

As the proof has a lot in common with that of Perelman's pseudolocality theorem, we will focus on the differences, especially the places where the generalized heat kernel comes into play, see \cite[Section 30-34]{KL} for details of the parts which we are brief about.

\begin{proof}
Without loss of generality, we assume $t_0=0$, and $\alpha<\frac{1}{300}$.
Suppose the assertion is not true. 
Then there are sequences $\epsilon_k\rightarrow 0$ and $\delta_k\rightarrow 0$, and pointed singular Ricci flows $(\M_k,(x_{0k},0),g_k(\cdot))$ which satisfy the hypotheses of the theorem but for which there is a point $x_k$ in the unscathed set $\bigcup_{t\in[0,\epsilon_k^2]}B_t(x_{0k}(t),\epsilon_k)$ with $|\Rm|(x_k)\ge\alpha t_k^{-1}+\epsilon_k^{-2}$.
By reducing $\epsilon_k$ if needed, we may also assume that
\begin{equation}\label{e: alpha-large}
    |\Rm|(x)<\alpha t_k^{-1}+2\epsilon_k^{-2},
\end{equation}
for all $x\in  \bigcup_{t\in[0,\epsilon_k^2]}B_t(x_{0k}(t),\epsilon_k)$. We abbreviate $d_t(x_{0k}(t),x)$ as $d(x,t)$.

Let $A_k=\frac{1}{300\epsilon_k}$. 
We say a point $y$ is an $\alpha$-large point if
$|\Rm|(y)\ge\frac{\alpha}{\mathfrak{t}(y)}$.
First, suppose
$\mathcal{P}_k:=\bigcup_{t\in[0,\epsilon_k^2]}B_t(x_{0k}(t),(2A_k+1)\epsilon_k)$ is unscathed. Then by a point-picking we can find an $\alpha$-large point $\ox_k\in\mathcal{P}_k$, $\mathfrak{t}(\ox_k)=\overline{t}_k$, such that 
\begin{equation}\label{e: parabolic region_1}
    |\Rm|(y)\le 4|\Rm|(\ox_k):= 4Q_k,
\end{equation}
holds for all $\alpha$-large points $y$, $\mathfrak{t}(y)=s$, with $s\in(0, \ot_k]$ and $d(y,s)\le d(\ox_k,\ot_k)+A_kQ_k^{-1/2}$.
By a distance distortion estimate we can show that \eqref{e: parabolic region_1} also holds on $P(\overline{x}_k,\frac{1}{10}A_kQ^{-1/2}_k,-\frac{1}{2}\alpha Q_k^{-1/2})$.

Next, suppose $\mathcal{P}_k$ is scathed.
They by Lemma \ref{l: blow-up converges to a k-solution}, we can also find an $\alpha$-large point $\ox_k\in\mathcal{P}_k$ so that for large $k$, \eqref{e: parabolic region_1} holds
on $P(\overline{x}_k,\frac{1}{10}A_kQ^{-1/2}_k,-\frac{1}{2}\alpha Q_k^{-1/2})$.

Now let $u_k=(4\pi (\ot_k-t))^{-n/2}e^{-f_k}$ be the heat kernel on $\M_k$ starting from $\ox_k$, and $v_k$ be defined by \eqref{e: v}.
Then $v_k\le0$. The following lemma says that a local integral of $v_k$ has a negative upper bound at some time earlier than $\ot_k$.

\begin{lem}\label{l: beta}(\cite[Lemma 33.4]{KL})
There is some $\beta>0$ so that for all sufficiently large $k$, there is some $\tilde{t}_k\in[\ot_k-\frac{1}{2}\alpha Q_k^{-1},\ot_k)$ with $\int_{B_k} v_k\,dV_k\le -\beta$, where $B_k$ is the time-$\tilde{t}_k$ ball of radius $\sqrt{\ot_k-\tilde{t}_k}$ centered at $\ox_k(\tilde{t}_k)$.
\end{lem}

We drop the subscript $k$ for a moment and consider a fixed $\M_k$ for $k$ large.
Let $\phi$ be a smooth non-increasing function on $\mathbb{R}$ such that: $\phi$ is $1$ on $(-\infty,1]$ and $0$ on $[2,\infty)$, and $\phi''\ge -10\phi$ and $(\phi')^2\le 10\phi$. 
Put $h(y)=\phi\left(\frac{d(y,\mathfrak{t}(y))+600\sqrt{\mathfrak{t}(y)}}{10A\epsilon}\right)$ on $\M_{t\le\epsilon^2}$. Then
\begin{equation}
   (\pt-\Delta)h=\frac{1}{10A\epsilon}\left(d_t-\Delta d+\frac{300}{\sqrt{t}}\right)\phi'-\frac{1}{(10A\epsilon)^2}\phi''.
\end{equation}

By \eqref{e: alpha-large} and Lemma \ref{l: distance laplacian}, we get
\begin{equation}\label{e: inequality for distance function}
d_t-\Delta d+\frac{300}{\sqrt{t}}\ge 0,
\end{equation}
for all points $y$, $t=\mathfrak{t}(y)$, such that $d(y,t)>\epsilon\ge\sqrt{t}$. In particular, if $\phi'\left(\frac{d(y,t)+600\sqrt{t}}{10A\epsilon}\right)\neq0$, then $9A\epsilon<d(y,t)<20A\epsilon$, and hence \eqref{e: inequality for distance function} holds at the point.
So we have
\begin{equation}\label{e: evolution of h}
    (\pt-\Delta)h\le \frac{-\phi''}{(10A\epsilon)^2}\le\frac{10\phi}{(10A\epsilon)^2}.
\end{equation}

First, for any open domain $\Omega\subset \M_t$ with smooth boundary, we can compute that
\begin{equation}\label{e: hu_1}
    \begin{split}
        \pt\int_{\Omega} hu\,d_tV&=\int_{\Omega} (\pt h-\Delta h)u\,d_tV+\int_{\partial\Omega}\left(-\frac{\partial h}{\partial\vec{n}}u+\frac{\partial u}{\partial\vec{n}}h\right)\,d_tS,
    \end{split}
\end{equation}
where $\vec{n}$ is the inwards unit normal vector field on $\partial\Omega$.
Applying Theorem \ref{e: uR^m less than C_m} as in Corollary \ref{l: integral equal to 1}, and using \eqref{e: evolution of h}, $h\le 1$, and $|\nabla h|\le\frac{\phi'}{10A\epsilon}$, we get
\begin{equation}\label{e: int hu ge}
    \left.\int hu\,d_tV\right|_{t=0}\ge\left.\int hu\,d_tV\right|_{t=\ot}-\frac{\ot}{(A\epsilon)^2}\ge 1-A^{-2}.
\end{equation}
Similarly, we can show
\begin{equation}\label{e: hv inequality}
\begin{split}
    \left.\int -hv\,d_tV\right|_{t=0}
    &\ge  \textnormal{exp}\left(\frac{-\tilde{t}}{10(A\epsilon)^2}\right)\left.\int -hv\,d_{t}V\right|_{t=\tilde{t}}
    \ge(1-A^{-2})\left.\int -hv\,d_{t}V\right|_{t=\tilde{t}}.
    \end{split}
\end{equation}
Also, replacing the function $h$ by $\overline{h}(y)=\phi\left(\frac{d(y,\mathfrak{t}(y))+600\sqrt{\mathfrak{t}(y)}}{5A\epsilon}\right)$, we can show for some constant $C>0$,
\begin{equation}\label{3}
    \int_{B_0(x_0,10A\epsilon)} u\,dV\ge1-CA^{-2}.
\end{equation}


By some distance distortion estimates using Lemma \ref{l: distance laplacian}, Lemma \ref{l: expanding lemma} and \eqref{e: alpha-large} we can establish the following inclusion
\begin{equation}\label{e: inclusion}
    B_{\tilde{t}}(\ox(\tilde{t}),\sqrt{\ot-\tilde{t}})\subset B_{\tilde{t}}(x_0(\tilde{t}),9A\epsilon).
\end{equation}
Since $h(\cdot,\tilde{t})=1$ on $B_{\tilde{t}}(x_0(\tilde{t}),9A\epsilon)$ and $v\le 0$, Lemma \ref{l: beta} implies $\left.\int -hv \,d_tV\right|_{t=\tilde{t}}\ge\beta$.
Hence by \eqref{e: hv inequality} we get
\begin{equation}\label{e: initial_nu}
    \left.\int-hv\,d_tV\right|_{t=0}\ge \beta(1-A^{-2}).
\end{equation}

Let $\tilde{u}(x)=h(x)u(x)$ for all $x\in\M_0$, and define  $\tilde{f}(x)$ by $\tilde{u}=(2\pi)^{-\frac{n}{2}}e^{-\tilde{f}}$. Then a direct computation shows
\begin{equation}\label{e: computation}
    \int_{\M_0} -hv\,d_0V=\int_{\M_0} (-\ot|\nabla\tilde{f}|^2-\tilde{f}+3)\tilde{u}\,d_0V + \int_{\M_0} \left(\ot\left(\frac{|\nabla h|^2}{h}-Rh\right)-h\log h\right)u\,d_0V.
\end{equation}

By \eqref{3} and $-h\log h\le 1$ when $h\le 1$, we have
\begin{equation}
\begin{split}
    \int_{\M_0}-uh\log h \,d_0V&=\int_{B_0(x_0,20A\epsilon)-B_0(x_0,10A\epsilon)}-uh\log h \,d_0V\le\int_{\M_0-B_0(x_0,10A\epsilon)}u \,d_0V\\
    &\le 1-\int_{B_0(x_0,10A\epsilon)}u \,d_0V\le CA^{-2}.
\end{split}
\end{equation}
Seeing also that $\frac{|\nabla h|^2}{h}\le\frac{10}{(10A\epsilon)^2}$, and $R\ge -1$ on $B_0(x_0,20A\epsilon)$, we can bound the second integral in \eqref{e: computation} above by $(1+C)A^{-2}+\epsilon^2$. This combining with \eqref{e: initial_nu} implies
\begin{equation}\label{4}
    \beta(1-A^{-2})\le\int_{\M_0}(-\ot|\nabla\tilde{f}|^2-\tilde{f}+3)\tilde{u}\,dV+(1+C)A^{-2}+\epsilon^{2}.
\end{equation}

Put $\widehat{g}=\frac{1}{2\ot}g(0)$, $\widehat{u}=(2\ot)^{\frac{n}{2}}\tilde{u}$, and define $\widehat{f}$ by $\widehat{u}=(2\pi)^{-\frac{n}{2}}e^{-\widehat{f}}$. Restoring the subscript $k$, then $\widehat{u}_k$ are supported in $B_0(x_{0k},2)$, and by \eqref{e: int hu ge} we get
\begin{equation}
    \lim_{k\rightarrow\infty}\int_{B_0(x_{0k},2)} \widehat{u}_k\,d\widehat{V}_k=1.
\end{equation}
Moreover, \eqref{4} implies the following for large $k$,
\begin{equation}
    \frac{1}{2}\beta\le\int_{B_0(x_{0k},2)}\left(-\frac{1}{2}|\nabla\widehat{f}_k|^2-\widehat{f}_k+3\right)\widehat{u}_k\,d\widehat{V}_k.
\end{equation}
This contradicts with the isoperimetric inequality in the assumption.

\end{proof}

\end{section}

\begin{section}{generalized singular Ricci flow}
\label{s: construction of non-compact SRF}
\subsection{Generalized singular Ricci flow: the definition and properties}\hfill\\
In this subsection, we give the definition and some properties of the generalized singular Ricci flow.

\begin{defn}
Let $(M,g)$ be a Riemannian manifold. For any $x\in M$, let
\begin{equation*}
    \overline{\rho}(x)=\sup\{r>0: B_g(x,r) \textnormal{ is relatively compact and }|\Rm|\le r^{-2} \textnormal{ on } B_g(x,r)\}.
\end{equation*}
Recall $\rho=R_+^{-1/2}$, it's clear that $\overline{\rho}\le \rho$.

By replacing the curvature scale $\rho$  by $\overline{\rho}$ in Definition \ref{d: complete}, we say a spacetime is weakly 0-complete (resp. forward or backward).
\end{defn}


\begin{defn}\label{d: generalized SRF}
A generalized singular Ricci flow is a Ricci flow spacetime $(\M,g(t))$, which satisfies: 
\begin{enumerate}
\item\label{p': complete} $\M_0=M$ is a complete orientable manifold.
\item\label{p': HI} $g(t)$ satisfies the Hamilton-Ivey pinching condition \eqref{e: Hamilton-Ivey} with $\varphi=\infty$.
\item\label{p': semi-backward complete} $\M$ is forward 0-complete, and weakly backward  0-complete.
\item\label{p': steps} For any $x_0\in \M$, there exist $N\in\mathbb{N}$ and a sequence of points $\{x_j\}_{j=0}^{N}$ with $\mathfrak{t}(x_j)=t_j$, such that  $t_0\ge t_1\ge\cdots\ge t_N=0$, $x_j$ survives until $t_{j+1}$, and $x_j(t_{j+1})),x_{j+1}$ are in the same connected component in $\M_{t_{j+1}}$.
\item\label{p': CNA} For any $x_0\in \M$ surviving on $[t_1,t_0]$ for some $t_1<t_0$, and any $A,\epsilon_{can}>0$, there is $r>0$ such that the $\epsilon_{can}$-canonical neighborhood assumption holds at scales less than $r$ on $B_{t}(x_0(t),A)$ for all $t\in[t_1,t_0]$.

\end{enumerate}
\end{defn}

\begin{defn}\label{d: semi-generalized SRF}
A semi-generalized singular Ricci flow is a Ricci flow spacetime $(\M,g(t),x_0)$ with $\mathfrak{t}(\M)=[0,t_0)$ for some $t_0\in[0,\infty]$, and $x_0\in\M_0$, which satisfies the following properties: 
\begin{enumerate}
\item\label{p: x_0,t_0,r_0} $x_0$ survives until $t$ for all $t\in[0,t_0)$, .
\item\label{p: HI} $g(t)$ satisfies the Hamilton-Ivey pinching condition \eqref{e: Hamilton-Ivey} with $\varphi=\infty$.
\item\label{p: semi-backward 0-complete} $\M$ is weakly backward 0-complete.
\item\label{p: connected} $\M_t$ is connected for each $t\in[0,t_0)$. 
\item\label{p: CNA and NC} 
For any $t_1\in[0,t_0)$, 
and any $A,\epsilon_{can}>0$ there is $r>0$ such that the $\epsilon_{can}$-canonical neighborhood assumption holds at scales less than $r$ on $B_{t}(x_0(t),A)$ for all $t\in[0,t_1]$.
\end{enumerate}
\end{defn}

\begin{remark}
Note that a singular Ricci flow satisfies \eqref{p': complete}\eqref{p': HI}\eqref{p': semi-backward complete}\eqref{p': steps}\eqref{p': CNA}, so it is a generalized singular Ricci flow.
Moreover, let $\M$ be a singular Ricci flow, $x_0\in\M_0$. Suppose $x_0$ survives on $[0,t_0)$ for some $t_0\in[0,\infty]$, and let $\M_{x_0}=\bigcup_{[0,t_0)}\bigcup_{A>0}B_t(x_0(t),A)$.
By the component stability \cite[Prop 5.17]{KL1}, the connected components are preserved when going backwards in time, it is clear that $\M_{x_0}$ is a semi-generalized singular Ricci flow.
\end{remark}

The following properties can be derived directly from the definition of the semi-generalized singular Ricci flow.

\begin{lem}\label{l: properties}
Let $(\M,g(t),x_0)$ be a semi-generalized singular Ricci flow on $[0,t_0)$. Let $t\in(0,t_0)$, then
\begin{enumerate}
    \item[\textnormal{(i)}]\label{p: scalar curvature is proper} For any $A>0$, the scalar curvature is proper on $\overline{B_t(x_0(t),A)}$.
    \item[\textnormal{(ii)}] \label{p: existence of uniform bw pb nbhd}  For any $A>0$, there exist $\overline{Q}, C>0$ such that for any $x\in B_t(x_0(t),A)$, letting $Q=\max\{\overline{Q},R(x)\}$, then
    $R\le CQ$ in $P(x,(CQ)^{-1/2},-(CQ)^{-1})$, which is contained in $\bigcup_{s\in[0,t]}B_s(x_0(s),2A)$.
\end{enumerate}
\end{lem}

\begin{proof}

For any $C>0$, consider the subset $K:=\overline{B_t(x_0(t),A)}\cap\{y\in\M: R(y)\le C\}$, equipped with the metric induced by the length metric on $\overline{B_t(x_0(t),A)}$. On the one hand, Lemma \ref{l: volume comparison} implies that $B_t(x_0(t),A)$ is totally bounded. So $K$ is totally bounded.
On the other hand, by the gradient estimate there exists $c>0$ such that for any $x\in K$, the ball $B_t(x,c)$ is unscathed and $R\le 2C$ in $B_t(x,c)$. From this it is easy see that $K$ is complete as a metric space. So $K$ is compact, which established (i).

For any $A>0$, by the gradient estimate, and the distance distortion estimate, and seeing that $\M$ is weakly backward 0-complete, we can find $\overline{Q},C>0$ such that $(C\overline{Q})^{-1}<t/2$, and the following holds: For any $x\in B_t(x_0(t),A)$,   $Q=\max\{\overline{Q},R(x)\}$, $x$ survives on $[t-(CQ)^{-1},t]$, and $R\le CQ$ in $B_{s}(x(s),(CQ)^{-1/2})$, which is contained in $B_{s}(x_0(s),2A)$.
By another distance distortion estimate this implies assertion (ii).

\end{proof}

The next proposition says that the connected components of a Ricci flow spacetime are preserved when going backwards in time, assuming the spacetime is weakly backward 0-complete and satisfies a distance-dependent canonical neighborhood assumption.
In particular, this component stability holds for generalized singular Ricci flows.

\begin{prop}(Component stability when going backwards in time)\label{l: component stability}
Let $\M$ be a Ricci flow spacetime, $x_0,x_1\in \M_{t_1}$ for some $t_1>0$. 
Suppose that $\M$ is weakly backward 0-complete. 
Suppose both $x_0,x_1$ survive until some $t_2<t_1$. 
Suppose also for any $A,\epsilon_{can}>0$ there exists $r>0$ such that the $\epsilon_{can}$-canonical neighborhood assumption holds in $B_t(x_0(t),A)$ at scales less than $r$ for all $t\in[t_2,t_1]$.

Suppose $x_0,x_1$ are in the same connected component of $\M_{t_1}$. Then $x_0(t_2),x_1(t_2)$ are in the same connected component of $\M_{t_2}$. 
\end{prop}

\begin{proof}
Without loss of generality we may assume that $x_0(t),x_1(t)$ are in the same connected component of $\M_t$ for all $t\in(t_2,t_1]$. 
Put
\begin{equation}
    \rho_0=\min\{\inf_{[t_2,t_1]}\overline{\rho}(x_0(t)),\inf_{[t_2,t_1]}\overline{\rho}(x_1(t)),1\}>0.
\end{equation}
So by the distance distortion estimate we can find $A>0$ such that $d_t(x_0(t),x_1(t))\le A$ for all $t\in(t_2,t_1]$. 
Fix a small $\delta>0$ and let $C_0(\delta)>0$ and $\epsilon_{can}(\delta)>0$ be from Lemma \ref{l: either neck or cap}. Choose some $\overline{r}\in(0,C_0^{-1}\rho_0)$ such that the $\epsilon_{can}$-canonical neighborhood assumption holds in $B_t(x_0(t),6A)$ at scales less than $4\overline{r}$ for all $t\in[t_2,t_1]$.
We may also assume $t_1-t_2<c$ for some $c(\overline{r},A)>0$ whose value will be determined in the course of the proof.

By Lemma \ref{l: existence of minimizing geodesic}, there exists a minimizing geodesic $\sigma:[0,1]\rightarrow \M_{t_1}$ between $x_0$ and $x_1$.
Choose a division of $[0,1]$ by $0=\alpha_0\le \alpha_1\le...\le\alpha_m=1$ such that one of the following two cases holds for each $i=0,1,..., m-1$:
\begin{enumerate}
    \item $\rho(x)\ge \overline{r}$ for all  $x\in\sigma_i:=\sigma([\alpha_i,\alpha_{i+1}])$;
    \item $\rho(x)\le 2\overline{r}$ for all $x\in\sigma([\alpha_i,\alpha_{i+1}])$, and $\rho(\sigma(\alpha_i))=\rho(\sigma(\alpha_{i+1}))=2\overline{r}$.
\end{enumerate}

Next, suppose by induction that for $\sigma_{i-1}$, $i\ge 1$, the following assumptions hold:
\begin{enumerate}
    \item[(a)] $\sigma_{i-1}$ survives backwards until $t_2$.
    \item[(b)]\label{e: stay within A}$\sigma_{i-1}(t)\subset  B_{t}(x_0(t),6A)$ for all $t\in[t_2,t_1]$.
\end{enumerate}
Suppose $\sigma_i$ satisfies case (1) and assume $c$ sufficiently small. Then by the gradient estimate, the distance distortion estimate, and the weakly backward 0-completeness of $\M$, we get that (a)(b) hold for $\sigma_i$. In particular, (a)(b) hold for $\sigma_0$.
So we can assume $\sigma_i$ satisfies case (2), $i\ge 1$. Let $t_3$ be the infimum of all times $t\in[t_2,t_1]$ such that $\sigma_i$ survives until $t$, and $\sigma_i(s)\subset B_s(x_0(s),6A)$ for all $s\in[t,t_1]$.

First, since $\overline{r}<C_0^{-1}\rho_0$ and $\sigma$ is a minimizing geodesic, it follows from Lemma \ref{l: either neck or cap} that $y_1:=\sigma(\alpha_i)$ and $y_2:=\sigma(\alpha_{i+1})$ are both centers of $\delta$-necks. 
Taking $c$ small, then by the gradient estimate we have for all $t\in(t_3,t_1]$ that $\rho(y_1(t)),\rho(y_2(t))\in[\overline{r},4\overline{r}]$, and $\rho(x(t))\le 4\overline{r}$ for all $x\in \sigma_i$.
Since the $\epsilon_{can}$-canonical neighborhood assumption holds in $B_t(x_0(t),6A)$ at scales less than $4\overline{r}$, by Lemma \ref{l: close to k-solution} that $y_1(t),y_2(t)$ are centers of $2\delta$-necks when $c$ is taken sufficiently small.
Moreover, by \cite[Proposition 19.21]{MT} we know that $\sigma_i(t)$ is contained in a $2\delta$-tube or a capped $2\delta$-tube.
Since $y_1(t),y_2(t)$ are the centers of $2\delta$-necks, it is easy to see that $\sigma_i(t)$ is contained in a $2\delta$-tube.
Then the evolution equation of scalar curvature implies $\partial_t R(x(t))>0$ for all $x\in \sigma_i$ and $t\in(t_3,t_1]$.
Therefore, $\sigma_i$ survives until $t_3$.

Next, for any $t\in[t_3,t_1]$, let $\mathcal{T}_t\subset\M_{t}$ be a $2\delta$-tube that contains the $100\overline{r}$-neighborhood of $\sigma_i(t)$, and let $\widehat{d}_t$ denote the length metric induced by $g(t)$ in $\mathcal{T}_t$.
Then for any $z_1, z_2\in \sigma_i(t)$,
$\widehat{d}_t(z_1,z_2)$ is realized by a smooth geodesic in $\mathcal{T}_t$. 
Let $\gamma_t$ be such a minimizing geodesic connecting $y_1(t)$ and $y_2(t)$. 
Then all the second variations along $\gamma_t$ are non-negative since it has the minimal length among all smooth curves in a neighborhood around it. 
So a distance distortion estimate as Lemma \ref{l: distance laplacian} shows
\begin{equation}\label{e: neck diameter shrinking}
    \partial_t \widehat{d}_t(y_1(t),y_2(t))\ge -C\overline{r}^{-1},
\end{equation}
for some universal constant $C>0$. 
Integrating this and taking $c$ sufficiently small, we get
\begin{equation}\label{e: prep}
    \widehat{d}_t(y_1(t),y_2(t))\le \widehat{d}_{t_1}(y_1,y_2)+ C\overline{r}^{-1}(t_1-t)\le 2A.
\end{equation}
Moreover, since $\rho(x)\le 4\overline{r}$ for all $x\in \sigma_i(t)$, by the triangle inequality we get
\begin{equation}
    \widehat{d}_t(y_1(t),x)\le \widehat{d}_t(y_1(t),y_2(t))+10\cdot 2\cdot 4\overline{r}\le 3A.
\end{equation}
By the distance distortion estimate we get $d_t(x_0(t),y_1(t))\le 2A$ and
\begin{equation}
    \begin{split}
        d_t(x,x_0)&\le d_t(x_0(t),y_1(t))+d_t(y_1(t),x)
        \le 2A+\widehat{d}_t(y_1(t),x)\le 5A.
    \end{split}
\end{equation} 
So $\sigma_i(t)\subset B_t(x_0(t),6A)$ for all $t\in[t_3,t_1]$.
By the infimum assumption of $t_3$, we get $t_3=t_2$. Hence (a)(b) hold for $\sigma_i$. 
So by induction the entire $\sigma$ survives backwards until $t_2$. It follows that $x_0(t_2),x_1(t_2)$ are in the same connected component of $\M_{t_2}$.
\end{proof}

The following corollary of Proposition \ref{l: component stability} gives the relation between a semi-generalized singular Ricci flow and a generalized singular Ricci flow.

\begin{cor}\label{r: semi}
A Ricci flow spacetime $(\M,g(t))$ is a generalized singular Ricci flow if and only if it satisfies property \eqref{p': complete}\eqref{p': HI}\eqref{p': semi-backward complete}\eqref{p': steps} in Definition \ref{d: generalized SRF}, and for any $x_0\in\M$, which survives on $[t_1,t_0)$, let $\M_{x_0}=\bigcup_{t\in[t_1,t_0)}\bigcup_{A>0}B_t(x_0(t),A)$, then $(\M_{x_0},x_0)$ is a semi-generalized singular Ricci flow.
\end{cor}

\begin{proof}
The `if' part is obvious by the definitions. For the `only if' part, we need to show $\M_{x_0}$ is a semi-generalized singular Ricci flow. It suffices to show $\M_{x_0}$ is weakly backward 0-complete.
Let $x\in\M_{x_0}$ be an arbitrary point. Suppose $\mathfrak{t}(x)=t_1$, $x$ survives on $(t_2,t_1]$ in $\M_{x_0}$ for some $t_2<t_1$, and $\inf_{t\in(t_2,t_1]}\overline{\rho}(x(t))>0$.
Since $\M$ is weakly backward 0-complete,  $x(t_2)=\lim_{t\searrow t_2}x(t)$ exists. By Proposition \ref{l: component stability}, $x(t_2)\in\M_{x_0}$. So $\M_{x_0}$ is weakly backward 0-complete.
\end{proof}



\subsection{Compactness and existence theorems}\hfill\\
First, we show a compactness theorem which gives a criterion for a sequence of singular Ricci flows to have a subsequence that converges to a semi-generalized singular Ricci flow.
Then we apply the compactness theorem to show the existence of generalized singular Ricci flows.
\begin{defn}(Partial convergence)\label{d: partial convergence}
We say a sequence of Ricci flow spacetimes $(\M_i,G_i)$ partially converges to a spacetime $(\M,G)$ if the following holds:
There is a sequence of diffeomorphisms $\phi_i:U_i\rightarrow V_i\subset\M_i$, where $U_i$ and $V_i$ are open domains in $\M$ and $\M_i$ respectively, such that given any compact subset $K\subset\M$, $k\in\mathbb{N}$ and $\epsilon>0$, we have $K\subset U_i$ for all large $i$, and $\|\phi_i^*G_i-G\|_{C^k(K,G)}\le\epsilon$.

Let $x_0\in\M$ and $x_{0i}\in\M_i$. We say the sequence of pointed spacetimes $(\M_i,x_{0i})$ partially converges to $(\M,x_0)$, if $x_0\in U_i$, $x_{0i}\in V_i$, and $\phi_i(x_{0})=x_{0i}$.
\end{defn}

\begin{theorem}[Compactness theorem]\label{t: compactness}
Let $\{(\M_i,x_{0i})\}_{i=1}^{\infty}$ be a sequence of singular Ricci flows, $x_{0i}\in\M_i$, $\mathfrak{t}(x_{0i})=t_{0i}$. Suppose 
\begin{enumerate}
    \item[(a)] \label{con: r_0} For some $r_0>0$, $x_{0i}$ survives until $t_{0i}+r_0^2$, and $|\Rm|\le r_0^{-2}$ on $P(x_{0i},r_0,r_0^2)$.
    \item[(b)] \label{con: CNA and NC} For any $\epsilon_{can}>0$ and $A>0$, there exist $r(A,\epsilon_{can})$ and $\kappa(A)$ such that the $\epsilon_{can}$-canonical neighborhood assumption and the $\kappa(A)$-non-collapsedness hold at scales less than $r(A,\epsilon_{can})$ in $B_t(x_{0i}(t),A)$ for all $t\in[t_{0i},t_{0i}+r_0^2]$.
\end{enumerate}

Then there exists a semi-generalized singular Ricci flow $(\M,x_0)$ on $[0,r_0^2)$, $x_0\in\M_0$, such that a subsequence of $(\M_{i,t\ge t_{0i}},x_{0i})$ partially converges to $(\M,x_0)$. 
\end{theorem}

\begin{proof}[Proof to Theorem \ref{t: compactness}]
We may assume $t_{0i}=0$ and $r_0\le1$ without loss of generality.
Let $d_{G_i}$ be the length metric on $\M_i$ induced by the spacetime meric $G_i$. 
For any $A>0$, restrict the metric $d_{G_i}$ on the subset
\begin{equation}
    \mathcal{P}_i(A)=\bigcup_{t\in[0,r_0^2)}B_t(x_{0i}(t),A)\cap\{x\in\M_i: |\Rm|(x)\le A^2\}.
\end{equation}
Then the diameter of every $(\mathcal{P}_i(A),d_{G_i})$ is bounded above by $2(A+r_0^2)$. Moreover, the following lemma shows that they are uniformly totally bounded.
\begin{lem}\label{l: N}
For any $A,\epsilon>0$, there exists $N=N(A,\epsilon)\in\mathbb{N}$ such that for all $i$, any $\epsilon$-separating subset in  $(\mathcal{P}_i(A),d_{G_i})$ has at most $N$ elements. 
\end{lem}

\begin{proof}
On the one hand, by a combination of assumption (b), the gradient estimate and the distance distortion estimate, 
we may assume that $\epsilon$ is sufficiently small (depending on $A$) so that the following holds: First, for any $x\in \mathcal{P}_i(A)$, $\mathfrak{t}(x)\ge\epsilon/8$, the backward parabolic neighborhood $P(x,\epsilon/8,-(\epsilon/8)^2)$ is unscathed and contained in $\mathcal{P}_i(2A)$.
Second, for any $x\in \mathcal{P}_i(A)$ with $\mathfrak{t}(x)\le\epsilon/8$, $x$ survives until $0$, and $x(0)\in\mathcal{P}_i(2A)$.
Furthermore, there exists $c(A,\epsilon)>0$ such that 
\begin{equation}\label{e: volume lower bound}
    vol_{G_i}\left(P\left(x,\epsilon/8,-\left(\epsilon/8\right)^2\right)\right)\ge c(A,\epsilon).
\end{equation}

On the other hand, assumption (b) allows us to apply Lemma \ref{l: volume comparison} on each time-slice $\M_{i,t}$, and deduce that $B_t(x_{0i}(t),2A)$ is uniformly totally bounded, and there is a constant $v(A)>0$ such that
\begin{equation}
   vol_{g_i(t)}B_{t}(x_{0i}(t),2A)\le v(A)
\end{equation}
for all $i$ and $t\in[0,r_0^2)$. 
Integrating this we get
\begin{equation}\label{e: volume upper bound}
    vol_{G_i}(\mathcal{P}_i(2A))\le\int_{0}^{r_0^2}vol_{g_i(t)}B_{t}(x_{0i}(t),2A)\,dt\le \,v(A).
\end{equation}

Now suppose $\{x_k\}_{k=1}^{N_i}$ is an $\epsilon$-separating subset of $(\mathcal{P}_i(A),d_{G_i})$, and $t_k=\mathfrak{t}(x_k)$.
Let $\{x_{k_j}\}_{j=1}^{J_i}$ be all $x_k$ with $t_k< \epsilon/8$, then each $x_{k_j}$ survives backwards until $0$ and $x_{k_j}(0)\in B_0(x_{0i},2A)$.
Since $d_{G_i}(x_{k_j},x_{k_l})>\epsilon$ for any $j\neq l$, by the triangle inequality, $\{x_{k_j}(0)\}_{j=1}^{J_i}$ is an $3\epsilon/4$-separating subset of $B_0(x_{0i},2A)$.
Since $B_0(x_{0i},2A)$ is uniformly totally bounded, there is $C(A,\epsilon)>0$ such that $J_i\le C$.
Therefore, in order to bound $N_i$ we may assume that $t_k\ge \epsilon/8$ for all $k$.

Then each $P(x_k,\epsilon/8,-\left(\epsilon/8\right)^2)$ is unscathed, and $d_{G_i}(x_k,y)\le \epsilon/8+(\epsilon/8)^2<\epsilon/4$ for all $y\in P(x_k,\epsilon/8,-(\epsilon/8)^2)$.
Since $d_{G_i}(x_k,x_j)>\epsilon$ for any $k\neq j$, by the triangle inequality, we see that $P(x_k,\epsilon/8,-(\epsilon/8)^2)$, $k=1,2,...,N_i$, are pairwise disjoint. 
Therefore, combining \eqref{e: volume lower bound} and \eqref{e: volume upper bound}, we conclude that there is $N(A,\epsilon)>0$ such that $N_i\le N(A,\epsilon)$ for all $i$.

\end{proof}

Now since $(\mathcal{P}_i(A),d_{G_i})$ have uniformly bounded diameter and are uniformly totally bounded for all $i$, by Gromov's compactness theorem \cite[Proposition 44]{petersen}, we may assume $(\mathcal{P}_i(A),d_{G_i},x_{0i})$ converges to a metric space $(X(A),d_A,x_0)$ in the pointed Gromov-Hausdorff sense.
Since for any $A_1\ge A_2$, $(\mathcal{P}_i(A_2),d_{G_i},x_{0i})$ isometrically embeds into $(\mathcal{P}_i(A_1),d_{G_i},x_{0i})$, we get $(X(A_2),d_{A_2},x_0)$ isometrically embeds into $(X(A_1),d_{A_1},x_0)$. 
Let $(X,d)=\bigcup_{A>0}(X(A),d_A)$, and $\mathcal{N}_i=\bigcup_{A>0}\mathcal{P}_i(A)\subset\M_i$, then $(\mathcal{N}_i,d_{G_i},x_{0i})$ converges to $(X,d,x_0)$ in the pointed Gromov-Hausdorff sense as $i\rightarrow\infty$.

Let $x\in X$, and suppose $x\in X(A)$ for some $A>0$. We say $x$ is a smooth point if there are a $\delta>0$, and a sequence of points $x_{i_{k}}\in\mathcal{P}_{i_k}(A)$ with $|\Rm|(x_{i_{k}})\le\delta^{-2}$ converging to $x$ (modulo the Gromov-Hausdorff approximations).
By the canonical neigborhood assumption in $\mathcal{P}_{i_k}(A)$ and the gradient estimate, we can find $\overline{\delta}=\overline{\delta}(\delta,A)$ such that $|\Rm|\le\overline{\delta}^{-2}$ in  $P(x_{i_k},\overline{\delta},-\overline{\delta}^2)$. 
Moreover, by the non-collapsing assumption in $\mathcal{P}_{i_k}(A)$, and the pseudolocality theorem, we get $|\Rm|\le\overline{\delta}^{-2}$ in  $P(x_{i_k},\overline{\delta},\overline{\delta}^2)$ with a possibly smaller $\overline{\delta}$.  Let $U(x_{i_k},\overline{\delta})=P(x_{i_k},\overline{\delta},-\overline{\delta}^2)\cup P(x_{i_k},\overline{\delta},\overline{\delta}^2)$, then by Shi's derivative estimate, the derivatives of $\Rm$ is also uniformly bounded in $U(x_{i_k},\overline{\delta})$.
So we obtain a smooth limit of $U(x_{i_k},\overline{\delta})$ in the Cheeger-Gromov sense by passing to a subsequence.
This defines a Ricci flow spacetime metric in a neighborhood of $x$ in $X$, which is isometric to that on $X$ by the uniqueness of the Gromov-Hausdorff limit.

Let $X_0\subset X$ be the set of all smooth points. 
Then we obtain a global Ricci flow spacetime metric on $X_0$, denoted by $G_{\infty}=dt^2+g(t)$.
In particular, $P(x_{0i},r_0,r_0^2)$ converges smoothly to $P(x_{0},r_0,r_0^2)\subset X_0$.
Moreover, by a standard gluing argument (see e.g. \cite[Theorem 72]{petersen}) we get a sequence of diffeomorphisms under which a subsequence of $\M_i$ partially converges to $X_0$, in the sense of Definition \ref{d: partial convergence}.

It implies that for any $\epsilon_{can},A>0$, there exists $r(A,\epsilon_{can})>0$ such that the $\epsilon_{can}$-canonical neighborhood assumption holds at scales less than $r$ in $X(A)\cap X_0$.
Furthermore, for any $A,C>0$, there exists $c=c(A,C)>0$ such that for all $x\in B_t(x_{0i}(t),A)\subset\mathcal{N}_i$ with $|\Rm|(x)\le C$, the backward parabolic neighborhood $P(x,c,-c^2)$ in $\mathcal{N}_i$ is unscathed.
So for all $x\in X(A)\cap X_0$ with $|\Rm|(x)\le C$, the region $P(x,c,-c^2)$ in $X_0$ is unscathed. 
From this it is easy to see that $X_0$ is weakly backward 0-complete.


Let $\M=\bigcup_{t\in[0,r_0^2)}\bigcup_{A>0}B_t(x_0(t),A)$ be a subset in $X_0$.
Then $\M$ is a smooth Ricci flow spacetime with connected time-slices, and a subsequence of $\M_i$ partially converge to $\M$.
It is clear that $\M$ satisfies property \eqref{p: x_0,t_0,r_0}\eqref{p: HI}\eqref{p: connected}\eqref{p: CNA and NC} in Definition \ref{d: semi-generalized SRF}. 
Moreover, applying Proposition \ref{l: component stability} to $X_0$, we see that $\M$ is also weakly backward 0-complete and hence satisfies property \eqref{p: semi-backward 0-complete}.
This proved Theorem \ref{t: compactness}. 
\end{proof}

The next lemma shows that the convergence of the initial manifolds implies the convergence of the singular Ricci flows to some semi-generalized singular Ricci flow.

\begin{lem}\label{t: construction}
Let $(\M_i,g_i(t))$ be a sequence of singular Ricci flows. Suppose for some $t_1\ge0$ and  $x_{0i}\in\M_{i,t_1}$, the sequence of time-slices $(\M_{i,t_1},g_i(t_1),x_{0i})$ smoothly converges to a 0-complete manifold $(M,g,x_0)$. Then there exists
a semi-generalized singular Ricci flow $(\M,g(t),x_0)$ on
$[0,t_0)$ for some $t_0>0$, such that $\M_0=M$ and a subsequence of $(\M_{i,t\ge t_1},x_{0i})$ partially converges to $(\M,x_0)$. Moreover, $t_0$ can be chosen such that $\inf_{[0,t_0)}\rho(x_0(t))=0$, i.e. $\limsup_{t\nearrow t_0}|\Rm|(x_0(t))=\infty$.
\end{lem}

\begin{proof}
Without loss of generality we may assume $t_1=0$.
On the one hand, by the pseudolocality theorem for singular Ricci flow, Theorem \ref{t: pseudolocality}, there exist $r_0,t_0>0$ such that for all large $i$ the domain $P(x_{0i},r_0,t_0)\subset\M_i$ is unscathed and $|\Rm|\le r_0^{-2}$ holds there.
Moreover, for any fixed $A>0$, there exists $t_A\in(0,t_0)$ such that the geometry $P(x_{0i},2A,t_A)\subset \M_i$ is uniformly bounded for all large $i$.
By a distance distortion estimate this implies the uniformly bounded geometry on $\bigcup_{t\in[0,t_A]}B_t(x_{0i}(t),A)\subset\M_i$ for a possibly smaller $t_A$.

On the other hand, for any $t\in(t_A,t_0)$, by Proposition \ref{p: noncollapsing} there are constants $r,\kappa>0$, such that the $\epsilon_{can}$-canonical neighborhood assumption and the $\kappa$-non-collapsing assumption hold at scales less than $r$ in $B_t(x_{0i}(t),A)$.
So by Theorem \ref{t: compactness} there is a subsequence of $(\M_i,x_{0i})$ which partially converges to $(\M,x_0)$.

If $\inf_{[0,t_0)}\rho(x_0(t))>0$, then by Lemma \ref{c: curvature blow up} we have $\inf_{[0,t_0)}\overline{\rho}(x_0(t))>0$. So there exist $\kappa',r'>0$ such that $\M$ is $\kappa'$-non-collapsed at $x_0(t)$ at scales less than $r'$ for all $t\in[0,t_0)$.
Repeating the above argument at $t$ sufficiently close to $t_0$, we can extend $\mathfrak{t}(\M)$ to $[0,t_1)$ with $t_1>t_0$.
So we may assume $\inf_{[0,t_0)}\rho(x_0(t))=0$.
\end{proof}

\begin{theorem}(Existence of a semi-generalized singular Ricci flow)\label{t: Z_2}
Let $(M,g)$ be a 3d orientable complete Riemannian manifold, $x_0\in M$.
Then there exists
a semi-generalized singular Ricci flow $(\M,g(t),x_0)$ on
$[0,t_0)$ for some $t_0>0$ with $\M_0=M$.

Moreover, if $M$ is the double cover of a non-orientable manifold, and $\sigma:M\rightarrow M$ is the corresponding deck transformation which acts as an isometry. Then there is a semi-generalized singular Ricci flow $\M$ with $\M_0=M$ such that $\sigma$ extends to an isometry on $\M$, which acts free on the open domain $\{x\in\M: x \textit{ survives until } t=0\}$.
\end{theorem}

\begin{proof}
The first assertion follows directly from Lemma \ref{t: construction}.
It only remains to establish the assertion about the $\mathbb{Z}_2$-symmetry. 
For this we assume $(M,g,x_0)$ is the double cover of a non-orientable manifold $(N,\overline{g})$, and $\sigma:M\rightarrow M$ is the non-trivial deck transformation in $\mathbb{Z}_2$, which acts as an isometry.
Let $N_i\subset N$ be a compact 3 dimensional submanifold with smooth boundary that contains $B_{\overline{g}}(\pi(x_0),i)$. 
Take $i>d_g(x_0,\sigma(x_0))$, then $\pi^{-1}(N_i)$ is a compact connected orientable manifold which has smooth orientable boundary $\pi^{-1}(\partial N_i)$, and $B_g(x_0,i)\cup B_g(\sigma(x_0),i)\subset \pi^{-1}(N_i)$.
First, we extend $\pi^{-1}(N_i)$ and the metric past a collar of its boundary, and assume the new metric $g_i$ is isometric to the product of a metric on $\pi^{-1}(\partial N_i)$ with an interval.
Next, since $\pi^{-1}(\partial N_i)$ is $\sigma$-invariant, we can extend the action of $\sigma$ to the collar neighborhood $\pi^{-1}(\partial N_i)\times[0,1]$ such that $\sigma(x,s)=\sigma(x,0)$ for all $x\in\pi^{-1}(\partial N_i)$ and $s\in[0,1]$.
Then by replacing $g_i$ with $\frac{1}{2}(g_i+\sigma^*g_i)$, we may assume $g_i$ is $\sigma$-invariant, and it is still a product metric near the new boundary. 
Therefore, by doubling the extended manifold, we get a closed, connected and orientable manifold $(M_i,g_i,x_0)$ with a deck transformation $\sigma_i$ which is an isometry, and $g_i=g$, $\sigma_i=\sigma$ on $B_{g_i}(x_0,i)$.

Let $(\M_i,g_i(t),x_0)$ be a sequence of singular Ricci flows starting from $(M_i,g_i,x_0)$. 
Then by Lemma \ref{t: construction} there is $t_0>0$ such that $\{(\M_i,g_i(t),x_0)\}$ converges to a semi-generalized singular Ricci flow $(\M,x_0)$ on $[0,t_0)$.
Moreover, by the uniqueness of singular Ricci flow in \cite{BK}, each $\sigma_i:M_i\rightarrow M_i$ can be uniquely extended to an isometry $\sigma_i:\M_i\rightarrow\M_i$.
So for any $x_1,x_2\in M_i$, if $\sigma_i(x_1)=x_2$ and $x_1$ survives until $t>0$, then $x_2$ also survives until $t$ and $\sigma_i(x_1(t))=\sigma_i(x_2(t))$.
Therefore, $\sigma_i$ converges to an isometry $\sigma:\M\rightarrow\M$, which acts free on
$\{x\in\M:x \textnormal{ survives back to } \M_{0}=M \}$
\end{proof}

The next lemma shows
that for two spacetimes $(\mathcal{N}_j,x_j)$, $j=1,2$, which have connected time-slices, suppose they are limits of a same sequence of Ricci flow spacetimes $\M_i$ under the partial convergence. Then they are isometric if the preimages of $x_1,x_2$ under the diffeomorphisms are contained in a parabolic region in $\M_i$.

\begin{lem}\label{r: uniqueness}
Let $\M_i$ be a sequence of Ricci flow spacetimes, $x_{1,i},x_{2,i}\in\M_{i,0}$. 
Suppose $(\M_i,x_{j,i})$ partially converges to a Ricci flow spacetime $(\mathcal{N}_j,x_j)$ on $[0,T)$, for some $T>0$, $x_j\in\mathcal{N}_{j,0}$, and each time-slice $\mathcal{N}_{j,t}$ is connected, $j=1,2$.
Suppose also there is $D>0$ such that  $x_{2,i}\in B_0(x_{1,i},D)$ for all $i$, and $P(x_{1},D,T-\delta)\subset \mathcal{N}_1$ is unscathed for any $\delta>0$.
Then $\mathcal{N}_{1}$ is isometric to $\mathcal{N}_{2}$.
\end{lem}

\begin{proof}
Let $H_i$ be the spacetime metric of $\M_i$, and $G_j$ the spacetime metric of $\mathcal{N}_j$, $j=1,2$. Let $\phi_{j,i}:\mathcal{N}_j\supset U_{j,i}\rightarrow V_{j,i}\subset\M_i$ be the two corresponding diffeomorphism sequences such that $\cup_{i=1}^{\infty}U_{j,i}=\mathcal{N}_j$ and $\|\phi_{j,i}^*H_i-G_j\|\le\epsilon_i\rightarrow0$.
Let $P_{j,k}=\bigcup_{t\in[0,T-k^{-1}]}\overline{B_t(x_j,k)}\cap\{x: \rho(x)\ge k^{-1}\}$,
then $\cup_{k=1}^{\infty} P_{j,k}=\mathcal{N}_{j}$.
By the assumption of $x_{1,i}$ and $x_{2,i}$, for a given $k$ there exists $\ell(k,D)\in\mathbb{N}$ such that for all large $i$, we have $\phi_{1,i}(P_{1,k})\subset\phi_{2,i}(P_{2,\ell})$.
So the maps $\phi_{2,i}^{-1}\circ\phi_{1,i}:\mathcal{N}_1\supset P_{1,k}\rightarrow\mathcal{N}_2$ are well-defined, and $\|(\phi_{2,i}^{-1}\circ\phi_{1,i})^*G_2-G_1\|\le\delta_i\rightarrow0$.
Moreover, $\phi_{2,i}^{-1}\circ\phi_{1,i}(P_{1,k})$ form an exhaustion of $\mathcal{N}_{2}$ as $i,k\rightarrow\infty$.
So $\mathcal{N}_{1}$ is isometric to $\mathcal{N}_{2}$.
\end{proof}

\begin{theorem}\label{t: max}(Theorem \ref{t: construction1} and \ref{t: convergence}, Existence of generalized singular Ricci flow)
Let $(M,g)$ be a 3d orientable complete Riemannian manifold, $x_0\in M$.
Let $(\M_i,g_i(t),x_{0i})$ be a sequence of singular Ricci flows with $(\M_{i,0},g_i(0),x_{0i})$ smoothly converging to $(M,g,x_0)$.
Then there exists a generalized singular Ricci flow $\M$ with $\M_0=M$, such that $(\M_i,x_{0i})$ partially converges to $(\M,x_0)$.

Moreover, if $M$ is the double cover of a non-orientable manifold, then the same conclusion as Theorem \ref{t: Z_2} holds.
\end{theorem}

\begin{proof}
Let $x_0\in M$, by Lemma \ref{t: construction} there exist $t_0>0$ and a semi-generalized singular Ricci flow $(\M^1,x_0)$ on $[0,t_0)$ such that $(\M_i,G_i,x_0)$ partially converges to $(\M^1,x_0)$ and $\inf_{[0,t_0)}\rho(x_0(t))=0$.

Suppose by induction that there is a 
sequence of Ricci flow spacetimes $\{\M^{j}\}_{j=1}^{k-1}$ such that $\M^{j-1}\subset\M^j$ and the followings hold for all $j=2,...,k-1$:
\begin{enumerate}
    \item A subsequence of $\M_i$ partially converges to $\M^{j}$.
    \item $\M^j$ is weakly backward 0-complete.
    \item For any $x\in\M^{j-1}$, let $a$ be the supremum of all times $t$ until which $x$ survives until in $\M^j$. Then $\inf_{[\mathfrak{t}(x),a)}\rho(x(t))=0$.
    \item For any $x\in\M^{j-1}$, suppose $x$ survives until some $t>\mathfrak{t}(x)$. Let $\M_x\subset\M^j$ be the subset $\bigcup_{s\in[\mathfrak{t}(x),t)}\bigcup_{A>0}B_s(x(s),A)$, then $(\M_x,x)$ is a semi-generalized singular Ricci flow on $[\mathfrak{t}(x),t)$.
\end{enumerate}

Let $\{x_j\}_{j=1}^{\infty}$ be a dense subset in $\M^{k-1}$.
For each $x_j$, by Lemma \ref{t: construction} there is a subsequence of $\{(\M_i,x_j)\}_{i=1}^{\infty}$ that partially converges to a semi-generalized singular Ricci flow $(\mathcal{N}_j,x_j)$, such that $x_j$ survives in $\mathcal{N}_j$ until $R(x_j(t))$ goes unbounded.
So by a diagonal argument we may assume that $\{(\M_i,x_j)\}_{i=1}^{\infty}$ converges to $(\mathcal{N}_j,x_j)$ for all $x_j$.

For any $y_1,y_2\in\M^{k-1}\sqcup\coprod_{j=1}^{\infty}\mathcal{N}_j$, we say $y_1\sim y_2$ if there is a sequence of points $w_i\in\M_i$ such that modulo the diffeomorphism maps we have $w_i\rightarrow y_1$ and $w_i\rightarrow y_2$ as $i\rightarrow\infty$.
This defines an equivalent relation in $\M^{k-1}\sqcup\coprod_{j=1}^{\infty}\mathcal{N}_j$.
If $y_1\sim y_2$, then by the uniqueness of the smooth limit, there is $\delta>0$ such that the neighborhoods of $P(y_i,\delta,\delta^2)\cup P(y_i,\delta,-\delta^2)$, $i=1,2$, are unscathed and the spacetime metrics on them are isometric.
So there is a well-defined smooth Ricci flow spacetime metric on the quotient space $\M^k:=\left(\M^{k-1}\sqcup\coprod_{j=1}^{\infty}\mathcal{N}_j\right)/\sim$.
So (1) holds for $j=k$.

Since each connected component of $\M^k_t$ is isometric to either $\M^{k-1}_t$ or some $\mathcal{N}_{j,t}$, we get that $\M^k_t$ is 0-complete. 
For any $x\in\M^k$, $\mathfrak{t}(x)=t_0$,
suppose $x$ survives on $(t_1,t_0]$ and $\lim_{t\rightarrow t_1}\overline{\rho}(x(t))>0$.
Assume $x\in \M^{k-1}$, then since $\M^{k-1}$ is weakly backward 0-complete, it follows that $x(t)\in\M^{k-1}$ and $\lim_{t\rightarrow t_1}x(t)$ exists.
Otherwise, assume $x\in\mathcal{N}_j$ for some $j\in\mathbb{N}$, and let $t_2\in(t_1,t_0]$ be the infimum of time $t$ such that $x(t)\in\mathcal{N}_j$.
Then $x(t_2)=\lim_{t\rightarrow t_2}x(t)$ exists because $\mathcal{N}_j$ is weakly backward 0-complete. 
If $t_2>t_1$, then we have $x(t_2)\in\mathcal{N}_j\cap\M^{k-1}$, and the existence of $\lim_{t\rightarrow t_1}x(t)$ exists by the weakly backward 0-completeness of $\M^{k-1}$.
So $\M^k$ is weakly backward 0-complete, and hence (2) holds.

It is clear that (3)(4) hold for each $x_j$.
We claim that (3)(4) hold for every point in $\M^{k-1}$.
To verify (3), let $x\in \M^{k-1}$ be an arbitrary point, $\mathfrak{t}(x)=t_1$.
Let $t_2>t_1$ be the supremum of all times until which $x$ survives in $\M^{k}$.
Suppose by contradiction that $\inf_{[t_1,t_2)}\rho(x(t))>0$. 
Then by Lemma \ref{t: construction} there is $\delta>0$ such that by passing to a subsequence, $(\M_i,x)$ partially converges to a semi-generalized singular Ricci flow $\mathcal{N}$ on $[t_1,t_2+\delta^2)$, and  $P(x,\delta,t_2-t_1+\delta^2)$ is unscathed. 
By the density of $\{x_j\}_{j=1}^{\infty}$, there exists $x_j\in P(x,\delta,\delta^2)\subset\M^{k-1}$.
Then $x_j$ survives on $[\mathfrak{t}(x_j),t_2+\delta^2)$ in $\mathcal{N}_j$.
So it follows from Lemma \ref{r: uniqueness} that $\mathcal{N}_j$ is isometric to $\mathcal{N}$ on $[t_1+\delta^2,t_2+\delta^2)$. 
In particular, $x\in\mathcal{N}_j\subset\M^k$ survives until $t_2+\delta^2/2$, contradicting with the supremum assumption of $t_2$.
This verifies (3).

To verify (4), let $x\in \M^{k-1}$ be an arbitrary point, $\mathfrak{t}(x)=t_1$, and assume $x$ survives until some $t_2>t_1$, and $\M_{x}$ is defined as in (4).
Choose $\delta>0$ such that $P(x,\delta,t_2-t_1)$ is unscathed, and pick some $x_j\in P(x,\delta,\delta^2)$ by the density of $\{x_j\}_{j=1}^{\infty}$. Then by Lemma \ref{r: uniqueness}, $\M_x$ is isometric to $\mathcal{N}_j$ on $[t_1+\delta^2,t_2)$, and hence $(\M_x,x)$ is a semi-generalized singular Ricci flow on $[t_1+\delta^2,t_2)$. 
Letting $\delta\rightarrow 0$, it implies that $\M_x$ is a semi-generalized singular Ricci flow on $[t_1,t_2)$.
This verifies (4).    

So by induction we obtain an infinite sequence of spacetimes $\{\M^k\}_{k=1}^{\infty}$ with $\M^{k-1}\subset\M^k$, which satisfies all inductive assumptions.
Let $\M=\bigcup_{k=1}^{\infty}\M^k$, then 
by passing to a subsequence $\M_i$ partially converges to $\M$.
By the `if' part of Corollary \ref{r: semi}, it is clear that $\M$ is a generalized singular Ricci flow. The assertion about the $\mathbb{Z}_2$-symmetry follows in the same way as Theorem \ref{t: Z_2}.

\end{proof}

\end{section}

\begin{section}{Ricci flows with non-negative Ricci curvature}
\label{s: it is actually smooth}
In this section, we prove Theorem \ref{t: existence}. First, by adapting the maximum principle argument in \cite{Chen} and \cite{localpinching} to a generalized singular Ricci flow, we show in Lemma \ref{l: R} and \ref{l: Ricci} that it preserves the non-negativity of scalar curvature and Ricci curvature.

Then we prove Lemma \ref{l: s_0}, which is the last ingredient needed to prove Theorem \ref{t: existence}. It says that in a 3-dimensional manifold with $\Ric\ge 0$, no singularity can form within finite distance along a minimizing geodesic covered by final time-slices of strong $\delta$-necks.



\begin{lem}\label{l: R}
Let $(M,g)$ be a 3 dimensional complete Riemannian manifold with $R\ge0$. Let $(\M,g(t))$ be a generalized singular Ricci flow starting from $(M,g)$. Then $R\ge0$ on $\M$.
\end{lem}

\begin{proof}
By property \eqref{p': steps} in Definition \ref{d: generalized SRF} and Corollary \ref{r: semi}, it suffices to prove the lemma for a semi-generalized singular Ricci flow $(\M,g(t),x_0)$ on $[0,t_0)$, $x_0\in\M_0=M$. 
We may assume that there is $r_0>0$ such that $\bigcup_{t\in[0,t_0)}B_t(x_0(t),r_0)$ is unscathed and $\Ric(x)\le r_0^{-2}$ there. Then by Lemma \ref{l: distance laplacian}, we have
\begin{equation}\label{e: distance distortion}
    (\partial_t-\Delta)d_t(x_0(t),x)\ge-\frac{10}{3}r_0^{-1},
\end{equation}
for all $x\in\M_t$ with $d_t(x,x_0(t))>r_0$.

Let $A\ge \frac{80}{3}r_0^{-2}t_0+2 $ and define the following function on $\M$
\begin{equation}
    u(x)=\varphi\left(\frac{d_t(x_0(t),x)+\frac{10}{3}r_0^{-1}t}{Ar_0}\right)R(x),
\end{equation}
for all $x\in\M_t$, $t\in[0, t_0)$, where we choose $\varphi$ to be a smooth non-negative non-increasing function such that $\varphi=1$ on $(-\infty,\frac{7}{8}]$, $\varphi=0$ on $[1,\infty)$ and $\left|\frac{2\varphi'^2}{\varphi}-\varphi''\right|\le C\sqrt{\varphi}$. 
Then with the choice of $A$, we have $u(x)=R(x)$ for all $x\in B_t(x_0(t),\frac{3}{4}Ar_0)$, and $u(x)=0$ for all $x\in\M_t\setminus B_t(x_0(t),Ar_0)$.

Let $u_{\min}(t):=\min\{\inf_{\M_t} u(\cdot),0\}$, $t\in[0,t_0)$. 
If $u_{\min}(t)<0$, we claim that $\inf_{\M_t} u(\cdot)$ can be achieved. Suppose not, then there exists a sequence of points $x_i\in B_t(x_0(t),Ar_0)$ such that $u(x_i)\rightarrow u_{\min}(t)$ as $i\rightarrow\infty$. 
By Lemma \ref{p: existence of uniform bw pb nbhd}, the properness of scalar curvature, we may assume that $R(x_i)\rightarrow\infty$. So $u(x_i)\ge 0$ for all large $i$, a contradiction.

Then we claim the following holds for all $t\in(0,t_0)$:
\begin{equation}\label{e: component stability}
  u_{\min}(t)\le\lim\inf_{s\searrow t} u_{\min}(s).
\end{equation}

Suppose this is not true at some $t\in(0,t_0)$.
Then there exist some $\epsilon>0$ and a sequence of times $s_i>t$ which converges to $t$ as $i\rightarrow\infty$ such that
\begin{equation}
    u_{\min}(t)>u_{\min}(s_i)+\epsilon,
\end{equation}
for all $i$. Let $x_i\in B_{s_i}(x_0(s_i),Ar_0)$ be a point such that
\begin{equation}\label{e: contradiction}
    u(x_i)\le u_{\min}(s_i)+\frac{\epsilon}{2}<u_{\min}(t)- \frac{\epsilon}{2}.
\end{equation}

If $\R(x_i)$ is not uniformly bounded, then $u(x_i)\ge 0$ for large $i$, which implies $u_{\min}(t)\ge\epsilon>0$, a contradiction. So we may assume $R(x_i)$ is uniformly bounded, and hence by Lemma \ref{p: existence of uniform bw pb nbhd} there is a $\delta>0$ such that $R\le\delta^{-2}$ in $P(x_i,\delta,-\delta^2)\subset\subset\bigcup_{t\in[0,t_0)}B_t(x_0(t),2Ar_0)$. So $u$ is uniformly continuous on $\bigcup_i P(x_i,\delta,-\delta^2)$. Since $s_i-t\rightarrow0$, this implies $u(x_i(t))\le u(x_i)+\frac{\epsilon}{2}$ for all large $i$. So
\begin{equation}
    u_{\min}(t)\le u(x_i(t))\le u(x_i)+\frac{\epsilon}{2},
\end{equation}
which contradicts with \eqref{e: contradiction}. So claim \eqref{e: component stability} is true.

Now we argue by maximum principle that the following holds for all times:
\begin{equation}\label{e: this}
    u_{\min}(t)\ge -\frac{2C_0}{(Ar_0)^2},
\end{equation}
where $C_0>0$ will be specified below. Suppose not and let $T$ be the supremum of all times $t$ such that \eqref{e: this} is true on $[0,t]$. Then $T>0$ and there exists a sequence $t_i>T$ converging to $T$ as $i\rightarrow\infty$ such that $u_{\min}(t_i)<-\frac{2C_0}{(Ar_0)^2}$. Using inequality \eqref{e: component stability} at $T$, we have that $u_{\min}(T)\le -\frac{2C_0}{(Ar_0)^2}<0$.

Since $u_{\min}(T)<0$, there exists $x_T\in\M_T$ such that $u_{\min}(T)=u(x_T)$. Then by the choice of $T$ it is easy to see the followings hold at $x_T$:
$\nabla u=0$, $\Delta u\ge 0$,  and $\pt u\le 0$. By a direct computation we get the following at $x_T$,
\begin{equation}\label{e: nabla_u=0}
    \begin{split}
    \quad \R\nabla\varphi+\varphi\nabla\scalar&=0,\\
    2\nabla\varphi\cdot\nabla\scalar&=-2\frac{|\nabla\varphi|^2}{\varphi}\scalar=-2\frac{\varphi'^2}{\varphi}\frac{1}{(Ar_0)^2}\scalar.
    \end{split}
\end{equation}
By the evolution equation $(\pt-\Delta)R=2|\Ric|^2$, we get
\begin{equation}
\begin{split}
    (\pt-\Delta)u=&\varphi'\R\frac{1}{Ar_0}[(\pt-\Delta)d_t(x_0(t),x)+\frac{10}{3}r_0^{-1}]\\
    &-\varphi''\frac{1}{(Ar_0)^2}\R+2\varphi|\Ric|^2-2\nabla\varphi\nabla\R,
\end{split}
\end{equation}
restricting which at $x_T$ and using \eqref{e: distance distortion}, \eqref{e: nabla_u=0} and $3|\Ric|^2\ge\R^2$, we obtain the following
\begin{equation}\label{e: u_{min}}
    \begin{split}
       0\ge (\pt-\Delta) u&\ge \frac{2}{3}\varphi\scalar^2 -\varphi''\frac{1}{(Ar_0)^2}\R+2\frac{\varphi'^2}{\varphi}\frac{1}{(Ar_0)^2}\R\\
        &\ge \frac{2}{3}\varphi\scalar^2-\frac{C}{(Ar_0)^2} \sqrt{\varphi}\R\\
        &\ge\frac{1}{3}(u_{\min}^2(T)-\frac{C_0^2}{(Ar_0)^4}),
    \end{split}
\end{equation}
where $C_0=\frac{3C}{2}$, and we have used $|\frac{2\varphi'^2}{\varphi}-\varphi''|\le C\sqrt{\varphi}$, and Cauchy inequality $\frac{C}{(Ar_0)^2}\sqrt{\varphi}\R\le\frac{1}{3}\varphi\R^2+\frac{C_0^2}{3(Ar_0)^4}$. Since $u_{\min}(T)<0$, \eqref{e: u_{min}} implies $u_{\min}(T)\ge\frac{-C_0}{(Ar_0)^2}$, a contradiction. So $u_{\min}(t)\ge\frac{-2C_0}{(Ar_0)^2}$ for all $t\in[0,t_0)$, and in particular it implies
\begin{equation}
    \R(x)\ge-\frac{2C}{(Ar_0)^2},
\end{equation}
for all $x\in B_t(x_0(t),\frac{3}{4}Ar_0)$, $t\in[0,t_0)$.
Letting $A$ go to infinity, we get
    $\R(x)\ge0$,
for all $x\in\M$.
\end{proof}

\begin{lem}\label{l: Ricci}
Let $(M,g)$ be a 3 dimensional complete Riemannian manifold with $\Ric\ge0$. Let $(\M,g(t))$ be a generalized singular Ricci flow starting from $(M,g)$. Then $\Ric\ge0$ on $\M$.
\end{lem}

\begin{proof}
For the same reason as in Lemma \ref{l: R}, it suffices to prove the lemma for a semi-generalized singular Ricci flow $(\M,g(t),x_0)$ on $[0,t_0)$, $x_0\in M$. We may assume that there is $r_0>0$ such that $\bigcup_{t\in[0,t_0)}B_t(x_0(t),r_0)$ is unscathed and $\Ric(x)\le r_0^{-2}$ there. 

Let $\lambda\ge\mu\ge\nu$ be the three eigenvalues of the curvature operator. Then it suffices to show that the following inequality holds on $\M$ for any $a>0$,
\begin{equation}\label{e: a}
    \R+a(\mu+\nu)\ge 0
\end{equation}
In fact, if this is true, then we have $\R+\epsilon^{-1}(\mu+\nu)\ge 0$ for any $\epsilon>0$. Multiplying both sides by $\epsilon$ and letting $\epsilon$ go to zero we get $\mu+\nu\ge 0$, i.e. $\Ric\ge 0$.

Now suppose by contradiction that \eqref{e: a} does not hold for all positive real numbers, then we can find $a,a'>0$ with $a<a'<a+\frac{1}{100}$ such that \eqref{e: a} holds for $a$ but not for $a'$.

By Lemma \ref{l: distance laplacian} we have
\begin{equation}\label{e: perelman's function}
    (\pt-\Delta)d_t(x_0(t),x)\ge-\frac{10}{3}r_0^{-1},
\end{equation}
whenever $d_t(x_0(t),x)>r_0$. Choose $\varphi:\mathbb{R}\rightarrow\mathbb{R}$ to be a smooth non-negative non-increasing function such that $\varphi=1$ on $(-\infty,\frac{7}{8}]$, $\varphi=0$ on $[1,\infty)$ and $\frac{2|\varphi'|^2}{\varphi}+|\varphi''|\le C_0$. 



Let $u:\M\rightarrow\mathbb{R}$ be defined by
\begin{equation}
    u(x)=\varphi\left(\frac{d_t(x_0(t),x)+\frac{10}{3}r_0^{-1}t}{Ar_0}\right)(\R+a'(\mu+\nu)),
\end{equation}
and $u_{\min}(t)=\min\{\inf_{\M_t} u(\cdot),0\}$. 

By the same reasoning as Lemma \ref{l: R} we can show the following inequality for all $t\in(0,t_0)$:
\begin{equation}\label{e: second component stability}
    u_{\min}(t)\le\lim\inf_{s\searrow t} u_{\min}(t).
\end{equation}

Let $T$ be the supremum of all $t$ such that $u(s)\ge -\frac{2C_0}{(Ar_0)^2}$ for all $s\in[0,t]$, where $C_0$ will be specified later. Then $T>0$ and by \eqref{e: second component stability} we get
\begin{equation}\label{e: u(T) less than}
u_{\min}(T)\le-\frac{2C_0}{(Ar_0)^2}.
\end{equation}

Since $u_{\min}(T)<0$, the minimum of $u$ is obtained at some point $x_T\in B_T(x_0(T),Ar_0)$. Let $\mathbb{V}_1,\mathbb{V}_2,\mathbb{V}_3$ be the orthonormal eigenvectors of $\Rm$ corresponding to eigenvalues $\lambda\ge\mu\ge\nu$ at the tangent space of $x_T$. We extend them smoothly to a neighborhood $\mathcal{P}$ around $x_T$ in the following way: first extend them to a neighborhood of $x_T$ in $\M_T$ by parallel translation along radial geodesic emanating from $x_T$ using $\nabla^{g(T)}$, and then extend them in time to make them constant in time in the sense that $\nabla_t\mathbb{V}_i=0$, $i=1,2,3$, where $\nabla_t$ is the natural space-time extension of $\nabla^{g(t)}$ such that it is compatible with the metric, i.e. $\pt\langle X,X \rangle_{g(t)}= 2\langle \nabla_tX,X \rangle_{g(t)}$. Then $\mathbb{V}_1,\mathbb{V}_2,\mathbb{V}_3$ is an orthonormal basis on $\mathcal{P}$, and $\Delta \mathbb{V}_i=0$ at $x_T$, $i=1,2,3$. 

Let
$\tilde{u}(x)=[\Rm(\mathbb{V}_1,\mathbb{V}_1)+\Rm(\mathbb{V}_2,\mathbb{V}_2)+\Rm(\mathbb{V}_3,\mathbb{V}_3)+a'(\Rm(\mathbb{V}_2,\mathbb{V}_2)+\Rm(\mathbb{V}_3,\mathbb{V}_3))]\cdot\varphi\left(\frac{d_t(x_0(t),x)+\frac{10}{3}r_0^{-1}t}{Ar_0}\right)$ for all $x\in\mathcal{P}$.
Then it is easy to see that $\tilde{u}(x)\ge u(x)$ in $\mathcal{P}$, and the equality is achieved at $x_T$. 

We can compute that
\begin{equation}\label{e: bb}
\begin{split}
    (\pt-\Delta)\tilde{u} &=-2\nabla\varphi\nabla(\frac{\tilde{u}}{\varphi})+\varphi\cdot(\pt-\Delta)[\Rm(\mathbb{V}_1,\mathbb{V}_1)+(a'+1)(\Rm(\mathbb{V}_2,\mathbb{V}_2)+\Rm(\mathbb{V}_3,\mathbb{V}_3))]\\
    &+[\Rm(\mathbb{V}_1,\mathbb{V}_1)+(a'+1)(\Rm(\mathbb{V}_2,\mathbb{V}_2)+\Rm(\mathbb{V}_3,\mathbb{V}_3))]\cdot(\pt-\Delta)\varphi\\
    &=-2\nabla\varphi\nabla(\frac{\tilde{u}}{\varphi})+\varphi\cdot \mathcal{I}+\mathcal{J}\cdot(\pt-\Delta)\varphi.
    \end{split}
\end{equation}

We estimate each term in \eqref{e: bb} at $x_T$. 
First, 
recall that $\Rm$ evolves by $ (\nabla_t-\Delta)\Rm=\Rm^2+\Rm^{\#}$ under Ricci flow, see \cite[Proposition 3.19]{MT}, where
\begin{equation}\label{e: evolution_2}
M^2+M^{\#}=
  \begin{bmatrix}
    \lambda^2+\mu\nu & 0 & 0  \\
    0 & \mu^2+\lambda\nu & 0 \\
    0 & 0 & \nu^2+\lambda\mu  
  \end{bmatrix},\quad \textnormal{for any matrix}\quad
M=
  \begin{bmatrix}
    \lambda & 0 & 0  \\
    0 & \mu & 0 \\
    0 & 0 & \nu  
  \end{bmatrix}.
\end{equation}
So by $\nabla\mathbb{V}_i=\Delta\mathbb{V}_i=\pt\mathbb{V}_i=0$ at $x_T$ we get
\begin{equation}
    (\pt-\Delta)(\Rm(\mathbb{V}_i,\mathbb{V}_i))=((\nabla_t-\Delta)\Rm)(\mathbb{V}_i,\mathbb{V}_i)=(\Rm^2+\Rm^{\#})(\mathbb{V}_i,\mathbb{V}_i)
\end{equation}
at $x_T$, $i=1,2,3$, and hence 
\begin{equation}\label{e: I}
\begin{split}
    \mathcal{I}(x_T)&=(\lambda^2+\mu\nu)+(a'+1)(\mu^2+\lambda\nu+\nu^2+\lambda\mu)\\
    &\ge\lambda[\lambda+(a+1)(\mu+\nu)]+(a'+1)(\mu^2+\nu^2)+(a'-a)\lambda(\mu+\nu)\\
    &\ge(a'+1)(\mu^2+\nu^2)+(a'-a)\lambda(\mu+\nu),
    \end{split}
\end{equation}
where we used $\lambda+(a+1)(\mu+\nu)\ge 0$.
Since $u(x_T)<0$, we have $\lambda<(a'+1)|\mu+\nu|$ at $x_T$ and hence
\begin{equation}
    \begin{split}
        |(a'-a)\lambda(\mu+\nu)|\le(a'-a)(a'+1)(\mu+\nu)^2
        \le\frac{a'+1}{100}(\mu+\nu)^2
        \le \frac{a'+1}{50}(\mu^2+\nu^2).
    \end{split}
\end{equation}
Substituting this into \eqref{e: I} and using $\lambda<(a'+1)|\mu+\nu|$ at $x_T$ again we get
\begin{equation}\label{e: I_1}
    \begin{split}
        \mathcal{I}(x_T)&\ge\frac{49}{50}(a'+1)(\mu^2+\nu^2)\ge\frac{49}{100}(a'+1)(\mu+\nu)^2\\
        &\ge\frac{49}{200(a'+1)}\{[(a'+1)(\mu+\nu)]^2+\lambda^2\}\\
        &\ge\frac{49}{400(a'+1)}[(a'+1)(\mu+\nu)+\lambda]^2\\
        &=\frac{49}{400(a'+1)\varphi^2}u_{\min}^2(T).
    \end{split}
\end{equation}
Then we estimate $\mathcal{J}\cdot(\pt-\Delta)\varphi$ at $x_T$ by using \eqref{e: perelman's function}, $\varphi<0$ and $u_{\min}(T)<0$ as below
\begin{equation}\label{e: J_1}
    \begin{split}
        (\mathcal{J}\cdot(\pt-\Delta)\varphi)(x_T)&=[\lambda+(a'+1)(\mu+\nu)](\pt-\Delta)\varphi\\
        &=\frac{1}{Ar_0}\left[\varphi'(\pt-\Delta)d_t(x_0(t),x)+\frac{10}{3}r_0^{-1}-\varphi''\frac{1}{Ar_0}\right]u_{\min}(T)\\
        &\ge \frac{1}{(Ar_0)^2\varphi}|\varphi''|u_{\min}(T).
    \end{split}
\end{equation}
Next, since $\tilde{u}$ obtains its minimum on $\mathcal{P}$ at $x_T$ and $\tilde{u}(x_T)=u(x_T)=u_{\min}(T)$, we get
\begin{equation}\label{e: nabla_1}
    \left(-2\nabla\varphi\nabla(\frac{\tilde{u}}{\varphi})\right)(x_T)=2\frac{|\nabla\varphi|^2}{\varphi^2}u_{\min}(T)=2\frac{|\varphi'|^2}{\varphi^2}\frac{1}{(Ar_0)^2}u_{\min}(T).
\end{equation}

Now applying the maximum principle at $x_T$ and using \eqref{e: I_1}, \eqref{e: J_1} and \eqref{e: nabla_1}, we get
\begin{equation}\label{e: maximum principle_2}
    \begin{split}
        0\ge (\pt-\Delta)\tilde{u}(x_T)&\ge\frac{49}{400(a'+1)\varphi}u_{\min}^2(T)+\frac{1}{(Ar_0)^2\varphi}|\varphi''|u_{\min}(T)+2\frac{|\varphi'|^2}{(Ar_0)^2\varphi^2}u_{\min}(T)\\
        &\ge \frac{49}{400(a'+1)\varphi}\left[u^2_{\min}(T)+\frac{1}{(Ar_0)^2}\left(\frac{2|\varphi'|^2}{\varphi}+|\varphi''|\right)u_{\min}(T)\right]\\
        &\ge \frac{49}{400(a'+1)\varphi}\left[u^2_{\min}(T)+\frac{C_0}{(Ar_0)^2}u_{\min}(T)\right],
    \end{split}
\end{equation}
where we have used $\frac{2|\varphi'|^2}{\varphi}+|\varphi''|\le C_0$.  Since $u_{\min}(T)<0$, \eqref{e: maximum principle_2} implies immediately $u_{\min}(T)\ge-\frac{C_0}{(Ar_0)^2}$, which contradicts with \eqref{e: u(T) less than}. 
So $u_{\min}(t)\ge-\frac{2C_0}{(Ar_0)^2}$ for all $t\in[0,t_0)$. Letting $A\rightarrow\infty$ we get $R+a'(\mu+\nu)\ge 0$ on $\M$, which
contradicts the assumption of $a'$. 
So \eqref{e: a} holds for all $a>0$, and hence by the argument at beginning the conclusion of the Lemma follows.

\end{proof}

The next lemma says that in a 3-dimensional manifold with $\Ric\ge 0$, no singularity can form within finite distance along a minimizing geodesic covered by final time-slices of strong $\delta$-necks.
We prove it by a contradiction argument, suppose the assertion does not hold, then
by the condition of $\Ric\ge 0$, we can show that the blow-up limit of the `singularity' is a smooth cone, and there is a Ricci flow whose final time-slice is in the smooth part of the cone, which is impossible.

\begin{lem}\label{l: s_0}
For any sufficiently small $\delta>0$ the following holds:
Let $(M,g)$ be a 3 dimensional Riemannian manifold with $\Ric\ge0$. 
Let $\gamma:[0,s_0)\rightarrow M$ (where $s_0\in \mathbb{R}_+\cup\{\infty\}$) be a unit speed minimizing geodesic such that $R(\gamma(s))$ does not stay bounded for $s\rightarrow s_0$, and assume there are constants $c,\varphi>0$ such that all points on $\gamma$ are centers of strong $\delta$-necks on the time interval $[-c,0]$, and the strong $\delta$-necks have $\varphi$-positive curvature.

Then $s_0=\infty$.
\end{lem}

\begin{proof}
Suppose by contradiction that $s_0<\infty$. Let $\eta$ be from Lemma \ref{l: derivative}.

Since every point on $\gamma$ is the center of some strong $\delta$-neck, we get that $\gamma$ lies inside some open subset $N\subset M$ that is diffeomorphic to $S^2\times(0,1)$ and which is covered by final time-slices of strong $\delta$-necks. Consider the length metric induced by the Riemannian metric on $N$, and then $N'$ be the completion of $N$. Then $N'$ is a disjoint union of $N$ and a single point $p$.


Consider the rescalings $iN'$ for all $i\in\mathbb{N}$.
Then by the Bishop-Gromov volume comparison we can deduce that for any $d>0$, the $d$-balls $B^{i N'}(p,d)$ in $i N'$ are uniformly totally bounded. 
Therefore, by Gromov's compactness theorem, we have the following Gromov-Hausdorff convergence by passing to a subsequence $\{i_k N'\}$:
\begin{equation}
    (i_k N',p)\xrightarrow{k\rightarrow\infty} (X,p_{\infty})
\end{equation}
We shall show that $X$ is a smooth metric cone with cone point $p_{\infty}$, and the convergence is actually smooth on $X_0=X-\{p_{\infty}\}$.

Let $x\in B(p,s_0)-\{p\}\subset N'$, then by Lemma \ref{l: derivative} we get
\begin{equation}\label{e: lower bound}
    R^{-1/2}(x)\le \eta\, d(p,x)\qquad\textnormal{on}\qquad B(p,s_0)-\{p\}.
\end{equation}

We claim that there exists $C>0$ such that
\begin{equation}\label{e: upper bound}
    R^{-1/2}(x)\ge C^{-1}d(p,x)\qquad\textnormal{on}\qquad B(p,s_0)-\{p\}.
\end{equation}
Suppose not, then there exists a sequence $\{x_k\}\subset B(p,s_0)-\{p\}$ such that
\begin{equation}\label{e: C_k}
    R^{-1/2}(x_k)\le C_k^{-1}d(p,x_k),
\end{equation}
where $C_k\rightarrow\infty$ as $k\rightarrow\infty$. We abbreviate $d(p,x_k)$ as $d_k$ and $R(x_k)$ as $R_k$. 

Since $x_k$ is the center of a  $\delta$-neck, there is a diffeomorphism onto its image $\phi_k: (-\delta^{-1},\delta^{-1})\times S^2\rightarrow N'$ under which $(N',x_k)$ is $\delta$-close to $(-\delta^{-1},\delta^{-1})\times S^2$ at scale $R^{-1/2}_k$. 
Let $U_k=\phi_k((-100,100)\times S^2)$, then $U_k$ separates $N'$ into two components.

Suppose $x\in B(p,s_0)-\{p\}$ is not in $U_k$, then it is easy to see either
\begin{equation}
    d(x,p)>d_k+10R^{-1/2}_k,\qquad\textnormal{or}\qquad d(x,p)<d_k-10R^{-1/2}_k.
\end{equation}
In other words, we have
\begin{equation}\label{e: inclusion in neck}
    B(p,d_k+10R^{-1/2}_k)-B(p,d_k-10R^{-1/2}_k)\subset U_k.
\end{equation}

Applying the Bishop-Gromov volume comparison on $N'$, we have $r^{-2}vol(\partial B(p,r))$ is non-increasing for all $r\in(0,s_0)$. In particular,
let $v_0=s_0^{-2}vol(\partial B(p,s_0))$, then $r^{-2}vol(\partial B(p,r))\ge v_0$ for all $0<r<s_0$. So by \eqref{e: inclusion in neck} we can estimate the volume of $U_k$ from below:
\begin{equation}\label{e: vol_1}
    \begin{split}
        vol(U_k)&\ge\int^{d_k+10R^{-1/2}_k}_{d_k-10R^{-1/2}_k} vol(\partial B(p,r))\,dr
        \ge\int^{d_k+10R^{-1/2}_k}_{d_k-10R^{-1/2}_k} v_0r^2\,dr\\
        &\ge\frac{9}{16}\int^{d_k+10R^{-1/2}_k}_{d_k-10R^{-1/2}_k} v_0d_k^2\,dr
        =\frac{45}{4}v_0d_k^2R^{-1/2}_k,
    \end{split}
\end{equation}
where in the third inequality we used \eqref{e: C_k}, which implies $d_k-R^{-1/2}_k\ge\frac{3}{4}d_k$ for large $k$.

By the closeness of the metric on $U_k$ with the standard cylindrical metric at scale $R^{-1/2}_k$, we get an upper bound on the volume of $U_k$:
\begin{equation}\label{e: vol_2}
    \begin{split}
        vol(U_k)\le 2\cdot R^{-3/2}_k\cdot200\cdot8\pi=3200\pi R^{-3/2}_k<3200\pi C_k^{-2}d_k^2R^{-1/2}_k,
    \end{split}
\end{equation}
where we used \eqref{e: C_k} in the last inequality. 
Combining \eqref{e: vol_1} with \eqref{e: vol_2} we get $C_k^2\le \frac{12800\pi}{45 v_0}$, which is impossible for large $k$. Thus there exists $C>0$ such that \eqref{e: upper bound} holds.

Therefore, by \eqref{e: upper bound} we see that the convergence on $X_0$ is smooth.
So there is $v_1>0$ such that $d^{-2}vol(\partial B^X(p_{\infty},d))= v_1$ for all $d\in(0,\infty)$.
So by the rigidity of volume comparison, we see that any Jacobi field along any geodesic emanating from $p_{\infty}\in X$ has linear growth, which implies that $X$ is a smooth metric cone. By \eqref{e: lower bound}, $X_0$ is nowhere flat. 

Since all points in a neighborhood of $p$ are centers of strong $\delta$-necks on $[-c,0]$, which has $\varphi$-positive curvature. So under the blow-up rescalings this implies that any point $x\in X_0$ is the center of a strong $2\delta$-neck on $[-\frac{1}{2}c,0]$, which has non-negative sectional curvature.
This contradicts the fact that open pieces in non-flat cones cannot arise as the result of Ricci flow with non-negative curvature \cite[Prop 4.22]{MT}.

\end{proof}

\begin{theorem}(Theorem \ref{t: existence})
Let $(M,g)$ be a 3d complete Riemannian manifold with $\Ric\ge 0$. There exist $T>0$ and a smooth Ricci flow $(M,g(t))$ with $g(0)=g$ defined on $[0,T)$. 
Moreover, if $T<\infty$, then $\limsup_{t\nearrow T}|\Rm|(x,t)=\infty$ for all $x\in M$.
\end{theorem}

\begin{proof}
First we assume $M$ is orientable. By Theorem \ref{t: max} there is a generalized singular Ricci flow  $(\M,g(t))$ starting from $(M,g)$, and by Lemma \ref{l: Ricci}, $\M$ has non-negative Ricci curvature. 

Let $x_0\in M$. Suppose $x_0$ survives until $t_0>0$ in $\M$.
We claim that the component of $\M_t$ that contains $x_0(t)$ is complete for all $t\in(0,t_0]$. Suppose not, then $\sup_{B_{t}(x_0(t),A)}R=\infty$ for some $A>0$ and $t\in(0,t_0]$.
By Lemma \ref{l: abundance of neck points} and \ref{l: close to k-solution}, we can find a minimizing geodesic
$\gamma:[0,1)\rightarrow\M_t$ such that $\lim_{s\rightarrow1}R(\gamma(s))=\infty$, and there exist $c,\varphi>0$ such that for all $s$ close to $1$, $\gamma(s)$ are centers of strong $\delta$-necks on $[-c,0]$, which have $\varphi$-positive curvature.
This contradicts Lemma \ref{l: s_0}.

Since $\Ric\ge0$, for any $A>0$, the parabolic neighborhood $P(x_0,A,t_0)$ is contained in $\bigcup_{t\in [0,t_0]}B_t(x_0(t),A)$, which is relatively compact. 
So every point in $M$ survives until $t_0$.
Let $T\in(0,\infty]$ be the supremum of all times until which $x_0$ survives. Then $T$ is also the supremum of the survival times of points in $M$. Suppose $T<\infty$, since $\M$ is forward 0-complete, we have $\limsup_{t\nearrow T}|\Rm|(x(t))=\infty$ for all $x\in M$.
So the spacetime restricted on the subset $\bigcup_{t\in[0,T)}M(t)$ is the desired smooth Ricci flow.

Now suppose $M$ is not orientable.
Let $\widehat{M}\rightarrow M$ be the 2-fold orientation covering.
By Theorem \ref{t: max}, there are a generalized singular Ricci flow $(\widehat{\M},g(t))$ starting from $\widehat{M}$, and an isometry $\sigma:\widehat{\M}\rightarrow\widehat{\M}$ that acts free on the subset of points that can survive back to $\widehat{M}$.
As before, there exists $T\in(0,\infty]$ such that $\widehat{M}$ survives on $[0,T)$, and $\limsup_{t\nearrow T}|\Rm|(x(t))=\infty$ for all $x\in \widehat{M}$ if $T<\infty$.
The smooth Ricci flow claimed in the theorem is the quotient of $\bigcup_{t\in[0,T)} \widehat{M}(t)$ by the free action of $\sigma$.

\end{proof}


\end{section}

\bibliography{Collection}
\bibliographystyle{alpha}

\end{document}